\documentclass[a4paper]{article}

\usepackage{amsfonts, amsmath, amssymb, amsthm}

\usepackage{graphics}
\usepackage{psfrag}

\theoremstyle{plain}
\newtheorem{letteredthm}{Theorem}
\newtheorem{letteredcor}[letteredthm]{Corollary}
\newtheorem{thm}{Theorem}[section]
\newtheorem{prop}[thm]{Proposition}
\newtheorem{lem}[thm]{Lemma}
\newtheorem{conj}[thm]{Conjecture}
\newtheorem{cor}[thm]{Corollary}
\newtheorem{quest}{Question}
\newtheorem*{clm}{Claim}
\theoremstyle{definition}
\newtheorem{defn}[thm]{Definition}
\newtheorem{rem}[thm]{Remark}

\renewcommand{\Vert}{\operatorname{Vert}}
\DeclareMathOperator{\Edge}{Edge}

\renewcommand{\labelenumi}{(\arabic{enumi})}

\newcommand{\mc}[1]{\mathcal{#1}}

\newcommand{\omc}[1]{\overline{\mathcal{#1}}}
\newcommand{\bast}{\displaystyle \ast}
\newcommand{\fm}[1]{\mc{#1}^{\pm\bast}}
\newcommand{\Z}{\mathbb{Z}}

\newcommand{\N}{\mathbb{N}}
\newcommand{\Q}{\mathbb{Q}}
\newcommand{\ol}{\overline}

\newcommand{\FreeEq}{\stackrel{\text{fr}}{=}}
\newcommand{\hast}{\mbox{\huge $\ast$ \normalsize}\!\!\!}

\DeclareMathOperator{\Dep}{Dep} \DeclareMathOperator{\Area}{Area}
\DeclareMathOperator{\Rad}{Rad}

\DeclareMathOperator{\height}{height}
\DeclareMathOperator{\Fr}{Fr}
\DeclareMathOperator{\dist}{d}
\DeclareMathOperator{\RArea}{RArea}
\DeclareMathOperator{\Back}{Back}
\DeclareMathOperator{\DEdge}{DEdge}
\DeclareMathOperator{\Aut}{Aut}
\DeclareMathOperator{\Ab}{Ab}
\DeclareMathOperator{\Int}{Int}
\DeclareMathOperator{\SDP}{SDP}
\DeclareMathOperator{\Depth}{Depth}

\begin{document}

\begin{center}
  \mbox{}\\
  \vspace{1in}
  \huge\bf{Isoperimetric functions for subdirect products and Bestvina-Brady groups}\\
  \vspace{2in}
  \LARGE\bf{William John Dison}
  \vspace{1in}\\
  Imperial College London\\
  \vspace{1in}
  \normalsize \emph{Thesis presented for the degree of\\
   Doctor of Philosophy, 2008}
\end{center}

\newpage

\begin{abstract}
  In this thesis we investigate the Dehn functions of two different classes of groups: subdirect products, in particular subdirect products of limit groups; and Bestvina-Brady groups.

  Let $D = \Gamma_1 \times \ldots \times \Gamma_n$ be a direct product of $n \geq 3$ finitely presented groups and let $H$ be a subgroup of $D$.  Suppose that each $\Gamma_i$ contains a finite index subgroup $\Gamma_i' \leq \Gamma_i$ such that the commutator subgroup $[D', D']$ of $D' = \Gamma_1' \times \ldots \times \Gamma_n'$ is contained in $H$.  Suppose furthermore that, for each $i$, the subgroup $\Gamma_i H$ has finite index in $D$.  We prove that $H$ is finitely presented and satisfies an isoperimetric inequality given in terms of area-radius pairs for the $\Gamma_i$ and the dimension of $(D'/H) \otimes \Q$.  In the case that each $\Gamma_i$ admits a polynomial-polynomial area-radius pair, it will follow that $H$ satisfies a polynomial isoperimetric inequality.

  As a corollary we obtain that if $K$ is a subgroup of a direct product of $n$ limit groups and if $K$ is of type $\textrm{FP}_{m}(\Q)$, where $m = \max \{ 2, n-1\}$, then $K$ is finitely presented and satisfies a polynomial isoperimetric inequality.  In particular, we obtain that all finitely presented subgroups of a direct product of at most $3$ limit groups satisfy a polynomial isoperimetric inequality.

  We also prove that if $B$ is a finitely presented Bestvina-Brady group, then $B$ admits a quartic isoperimetric function.
\end{abstract}

\newpage

\tableofcontents

\newpage

\begin{center}
  {\large\bf{Acknowledgements}}
\end{center}

\noindent \emph{I owe an immense debt of gratitude to my supervisor, Martin Bridson, without whom this thesis could never have been written.  It has been a privilege to learn mathematics from him and I am hugely grateful for all his invaluable advice.  I would also like to thank Tim Riley, Michael Tweedale and Henry Wilton for all the help and support they have given me, and all my friends and family for being so wonderful.}

\newpage

\section{Introduction}

Since its articulation by Dehn in the early $20^\text{th}$ century, the word problem has been one of the guiding problems in combinatorial and geometric group theory.  Given some finite group presentation, it asks whether there is an algorithm which will effectively determine whether any given word is trivial in the group.  Once it has been determined that a particular group, or class of groups, in which one is interested has a solvable word problem, then it is natural to inquire into the complexity of such an algorithmic solution. In this thesis we study a particular measure of the complexity of the word problem of a group, known as the \emph{Dehn function}.

We give a formal definition in Section~\ref{sec12} below, but, roughly, the Dehn function of a finitely presented group is the least upper bound on the number of defining relations which must be applied to demonstrate that a word in the generators is trivial in the group, with the bound being given in terms of the length of the word.  An \emph{isoperimetric function} for a group is an upper bound on the Dehn function.  In this thesis we will frequently be concerned with whether a group admits a polynomial isoperimetric function.  If one is interested in a class of groups, one might refine this criterion by asking for a single (uniform) polynomial which is an isoperimetric function for all the groups in the class.  Some justification for the choice of this dichotomy is provided by a result of Birget, Rips and Sapir \cite{Birget1}, who proved that the word problem of a finitely generated group $G$ is an $NP$-problem if and only if $G$ embeds in a finitely presented group which admits a polynomial isoperimetric function.

Thus far we have discussed Dehn functions in the language of combinatorial group theory.  The following geometric interpretation provides further justification for their study.  Given a Riemannian manifold $M$, Plateau's problem asks whether every simple null-homotopic loop in $M$ spans a least-area filling disc.  Under mild hypotheses Plateau's problem can be shown to have a positive solution \cite{Rado1}, \cite{Douglas1}, \cite{Morrey1}, and in this case one can define the \emph{filling function} of $M$.  This is the least function which bounds the area of least-area filling discs of rectifiable null-homotopic loops, with the bound being given in terms of the length of the loop.  Gromov's \emph{Filling Theorem} asserts that if $M$ is closed, then its filling function is essentially the same as the Dehn function of $\pi_1 M$.

We now introduce a method for constructing interesting classes of groups that will form the principle objects of study for much of this thesis.  Given a class of groups $\mc{C}$, the collection of \emph{subdirect products} of $\mc{C}$ is defined to be $$\SDP(\mc{C}) = \{ S \leq C_1 \times \ldots \times C_n : C_i \in \mc{C} \text{ and $S$ projects onto each $C_i$} \}.$$  In many cases the requirement that the subgroup projects onto each factor will be immaterial since one can replace the direct product $C_1 \times \ldots \times C_n$ by $p_1(S) \times \ldots \times p_n(S)$, where $p_i : S \rightarrow C_i$ is the projection homomorphism.

Recently, subdirect products have been recognised as worthy objects of study in their own right (see, for example, \cite{Bridson2}).  Typically, one chooses an input class $\mc{C}$ which is already well understood, and asks what can be said about $\SDP(\mc{C})$.  What is surprising, and fascinating, about this construction is that it only involves two absolutely basic group theoretic operations (taking direct products and passing to subgroups), and yet even when the input class is well understood, the same is not necessarily true of the output class.  For example, suppose one takes as input the class $\mc{F}$ of free groups: despite this being perhaps the most basic class of infinite groups, a whole raft of results indicate that the groups in $\SDP(\mc{F})$ are surprisingly diverse.  Stallings \cite{stal63} constructed a subgroup of $F \times F \times F$, where $F$ is a rank-2 free group, as the first example of a finitely presented group whose third integral homology group is not finitely generated.  Bieri \cite{Bieri1} showed that Stallings' group is one element of a sequence of groups $\textrm{SB}_n \leq F^n$, with $\textrm{SB}_n$ being of type $\textrm{F}_{n-1}$ but not of type $\textrm{FP}_n$.  Baumslag and Roseblade \cite{baum84} proved that there exist uncountably many finitely generated non-isomorphic subgroups of $F \times F$, and Miha{\u\i}lova \cite{miha58} and Miller \cite{mill71} exhibited examples with unsolvable conjugacy problems and unsolvable membership problems.  In \cite{Bridson2} Bridson and Miller proved that there exists a recursive sequence of finitely generated subgroups $G_i \leq F \times F$ such that there is no algorithm to determine the rank of $H_1(G_i, \Z)$, nor to decide whether it has any non-trivial torsion elements.

Hopefully, these examples will have convinced the reader of the inherent wildness of $\SDP(\mc{F})$.  From our point of view, it is then natural to ask whether this wildness manifests itself in the Dehn functions of these groups.

\begin{quest} \label{q1}
  Does every finitely presented group in $\SDP(\mc{F})$ admit a polynomial isoperimetric function?  Does there exist a uniform polynomial isoperimetric function for the whole class?
\end{quest}

Various authors have obtained results that bear on this question.  Gersten \cite{gers95} proved that, for $n \geq 3$, the Stallings-Bieri group $\textrm{SB}_n$ admits a polynomial isoperimetric function.  Elder, Riley, Young and the present author have proved \cite{Dison07} that the Dehn function of Stallings' group $\textrm{SB}_3$ is actually quadratic.  It follows from a theorem of Baumslag and Roseblade (see below) that all of the finitely presented subgroups of a direct product of at most $2$ free groups have either linear or quadratic Dehn functions.  By a result of Bridson, Howie, Miller and Short (Theorem~\ref{thm16} below), the same is true of a subgroup of a direct product of $n$ free groups which satisfies the finiteness condition $\textrm{FP}_n$.  We also note that there are various other lines for investigation naturally related to Question~\ref{q1}.  For example, can one find `nice' presentations for particular groups in $\SDP(\mc{F})$?  Do there exist finitely presented groups in $\SDP(\mc{F})$ whose Dehn functions are actually different from that of the ambient direct product?

Recent results suggest that the wildness encountered amongst the arbitrary finitely generated groups in $\SDP(\mc{F})$ is a manifestation of their failure to possess a strong enough degree of finiteness.  Baumslag and Roseblade \cite{baum84} showed that the only finitely presented subgroups of a direct product of $2$ free groups are the `obvious' ones, \textit{i.e.} those which are themselves virtually a direct product of at most $2$ free groups.  The following result of Bridson, Howie, Miller and Short extends the Baumslag-Roseblade theorem to an arbitrary number of factors.

\begin{thm}[\cite{brid02B}] \label{thm16}
  Let $F_1, \ldots, F_n$ be free groups.  A subgroup $G \leq F_1 \times \ldots \times F_n$ is of type $\textrm{FP}_n$ if and only if it has a subgroup of finite index which is itself a direct product of (at most $n$) free groups.
\end{thm}

Even if a subdirect product does not enjoy any finiteness properties stronger than being finitely presented, one still has the following structural result of Bridson and Miller.  Recall that the lower cental series $(\gamma_i(G))_{i=1}^\infty$ of a group $G$ is defined recursively by $\gamma_1(G)= G$ and $\gamma_i(G) = [\gamma_{i-1}(G), G]$.

\begin{thm}[\cite{Bridson2}] \label{thm17}
  Let $F_1, \ldots, F_n$ be free groups.  If a subdirect product $G \leq F_1 \times \ldots \times F_n$ is finitely presented  and intersects each factor $F_i$ non-trivially, then each $F_i$ contains a finite index normal subgroup $K_i$ such that $$\gamma_{n-1}(K_i) \leq G \cap F_i \leq K_i.$$
\end{thm}

In the $3$-factor case this yields the following result.

\begin{thm}[\cite{Bridson2}] \label{thm18}
  Let $F_1, F_2, F_3$ be finitely generated free groups and let $G \leq F_1 \times F_2 \times F_3$ be a subdirect product which intersects each factor $F_i$ non-trivially.  Then $G$ is finitely presented if and only if each $F_i$ contains a finite index normal subgroup $K_i$ such that the subgroup $G' = G \cap (K_1 \times K_2 \times K_3)$ satisfies the following condition:  there is an abelian group $Q$ and epimorphisms $\phi_i : K_i \rightarrow Q$ such that $G'$ is the kernel of the map $\phi_1 + \phi_2 + \phi_3$.
\end{thm}

The previous two results suggest that the first step in an attack on Question~\ref{q1} is to restrict attention to those groups in $\SDP(\mc{F})$ which virtually contain the commutator subgroup of the ambient direct product.  The BNS invariants (see \cite{Bieri2}, \cite{Bieri3} for definitions) of direct products of free groups have been calculated by Meinert \cite{mein94} and so, given its finiteness type, one can readily determine how such a co-abelian subgroup sits inside the direct product, and \textit{vice versa}.

One interpretation of Question~\ref{q1} is as a prototype for a much more profound question regarding the class $\mc{L}$ of \emph{limit groups}.  In \cite{BridsonUP} the authors ask the first part of the following question:

\begin{quest} \label{q2}
  Does every finitely presented group in $\SDP(\mc{L})$ admit a polynomial isoperimetric function?  Does there exist a uniform polynomial isoperimetric function for the whole class?
\end{quest}

Limit groups were introduced by Sela (\cite{Sela01} \textit{et seq.}) and separately by Kharalampovich and Myasnikov (\cite{Kharlampovich98}, \cite{Kharlampovich98B}, \cite{Kharlampovich06}) in their solutions to Tarski's question of which groups have the same elementary theory as finitely generated non-abelian free groups.  The class contains all finitely generated free and free abelian groups and all compact surface groups of Euler characteristic $<-1$.  In some sense, $\mc{L}$ is the class of groups that are `almost free'; indeed, one fascinating aspect of the theory is that several \textit{a priori} unrelated notions of what it means to be `almost free' turn out to define the same class of groups.

The simplest definition of limit groups is that they are the finitely generated fully residually free groups, where a group $G$ is defined to be \emph{fully residually free} if for every finite subset $X \subseteq G$ there exists a homomorphism $G \rightarrow F$ to a non-abelian free group that is injective on $X$.  From a logical perspective, limit groups are precisely the finitely generated groups with the same existential theory as non-abelian free groups; from a geometric perspective, they are the finitely generated groups that have a Cayley graph in which each ball of finite radius is isometric to a ball of the same radius in some Cayley graph of a free group of finite rank.  Limit groups can also be defined in an algebraic context as limits of stable homomorphisms to a free group.

Aside from its own intrinsic interest, several results add further weight to Question~\ref{q2}.  It follows from a result of Baumslag, Myasnikov and Remeslennikov \cite{Baumslag99} and of Sela \cite{Sela01}, that the finitely presented groups in $\SDP(\mc{L})$ are precisely the finitely presented residually free groups.  In a more geometric direction, work of Delzant and Gromov \cite{Delzant1} implies that an answer to Question~\ref{q2} would provide important information about the isoperimetric behaviour of K\"{a}hler groups and compact K\"{a}hler manifolds.

Bridson, Howie, Miller and Short \cite{BridsonUP} have proved that the analogues of Theorems~\ref{thm16}, \ref{thm17} and \ref{thm18} hold with the words `free groups' replaced by `limit groups'.  Building on this and other structural results in \cite{Bridson1}, Kouchloukova \cite{Kochloukova1} proved that if $G$ is a subgroup of a direct product $D = L_1 \times \ldots \times L_n$ of limit groups (with certain additional conditions) and if $G$ is of type $\textrm{FP}_{s}(\Q)$ for some $s \geq 2$, then the projection homomorphism from $G$ to the direct product of any $s$ of the $L_i$ is virtually surjective.  It follows that if $G$ is a subgroup of a direct product of $n \geq 3$ limit groups and if $G$ is of type $\textrm{FP}_{n-1}(\Q)$, then $G$ contains a finite index subgroup $G'$ isomorphic to the kernel of a homomorphism $\phi: L_1 \times \ldots \times L_m \rightarrow A$ where $L_1, \ldots, L_m$ are limit groups, $A$ is abelian, $m \leq n$, and restriction of $\theta$ to each factor $L_i$ is surjective.

One interpretation of a direct product of free groups is as an example (perhaps the canonical example) of a type of group known as a \emph{right-angled Artin group} (RAAG).  Much of the interest in RAAGs amongst geometric group theorists stems from the fact that their definition is flexible enough for them to admit interesting subgroups, and yet they possess enough structure (in particular they have finite $K(\pi_1,1 )$-complexes with the structure of non-positively curved cube complexes) to enable the proof of interesting results.  For example, Bestvina and Brady \cite{Bestvina1} defined a collection of subgroups of RAAGs (known as Bestvina-Brady groups --- see Section~\ref{sec3} for definitions) in their solution to the old problem of whether the finiteness conditions $\mathrm{F}_2$ and $\mathrm{FP}_2$ are equivalent.  They also constructed a Bestvina-Brady group $G$ such that either $G$ is a counterexample to the Eilenberg-Ganea conjecture, or else there exists a counterexample to the Whitehead conjecture.

In general the richness of the subgroup structure of RAAGs suggests that questions about their arbitrary finitely presented subgroups will be hard.  It it thus natural to begin by restricting attention to the Bestvina-Brady subgroups.

\begin{quest} \label{q3}
  Do all finitely presented Bestvina-Brady groups admit a polynomial isoperimetric function?  Does their exist a uniform polynomial isoperimetric inequality?
\end{quest}

In \cite{Brady1}, Brady suggests that the answer to the second part of this question is no: he constructs a sequence $(\Gamma_k)_{k=1}^\infty$ of finitely presented Bestvina-Brady groups and claims that the Dehn function of $\Gamma_k$ is polynomial of degree $k+2$.  However, a result in this thesis shows that in fact $n^4$ is an isoperimetric function for all finitely presented Bestvina-Brady groups, and hence Brady's construction can not be made to work.

Questions~\ref{q1}--\ref{q3} acted as the guides for much of the research in this thesis; we have obtained partial answers to Questions~\ref{q1} and \ref{q2}, and a complete answer to Question~\ref{q3}.  The structure of the thesis is as follows.  After describing our notation in Section~\ref{sec13}, Section~\ref{sec12} gives the required background on Dehn functions and other related filling invariants.  All of this material is standard, although some of the terminology is novel.  Section~\ref{sec9} then gives a brief introduction to distortion functions: just as the Dehn function gives a particular measure of the complexity of the word problem for a finitely presented group, so the distortion function gives a measure of the complexity of the membership problem for a pair of finitely generated groups $H \leq G$.  Again, the material in this section is standard.  Although this thesis is primarily concerned with Dehn functions, when investigating subdirect products our methods will frequently also give analogous results concerning distortion.

From Section~\ref{sec14} onwards all results are original, except where stated.  In Sections~\ref{sec14}--\ref{sec8} we prove various general results of a preliminary nature, that are then applied in Sections~\ref{sec7}--\ref{sec3} in an attack on Questions~\ref{q1}--\ref{q3}.  Each section begins with an introduction explaining its contents.

Guided by Theorems~\ref{thm17} and \ref{thm18} (and their limit group analogues) we focus in Section~\ref{sec7} on a class of subdirect products which virtually contain the commutator subgroup of the ambient direct product.  For definitions of the terms `virtually-full', `virtually-coabelian' and `corank', see Section~\ref{sec16}

\begin{letteredthm} \label{thm20}
  Let $H$ be a virtually-full, virtually-coabelian subgroup of a direct product $D = \Gamma_1 \times \ldots \times \Gamma_n$, with corank $r$. \begin{enumerate}
    \item Suppose each $\Gamma_i$ is finitely generated and $n \geq 2$.  Then $H$ is finitely generated and the distortion function $\Delta$ of $H$ in $D$ satisfies $\Delta(l) \preccurlyeq l^2$.

    \item Suppose each $\Gamma_i$ is finitely presented and $n \geq 3$.  Then $H$ is finitely presented.

    \item Suppose each $\Gamma_i$ is finitely presented and $n \geq 3$.  For each $i$, let $(\alpha_i, \rho_i)$ be an area-radius pair for some finite presentation of $\Gamma_i$.  Define $$\alpha(l) = \max ( \{l^2\} \cup \{ \alpha_i(l) \, : \, 1 \leq i \leq n\})$$ and $$\rho(l) = \max ( \{l\} \cup \{ \rho_i(l) \, : \, 1 \leq i \leq n\}).$$  Then $\rho^{2r} \alpha$ is an isoperimetric function for $H$

    \item Suppose that each $\Gamma_i$ is finitely presented and that $n \geq \max\{3, 2r\}$.  Let $\beta_1$ and $\beta_2$ be the Dehn functions of some finite presentations of $\Gamma_1 \times \ldots \times \Gamma_{n-r}$ and $\Gamma_{n-r+1} \times \ldots \times \Gamma_n$ respectively.  Then the function $\beta$ defined by $$\beta(l) = l\beta_1(l^2) + \beta_2(l)$$ is an isoperimetric function for $H$.
  \end{enumerate}
\end{letteredthm}

In Section~\ref{sec15} we focus on subgroups of direct products of limit groups, and use Theorem~\ref{thm20} to prove the following result.

\begin{letteredthm}
  Let $L_1, \ldots, L_n$ be limit groups and let $H$ be a subgroup of the direct product $D = L_1 \times \ldots \times L_n$.  Suppose that $H$ is of type $\mathrm{FP}_{m}(\Q)$, where $m = \max \{2, n-1 \}$.  Then $H$ is finitely presented and satisfies a polynomial isoperimetric inequality, and the distortion function $\Delta$ of $H$ in $D$ satisfies $\Delta(l) \preccurlyeq l^2$.
\end{letteredthm}

In particular this result applies to all finitely presented subgroups of a direct product of at most $3$ limit groups:

\begin{letteredcor}
  Let $H$ be a finitely presented subgroup of a direct product $D$ of at most $3$ limit groups.  Then $H$ satisfies a polynomial isoperimetric inequality and the distortion function $\Delta$ of $H$ in $D$ satisfies $\Delta(l) \preccurlyeq l^2$.
\end{letteredcor}

These results provide a partial solution to Questions~\ref{q1} and \ref{q2}.

In Section~\ref{sec10} we focus on a class of subdirect products of free groups which have particularly regular structure.  This class includes the Stallings-Bieri groups, and also contains what are perhaps the next most simple groups in $\SDP(\mc{F})$ which are not already well understood.

\begin{letteredthm}
  Let $F_1, F_2, F_3$ be rank $2$ free groups and, for each $i$, let $\theta_i : F_i \rightarrow \Z^2$ be the abelianisation homomorphism.  Define $\theta : F_1 \times F_2 \times F_3 \rightarrow \Z^2$ to be the homomorphism $\theta_1 + \theta_2 + \theta_3$.  Then the kernel of $\theta$ is finitely presented and has Dehn function $\delta$ satisfying $\delta(l) \succeq l^3$.
\end{letteredthm}

This provides the first known example of a group in $\SDP(\mc{F})$ that has Dehn function growing faster than that of the ambient direct product.  We also derive an explicit finite presentation for this group.

Finally, in Section~\ref{sec3}, we prove the following result, which gives a complete solution to Question~\ref{q3}.

\begin{letteredthm}
  Every finitely presented Bestvina-Brady group has $l^4$ as an isoperimetric function.
\end{letteredthm}

\section{Notation} \label{sec13}

Given a set $\mc{A}$, write $\mc{A}^{-1}$ for the set $\{ a^{-1} : a \in \mc{A} \}$ of formal inverses to the elements of $\mc{A}$ and write $\mc{A}^{\pm1}$ for the set $\mc{A} \cup \mc{A}^{-1}$.  Write $\fm{A}$ for the free monoid on $\mc{A}^{\pm1}$ and $\Fr(\mc{A})$ for the free group on $\mc{A}$.  We call the elements of $\mc{A}^{\pm1}$ \emph{letters} and the elements of $\fm{A}$ \emph{words}.  Given words $w_1, w_2 \in \fm{A}$, write $w_1 \equiv w_2$ if $w_1$ and $w_2$ are equal as elements of $\fm{A}$ and $w_1 \FreeEq w_2$ if $w_1$ and $w_2$ are equal as elements of $\Fr(\mc{A})$.  Write $\emptyset$ for the empty word.

Given a word $w = a_1 \ldots a_n \in \fm{A}$, write $|w|$ for the length $n$ of $w$ and $\|w\|$ for the length of the free reduction of $w$, \textit{i.e.} the length of the unique freely reduced word $w'$ with $w \FreeEq w'$.  Write $w(i)$ for the $i^{\text{th}}$ letter $a_i$ of $w$ and $w[i]$ for the $i^{\text{th}}$ prefix $a_1 \ldots a_i$ of $w$.  If $i > |w|$ then set $w[i] \equiv w$.  Write $w^{-1}$ for the inverse word $a_n^{-1} \ldots a_1^{-1}$.  Given a set of words $\mc{S} \subseteq \fm{A}$, write $\mc{S}^{-1}$ for the set of inverses $\{s^{-1} : s \in \mc{S} \}$ and $\mc{S}^{\pm1}$ for the set $\mc{S} \cup \mc{S}^{-1}$.  Given words $w_1, \ldots, w_n \in \fm{A}$, write $\prod_{j=1}^n w_i$ for the concatenated word $w_1 \ldots w_n$.  Given letters $a_1, a_2 \in \mc{A}^{\pm1}$, write $[a_1, a_2]$ for the word $a_1 a_2 a_1^{-1} a_2^{-1} \in \fm{A}$, write $a_1^{a_2}$ for the word $a_2 a_1 a_2^{-1} \in \fm{A}$, and write $a_1^{-a_2}$ as shorthand for $\left(a_1^{a_2}\right)^{-1} \equiv a_2 a_1^{-1} a_2^{-1}$.  If $\mc{A}$ is a generating set for a group $G$, then write $d_\mc{A}$ for the word metric on $G$ with respect to $\mc{A}$.

As well as considering words as being elements of the free monoid on an alphabet, we sometimes, abusing notation, take the viewpoint that words are maps: we consider a word as being a function which assigns to an ordered set $\mc{S}$ of fixed, finite cardinality an element of $\fm{S}$.  For example, if $\mc{S} = \{x, y\}$ and $\mc{S}' = \{x', y' \}$, and $w(\mc{S}) = xyx$, then $w(\mc{S}') = x'y'x'$.  More generally, we will also sometimes consider words which take as input an $n$-tuple of finite ordered sets $\mc{S}_1, \ldots, \mc{S}_n$ and output a word in $(\mc{S}_1 \cup \ldots \cup \mc{S}_n)^{\pm\bast}$.  In this context, by, for example, $w(\mc{S}_1, \emptyset)$ we mean the image of $w(\mc{S}_1, \mc{S}_2)$ under the projection map $(\mc{S}_1 \cup \mc{S}_2)^{\pm\bast} \rightarrow \mc{S}_1^{\pm\bast}$.  It will always be clear from context whether we are using the term `word' in the sense of being a map $w$ or in the more usual sense of being an evaluation of $w$ on a specific set.

\section{Filling functions} \label{sec12}

Throughout this section $\mc{P} = \langle \mc{X} \, | \, \mc{R} \rangle$ is a group presentation with $\mc{X}$ finite.   We introduce the notions of $\mc{P}$-expressions, $\mc{P}$-sequences, $\mc{P}$-pictures and $\mc{P}$-van Kampen diagrams which provide means for representing null-homotopies of words in $\fm{X}$.  This allow us to define various \emph{filling invariants}, including Dehn functions, isoperimetric functions and area-radius pairs.  Aside from some of the terminology, all of the definitions given here are standard, except that we do not make the usual assumption that $\mc{R}$ is finite.  For a more thorough introduction to these ideas, see, for example, \cite{brid02}, \cite{rile05}, \cite{Fenn1} or \cite{Pride1}.

\subsection{Representing null-homotopies}

\begin{defn}
  A word $w \in \fm{X}$ is said to be \emph{null-homotopic} over $\mc{P}$ if it represents the identity in the group presented by $\mc{P}$.
\end{defn}

\begin{defn}[$\mc{P}$-expressions]
  A \emph{$\mc{P}$-expression} is a finite sequence $\mc{E} = (x_i, r_i)_{i=1}^m$ of elements of $\fm{X} \times \mc{R}^{\pm1}$.  The \emph{area} of $\mc{E}$, written $\Area(\mc{E})$, is defined to be the integer $m$.  The \emph{radius} of $\mc{E}$, written $\Rad(\mc{E})$, is defined to be $\max \{ |x_i| \, : \, 1 \leq i \leq m \}$.  We allow the empty sequence which is defined to have both zero area and zero radius.  We write $\partial \mc{E}$ for the word $\prod_{i=1}^m x_i r_i x_i^{-1}$.  If $\mc{E}_1$ and $\mc{E}_2$ are $\mc{P}$-expressions then we write $\mc{E}_1 \mc{E}_2$ for the $\mc{P}$-expression given by concatenating the two sequences.  A $\mc{P}$-expression for a word $w \in \fm{X}$ is a $\mc{P}$-expression $\mc{E}$ with $\partial \mc{E}$ freely equal to $w$.
\end{defn}

\begin{defn}[$\mc{P}$-sequences]
  A \emph{$\mc{P}$-sequence} is a sequence $\Sigma = (\sigma_i)_{i=0}^m$ of words in $\fm{X}$ where, for each $i$, the word $\sigma_{i+1}$ is obtained from $\sigma_i$ in one of the following ways: \begin{itemize}
  \item  \textbf{Free contraction:} $\sigma_i \equiv u x x^{-1} v$ and $\sigma_{i+1} \equiv uv$, where $u, v \in \fm{X}$ and $x \in \mc{X}^{\pm1}$.

  \item \textbf{Free expansion:} $\sigma_i \equiv u v$ and $\sigma_{i+1} \equiv u x x^{-1} v$, where $u, v \in \fm{X}$ and $x \in \mc{X}^{\pm1}$.

  \item \textbf{Application-of-a-relator move:} $\sigma_i \equiv u r v$ and $\sigma_{i+1} \equiv u s v$, where $u, v \in \fm{X}$ and $rs^{-1}$ is a cyclic conjugate of a word in $\mc{R}^{\pm1}$.
\end{itemize}  Such a $\mc{P}$-sequence is said to convert the word $\sigma_0$ to the word $\sigma_m$.  A \emph{null $\mc{P}$-sequence} for a word $w \in \fm{X}$ is a $\mc{P}$-sequence converting $w$ to the empty word $\emptyset$.  The \emph{area} of a $\mc{P}$-sequence $\Sigma = (\sigma_i)_{i=0}^m$, written $\Area(\Sigma)$, is defined to be the number of $i$ for which the transition from $\sigma_i$ to $\sigma_{i+1}$ is an application-of-a-relator move.  If $\Sigma_1 = (\sigma_i^{(1)})_{i=0}^{m_1}$ and $\Sigma_2 = (\sigma_i^{(2)})_{i=0}^{m_2}$ are $\mc{P}$-sequences with $\sigma_{m_1}^{(1)} \equiv \sigma_0^{(2)}$ then we write $\Sigma_1 \Sigma_2$ for the $\mc{P}$-sequence $(\sigma_0^{(1)}, \ldots, \sigma_{m_1}^{(1)}, \sigma_1^{(2)}, \ldots, \sigma_{m_2}^{(2)})$.  Note that $\Area(\Sigma_1 \Sigma_2) = \Area(\Sigma_1) + \Area(\Sigma_2)$.
\end{defn}

\begin{defn}[$\mc{P}$-pictures]
  A $\mc{P}$-picture $\mathbb{P}$ consists of a closed $2$-disc $D$ (the \emph{ambient disc}); a collection of closed $2$-discs $D_1, \ldots, D_m$ (the \emph{relator discs}) embedded pairwise disjointly in the interior of $D$; and a collection of compact, connected, normally orientated $1$-manifolds $\alpha_1, \ldots, \alpha_l$ (the \emph{arcs}) embedded pairwise disjointly in $D \smallsetminus \cup_{i=1}^m \Int D_i$.  The ambient disc $D$ is equipped with a basepoint $b \in \partial D$, and each relator disc $D_i$ is equipped with a basepoint $b_i \in \partial D_i$.  We require that each arc is disjoint from all basepoints, and that the interior of each arc is disjoint from $\partial D$ and disjoint from each $D_i$.  Each relator disc is labelled by an element of $\mc{R}^{\pm1}$ and each arc is labelled by an element of $\mc{X}$.

  Reading anticlockwise from its basepoint around the boundary of a relator disc or the ambient disc defines a word in $\fm{X}$, where we understand that if we pass an arc labelled $x$ in the direction of its normal orientation then we read $x$, and if we pass the arc in the opposite direction to its normal orientation we read $x^{-1}$.  We require that the word associated to each relator disc in this way is precisely the element of $\mc{R}^{\pm1}$ labelling the disc.

  The \emph{area} of $\mathbb{P}$, written $\Area \mathbb{P}$, is defined to be the number of relator discs.  Define the \emph{background} of $\mathbb{P}$ to be $$\Back \mathbb{P} := D \smallsetminus \left( (\cup_{i=1}^m D_i) \bigcup (\cup_{i=1}^l \alpha_l) \right).$$ By a \emph{complementary region} of $\mathbb{P}$ we mean a connected component of $\Back \mathbb{P}$.  Given points $p,q \in \Back \mathbb{P}$ a \emph{transverse path} from $p$ to $q$ is a path in $D \smallsetminus \cup_{i=1}^m D_i$ with initial point $p$ and terminal point $q$ which intersects each arc $\alpha_i$ transversely and only finitely many times.  Define the \emph{intersection number} of such a path to be the number of times it intersects $\cup_{i=1}^l \alpha_i$.  Given a complementary region $C$, define $d(b, C)$ to be the minimum intersection number over all transverse paths from $b$ to a point in $C$.  Define the \emph{radius} of $\mathbb{P}$, written $\Rad \mathbb{P}$, to be the maximum value of $d(b, C)$ over all complementary regions $C$.

  The \emph{boundary label} of $\mc{P}$ is defined to be the word in $\fm{X}$ given by reading anticlockwise around $\partial D$ from the basepoint $b$.  A $\mc{P}$-picture $\mathbb{P}$ for a word $w \in \fm{X}$ is a $\mc{P}$-picture with boundary label $w$.
\end{defn}

In order to give our fourth, and final, means of representing null-homotopies, namely van Kampen diagrams, we require the notion of a \emph{combinatorial CW-complex}.

\begin{defn}
  A cellular map between CW-complexes is said to be \emph{combinatorial} if its restriction to each open cell of the   domain complex is a homeomorphism onto some open cell of the   codomain complex.

  The notion of a CW-complex being \emph{combinatorial} is defined by   recursion on dimension.  By definition every $0$-dimensional   CW-complex is combinatorial.  An $n$-dimensional CW-complex $X$ is   combinatorial if $X^{(n-1)}$ is combinatorial and for each $n$-cell   $e_i^n$ the attaching map $\theta_i^n : \mathbb{S}^{n-1} \rightarrow   X^{(n-1)}$ is combinatorial for some combinatorial CW-complex structure   on $\mathbb{S}^{n-1}$. \end{defn}

\begin{defn}[$\mc{P}$-van Kampen diagrams]
  A \emph{singular disc diagram} $\Delta$ is a finite, planar,
  contractible combinatorial CW-complex with a specified base vertex
  $\star$ in its boundary.  The \emph{area} of $\Delta$, written
  $\Area(\Delta)$, is defined to be the number of $2$-cells of which
  $\Delta$ is composed.  The \emph{boundary cycle} of $\Delta$ is
  the edge loop in $\Delta$ which starts at $\star$ and traverses
  $\partial \Delta$ in the anticlockwise direction.  The
  interior of $\Delta$ consists of a number of disjoint open
  $2$-discs, the closures of which are called the \emph{disc components} of $\Delta$.

  Each $1$-cell of $\Delta$ has associated to it two directed
  edges $\epsilon_1$ and $\epsilon_2$, with $\epsilon_1^{-1} =
  \epsilon_2$.  Let $\DEdge(\Delta)$ be the set of directed edges
  of $\Delta$.  A \emph{labelling} of $\Delta$ over a set $\mc{S}$ is
  a map $\lambda : \DEdge(\Delta) \rightarrow \mc{S}^{\pm1}$ such that
  $\lambda(\epsilon^{-1}) = \lambda(\epsilon)^{-1}$.  This induces a
  map from the set of edge paths in $\Delta$ to $\fm{S}$.
  The \emph{boundary label} of $\Delta$ is the word in $\fm{S}$
  associated to the boundary cycle.

  A \emph{$\mc{P}$-van Kampen diagram} for a word $w
  \in \fm{X}$ is a singular disc diagram $\Delta$ labelled over $\mc{X}$ with
  boundary label $w$ and such that for each $2$-cell $c$ of $\Delta$
  the anticlockwise edge loop given by the attaching map of $c$,
  starting at some vertex in $\partial c$, is labelled by a word in
  $\mc{R}^{\pm1}$.
\end{defn}

\begin{defn}[Cayley complexes]
  The \emph{presentation $2$-complex} of $\mc{P}$ is a combinatorial $2$-complex consisting of a single $0$-cell; orientated $1$-cells in bijective correspondence with $\mc{X}$; and $2$-cells in bijective correspondence with $\mc{R}$.  The $2$-cell associated to a relator $r \in \mc{R}$ has $|r|$ edges and is attached by identifying its boundary circuit with the edge path along which the word $r$ is read.

  The \emph{Cayley $2$-complex} $Cay^2(\mc{P})$ of $\mc{P}$ is defined to be the universal cover of the presentation $2$-complex.  The edges of $Cay^2(\mc{P})$ inherit labels and orientations from the presentation $2$-complex.  If $G$ ia the group presented by $P$ then, after choosing a basepoint, the $0$-skeleton of $Cay^2(\mc{P})$ is identified with $G$ and there is a natural left action of $G$ on $Cay^2(\mc{P})$.  The \emph{Cayley graph} $Cay^1(G, \mc{X})$ of $G$ with respect to $\mc{X}$ is defined to be the $1$-skeleton of $G$.
\end{defn}

If $\Delta$ is $\mc{P}$-van Kampen diagram then there is a unique combinatorial basepoint-preserving and label-preserving map $\Delta \rightarrow Cay^2(\mc{P})$.

\subsection{Dehn functions and the areas of words}

\begin{defn}[van Kampen's Lemma]
  The following are equivalent for a word $w \in \fm{X}$: \begin{itemize}
    \item $w$ is null-homotopic;
    \item there exists a $\mc{P}$-expression for $w$;
    \item there exists a null $\mc{P}$-sequence for $w$;
    \item there exists a $\mc{P}$-picture for $w$;
    \item there exists a $\mc{P}$-van Kampen diagram for $w$.
  \end{itemize}  Furthermore, if $w$ is null-homotopic, then the following integers are equal: \begin{itemize}
    \item $\min \{ \Area(\mc{E}) :  \text{$\mc{E}$ a $\mc{P}$-expression for $w$} \}$;
    \item $\min \{ \Area(\Sigma) :  \text{$\Sigma$ a null-$\mc{P}$-sequence for $w$} \}$;
    \item $\min \{ \Area(\mathbb{P}) :  \text{$\mathbb{P}$ a $\mc{P}$-picture for $w$} \}$;
    \item $\min \{ \Area(\Delta) :  \text{$\Delta$ a $\mc{P}$-van Kampen diagram for $w$} \}$;
  \end{itemize}
   and these all serve to define the \emph{area} of $w$, written $\Area(w)$.  If we wish to emphasise which presentation we are working with we talk of the \emph{$\mc{P}$-area} of $w$, written $\Area_\mc{P}(w)$.
\end{defn}

\begin{defn}
  The \emph{Dehn function} of $\mc{P}$ is defined to be the function $\delta_\mc{P} : \N \rightarrow \N$ given by $$\delta_\mc{P}(l) = \max \{ \Area(w) \, : \, w \in \fm{X} \text{ is null-homotopic and } |w| \leq l \}.$$
\end{defn}

Different finite presentations of the same group may have different Dehn functions, but, in a way which we now make precise, the Dehn functions will have the same asymptotic behaviour.

\begin{defn}
  Let $f, g$ be functions $\N \rightarrow \N$.  Write $f \preceq g $ if there exists a constant $C \in \N$ so that $f(l) \leq C g(Cl + C) + Cl + C$.  Write $f \simeq g$ if $f \preceq g$ and $g \preceq f$.
\end{defn}

The following lemma is standard, see for example \cite{brid02}.

\begin{lem} \label{lem27}
  Let $Q$ be a finite presentation presenting the same group as $\mc{P}$.  Then $\delta_\mc{P} \simeq \delta_\mc{Q}$.
\end{lem}

Thus, up to $\simeq$-equivalence, it makes sense to talk about \emph{the} Dehn function of a finitely presented group.  We emphasise that although we will sometimes make use of infinite presentations as calculatory tools, the Dehn function of a finitely presented group always refers to the Dehn function of some finite presentation of the group.

\begin{defn}
  Let $G$ be a finitely presented group.  Then a function $\alpha: \N \rightarrow \N$ is said to be an \emph{isoperimetric function} for $G$ if $\delta_\mc{P} \preceq \alpha$ for some (and hence any) choice of finite presentation $\mc{P}$ for $G$.  We say that $G$ satisfies a \emph{polynomial isoperimetric inequality} if it has a polynomial as an isoperimetric function.
\end{defn}

\subsection{$\mc{P}$-schemes}

In this thesis we will frequently present bounds on the areas of words, and we wish to convey to the reader how these bounds have been derived.  For reasons of space and readability we wish to avoid having to display all of the data required to define a particular null-homotopy.  Instead we make use of the notion of null $\mc{P}$-schemes, which are essentially skeletons of null-homotopies and which provide enough detail to allow the reader to reconstruct a particular null-homotopy and hence a bound on the area of the word in question.

\begin{defn}
  A \emph{$\mc{P}$-scheme} consists of a finite sequence of words $(\sigma_i)_{i=1}^m$ in $\fm{X}$ and a finite sequence of integers $(\alpha_i)_{i=1}^{m-1}$ such that, for each $i$, the word $\sigma_i (\sigma_{i+1})^{-1}$ is null-homotopic over $\mc{P}$ with area at most $\alpha_i$.  Such a $\mc{P}$-scheme is said to convert the word $\sigma_1$ to the word $\sigma_m$.  We frequently display $\mc{P}$-schemes in a table, with the $i^\text{th}$ row containing the word $\sigma_i$ and the number $\alpha_i$.  Since there is a disparity between the number of terms in the sequences $(\sigma_i)$ and $(\alpha_i)$, the last row of such a table will consist of just the word $\sigma_m$.

A \emph{null $\mc{P}$-scheme} for a word $w \in \fm{X}$ is a $\mc{P}$-scheme converting $w$ to the empty word.  When displaying a null $\mc{P}$-scheme in a table we omit the final row, since this does not contain any non-trivial data.  Note that if there exists a null $\mc{P}$-scheme for a word $w$, then $w$ is null-homotopic over $\mc{P}$ with area at most the sum of the integers $\alpha_i$.
\end{defn}

As an example, suppose that $\mc{P}$ is the presentation $\langle x, y \, | \, [x, y] \rangle$ of a rank $2$ free abelian group.  The following null $\mc{P}$-scheme demonstrates that the word $x^2 y x^{-1} y x y x^{-2} y^{-3}$ is null-homotopic over $\mc{P}$ with area at most $5$.

\medskip

  \begin{center}
    \begin{tabular}{ccc}
      \hline
      \\[-10pt]
      $j$ & $\sigma_j$ & Area \\
      \\[-10pt]
      \hline
      \\[-10pt]
      $1$ & $x^2 y x^{-1} y x y x^{-2} y^{-3}$ & $2$ \\
      $2$ & $x^2 y x^{-1} y x^{-1} y^{-2}$ & $1$ \\
      $3$ & $x^2 y x^{-2} y^{-1}$ & $2$ \\
      \\[-10pt]
      \hline
      \\[-10pt]
      & Total & $5$ \\
      \\[-10pt]
      \hline
    \end{tabular}
  \end{center}

\subsection{Area-radius pairs}
As well as bounding the areas of $\mc{P}$-expressions for words, we sometimes wish to simultaneously bound their radii.

\begin{defn}
  A pair of functions $(\alpha, \rho)$, each $\N \rightarrow \N$, is said to be an \emph{area-radius pair} for $\mc{P}$ if, for every null-homotopic word $w \in \fm{X}$ with $|w| \leq l$, there exists a $\mc{P}$-expression $\mc{E}$ for $w$ with $\Area(\mc{E}) \leq \alpha(l)$ and $\Rad(\mc{E}) \leq \rho(l)$.
\end{defn}

The following result shows how area-radius pairs transform under change of presentation.

\begin{prop} \label{prop11}
  Let $\mc{P}$ and $\mc{Q}$ be finite presentations of the same group.  If $(\alpha, \rho)$ is an area-radius pair for $\mc{P}$ then there exists an area-radius pair $(\alpha', \rho')$ for $\mc{Q}$ with $\alpha \simeq \alpha'$ and $\rho \simeq \rho'$.
\end{prop}

\begin{proof}
  Since $\mc{P}$ can be converted to $\mc{Q}$ by a finite sequence of Tietze transformations, it suffices to prove the proposition in the situation that $\mc{P}$ and $\mc{Q}$ are related by a single such transformation.  There are four cases to consider.

  \textbf{Case 1.}  Suppose that $\mc{P} = \langle \mc{A} \, | \, \mc{R} \rangle$ and $\mc{Q} = \langle \mc{A} \, | \, \mc{R}, s \rangle$ where $s \in \fm{A}$ is null-homotopic over $\mc{P}$. A $\mc{P}$-expression for a word $w \in \fm{A}$ is also a $\mc{Q}$-expression for $w$, so $(\alpha, \rho)$ is itself an area-radius pair for $\mc{Q}$.

  \textbf{Case 2.}  Suppose that $\mc{P} = \langle \mc{A} \, | \, \mc{R}, s \rangle$ and $\mc{Q} = \langle \mc{A} \, | \, \mc{R} \rangle$ where $s \in \fm{A}$ is null-homotopic over $\mc{Q}$.  Let $(x_i, r_i)_{i=1}^M$ be a $\mc{Q}$-expression for $s$ with area $M$ and radius $K$.  If $w \in \fm{A}$ is a null-homotopic word of length at most $n$ then there exists a $\mc{P}$-expression $\Sigma = (y_i, z_i)_{i=1}^L$ for $w$ with area $L \leq \alpha(n)$ and radius at most $\rho(n)$.  Substituting $\prod_{i=1}^M x_i r_i x_i^{-1}$ for each occurrence of $s$ in the product $\prod_{i=1}^L y_i z_i y_i^{-1}$ gives a product which is freely equal to $w$ in $F(\mc{A})$.  The corresponding $\mc{Q}$-expression has area at most $ML$ and radius at most $\rho(n) + K$.  Thus $(M \alpha(n), \rho(n) +K)$ is an area-radius pair for $\mc{Q}$.

  \textbf{Case 3.}  Suppose that $\mc{P} = \langle \mc{A} \, | \mc{R} \rangle$ and $\mc{Q} = \langle \mc{A}, b \, | \, \mc{R}, bu_b^{-1} \rangle$ where $u_b \in \fm{A}$ and $b u_b^{-1}$ is null-homotopic over $\mc{P}$.  Define $K = |u_b|$.  Suppose $w \in (\mc{A} \cup \{b\})^{\pm\bast}$ is a null-homotopic word of length at most $n$; say $w \equiv v_0 b^{\epsilon_1} v_1 \ldots b^{\epsilon_L} v_L$ for some $v_i \in \fm{A}$ and $\epsilon_i \in \{\pm1\}$.  Insert cancelling pairs $u_b^{-1} u_b$ into $w$ to obtain the word $w' \equiv v_0 (b u_b^{-1} u_b)^{\epsilon_1} v_1 \ldots (b u_b^{-1} u_b)^{\epsilon_L} v_L$ with $w' \FreeEq w$. Define $v_0', \ldots, v_L'$ to be the words in $\fm{A}$ such that $w' \equiv v_0' (b u_b^{-1})^{\epsilon_1} v_1' \ldots (b u_b^{-1})^{\epsilon_L} v_L'$ and note that $\sum_{i=1}^L |v_i'| \leq K|w| \leq Kn$.  For each $i \in \{0, \ldots, L\}$ define $\tau_i \equiv v_i' v_{i+1}' \ldots v_L'$.  Then $$w' \FreeEq  \tau_0 \prod_{i=1}^L \tau_i^{-1} (b u_b)^{\epsilon_i} \tau_i$$ and $|\tau_i| \leq \sum_{i=1}^L |v_i'| \leq Kn$.  The word $\tau_0$ is null-homotopic over $\mc{Q}$ and hence over $\mc{P}$ and so there exists a $\mc{P}$-expression $(x_i, r_i)_{i=1}^M$ for $\tau_0$ with area at most $\alpha(Kn)$ and radius at most $\rho(Kn)$.  Thus $$w \FreeEq \prod_{i=1}^M x_i r_i x_i^{-1} \prod_{i=1}^L \tau_i^{-1} (b u_b^{-1})^{\epsilon_i} \tau_i$$ and so we obtain a $\mc{Q}$-expression for $w$ with area at most $M + L \leq \alpha(Kn) + n$ and radius at most $\max \{ \max_i |x_i|, \max_i|v_i'| \} \leq \max\{\rho(Kn), Kn\} \leq \rho(Kn) + Kn$.  Thus $(\alpha(Kn) + n, \rho(Kn) + Kn)$ is an area-radius pair for $\mc{Q}$.

  \textbf{Case 4.}  Suppose that $\mc{P} = \langle \mc{A}, b \, | \, \mc{R}, bu_b^{-1} \rangle$ and $\mc{Q} = \langle \mc{A} \, | \mc{R} \rangle$ where $u_b \in \fm{A}$ and $b u_b^{-1}$ is null-homotopic over $\mc{Q}$.  Define $K = |u_b|$.  Consider the retraction $\pi : (\mc{A} \cup \{b\})^{\pm\bast} \rightarrow \fm{A}$ which is the identity on $\mc{A}$ and maps $b^{\pm1} \mapsto u_b^{\pm1}$.  Note that $\pi$ induces a retraction $F(\mc{A} \cup \{b\}) \rightarrow F(\mc{A})$.  Suppose $w \in \fm{A}$ is a null-homotopic word of length at most $n$ and let $(x_i, z_i)_{i=1}^M$ be a $\mc{P}$-expression for $w$ with area at most $\alpha(n)$ and radius at most $\rho(n)$.  Let $S$ be the subset of $\{1, \ldots, m\}$ consisting of those $i$ for which $z_i \in \mc{R}^{\pm1}$.  Then $(\pi(x_i), \pi(z_i))_{i \in S}$ is a $\mc{Q}$-expression for $w$ with area at most $M$ and radius at most $K \rho(n)$.  Thus $(\alpha(n), K\rho(n))$ is an area-radius pair for $\mc{Q}$.
\end{proof}

As with the areas of words, area-radius pairs have interpretations in terms of $\mc{P}$-sequence, $\mc{P}$-pictures and $\mc{P}$-van Kampen diagrams; of these we only consider the pictorial interpretation.

\begin{defn}
  A pair $(\alpha, \rho)$ of functions $\alpha, \rho : \N \rightarrow \N$ is said to be a \emph{pictorial area-radius pair} for the presentation $\mc{P}$ if, for all null-homotopic words $w \in \fm{A}$ with $|w| \leq l$, there exists a $\mc{P}$-picture $\mathbb{P}$ for $w$ with $\Area \mathbb{P} \leq \alpha(l)$ and $\Rad \mathbb{P} \leq \rho(l)$.
\end{defn}

\begin{prop}\label{prop12}
  If $(\alpha, \rho)$ is an area-radius pair for a presentation $\mc{P}$ then there exists a pictorial area-radius pair $(\alpha', \rho')$ for $\mc{P}$ with $\alpha \simeq \alpha'$ and $\rho \simeq \rho'$.  Conversely if $(\alpha, \rho)$ is a pictorial area-radius pair for $\mc{P}$ then there exists an area-radius pair $(\alpha', \rho')$ for $\mc{P}$ with $\alpha \simeq \alpha'$ and $\rho \simeq \rho'$.
\end{prop}

The only place in this thesis where we make use of Proposition~\ref{prop12} is in the proof of Theorem~\ref{thm2}.  We thus omit the proof of this proposition since Theorem~\ref{thm2} is implied by the stronger Theorem~\ref{thm1}.

\subsection{Finite index subgroups}

We will frequently simplify arguments by passing to finite index subgroups.  The following lemma shows that Dehn functions and area-radius pairs are unaffected by this transition.

\begin{lem} \label{lem22}
  Let $H \leq G$ be a pair groups with finite presentations $\mc{P}$ and $\mc{Q}$ respectively.  Suppose that $H$ has finite index in $G$. \begin{enumerate}
    \item Let $\delta_\mc{P}$ and $\delta_\mc{Q}$ be the Dehn functions of $\mc{P}$ and $\mc{Q}$ respectively.  Then $\delta_\mc{P} \simeq \delta_\mc{Q}$.

    \item Let $(\alpha, \rho)$ be an area-radius pair for $\mc{Q}$.  Then there exists an area-radius pair $(\alpha', \rho')$ for $\mc{P}$ with $\alpha \simeq \alpha'$ and $\rho \simeq \rho'$.
  \end{enumerate}
\end{lem}

NB: It is also true that if $(\alpha, \rho)$ is an area-radius pair for $\mc{P}$ then there exists an area-radius pair $(\alpha', \rho')$ for $\mc{Q}$ with $\alpha \simeq \alpha'$ and $\rho \simeq \rho'$.  However we will not need this result.

\begin{proof}[Proof of Lemma~\ref{lem22}]
  For (1), observe that since $H$ has finite index in $G$, these two groups are quasi-isometric.  The result then follows since quasi-isometric groups have $\simeq$-equivalent Dehn functions \cite{Alonso90}.

  The assertion (2) is standard.  We give our own proof in Section~\ref{sec8} as a corollary to Proposition~\ref{prop4}.
\end{proof}

\section{Distortion Functions} \label{sec9}

Let $H \leq G$ be a pair of groups with finite generating sets $\mc{X}$ and $\mc{Y}$ respectively.  The \emph{distortion function} of $H$ in $G$ with respect to $\mc{X}$ and $\mc{Y}$ is defined to be the function $\Delta : \N \rightarrow \N$ given by $$\Delta(l) = \max \{ d_\mc{X} (1, h) \, : \, h \in H, d_\mc{Y}(1, h) \leq l \}.$$  Different choices of generating sets will give rise to different distortion functions, but, in a way we now make precise, these will have the same asymptotic behaviour.

\begin{defn}
  Let $f, g$ be functions $\N \rightarrow \N$.  Write $f \preccurlyeq g$ if there exists a constant $C \in \N$ so that $f(l) \leq Cg(Cl)$.  Write $f \approx g$ if $f \preccurlyeq g$ and $g \preccurlyeq f$.
\end{defn}

The following lemma is standard.

\begin{lem} \label{lem23}
  For each $i = 1,2$, let $\Delta_i$ be the distortion function of $H$ in $G$ with respect to some finite generating sets $\mc{X}_i$ and $\mc{Y}_i$ for $H$ and $G$ respectively.  Then $\Delta_1 \approx \Delta_2$.
\end{lem}

Thus we may talk of \emph{the} distortion function of $H$ in $G$, without making any mention of a choice of generating sets, provided we bear in mind that this is only defined up to $\approx$-equivalence.

We say $H$ has \emph{polynomial distortion} in $G$ if the distortion function with respect to some (and hence any) finite generating sets is bounded above by a polynomial.  We say $H$ is \emph{undistorted} in $G$ if the distortion function with respect to some (and hence any) finite generating sets is linear.  For example, finite index subgroups are undistorted, as are direct factors or, more generally, retracts.

The following lemma gives various transitivity properties of distortion functions.

\begin{lem} \label{lem24}
  Let $G_1 \leq G_2 \leq G_3$ be groups with finite generating sets $\mc{X}_1$, $\mc{X}_2$, $\mc{X}_3$ respectively.  For each $1 \leq i < j \leq 3$, let $\Delta_i^j$ be the distortion function of $G_i$ in $G_j$ with respect to $\mc{X}_i$ and $\mc{X}_j$.  \begin{enumerate}
    \item $\Delta_1^3(l) \leq \Delta_1^2 (\Delta_2^3(l))$.

    \item If $G_1$ has finite index in $G_2$ then $\Delta_2^3 \approx \Delta_1^3$.

    \item If $G_2$ has finite index in $G_3$ then $\Delta_1^2 \approx \Delta_1^3$.
  \end{enumerate}
\end{lem}

\begin{proof}
  Property (1) is immediate.

  For (2),  the direction $\Delta_1^3 \preccurlyeq \Delta_2^3$ follows immediately from property (1).  For the converse, note that by Lemma~\ref{lem23} we are at liberty to choose any finite generating sets convenient to our purposes.  Choose $\mc{X}_2$ to contain a collection $k_1, \ldots, k_n$ of right coset representatives of $G_1$ in $G_2$ and choose $\mc{X}_3$ to contain $\mc{X}_2$.  Let $w \in \mc{X}_3^{\pm \bast}$ represent a non-identity element of $G_2$.  Then there exists $i$ so that $wk_i^{-1}$ represents an element of $G_1$.  Since $wk_i^{-1} \in \mc{X}_3^{\pm \bast}$, there exists $w' \in \mc{X}_1^{\pm \bast}$ representing the same element as $w k_i^{-1}$ with $|w'| \leq \Delta_1^3(|w| + 1) \leq \Delta_1^3(2|w|)$.  Then $w' k_i \in \mc{X}_2^{\pm \bast}$ represents $w$ and has length at most $\Delta_1^3(2|w|) + 1 \leq 2\Delta_1^3(2|w|)$.

  For (3), the direction $\Delta_1^3 \preccurlyeq \Delta_1^2$ follows immediately from property (1).  For the converse, choose the generating set $\mc{X}_3$ to contain $\mc{X}_2$.  Then $\mc{X}_2^{\pm\bast} \subseteq \mc{X}_3^{\pm\bast}$ and so $\Delta_1^2(l) \leq \Delta_1^3(l)$ for all $l$.
\end{proof}

\begin{cor} \label{cor1}
  Let $H$, $H'$ and $G'$ be finitely generated subgroups of the finitely generated group $G$, with $H' \leq H \cap G'$.  Suppose that $H'$ has finite index in $H$ and $G'$ has finite index in $G$.  Let $\Delta$ and $\Delta'$ be the distortion functions of $H$ in $G$ and $H'$ in $G'$ respectively.  Then $\Delta \approx \Delta'$.
\end{cor}

\begin{lem} \label{lem25}
  Let $H \leq G$ be finitely generated groups and let $p: G \rightarrow G'$ be a surjective homomorphism which is injective on $H$.  Let $\Delta$ and $\Delta'$ be the distortion functions of $H$ in $G$ and $p(H)$ in $G'$ respectively.  Then $\Delta \preccurlyeq \Delta'$.
\end{lem}

\begin{proof}
  Let $\mc{X}$ and $\mc{Y}$ be finite generating sets for $H$ and $G$ respectively.  Define $\mc{X}' = p(\mc{X})$ and $\mc{Y}' = p(\mc{Y})$ and note that these are finite generating sets for $p(H)$ and $G'$ respectively.  By Lemma~\ref{lem23} we may assume that $\Delta$ and $\Delta'$ are defined with respect to these generating sets.  Then, for any $g_1, g_2 \in G$ one then has that $\dist_{\mc{Y}'}(p(g_1), p(g_2)) \leq \dist_\mc{Y}(g_1, g_2)$.  Since the restriction of $p$ to $H$ is an isomorphism onto its image, $\dist_{\mc{X}'}(p(h_1), p(h_2)) = \dist_\mc{X}(h_1, h_2)$ for all $h_1, h_2 \in H$.  Thus for any $l \in \N$, we have the inclusion of sets $$\{\dist_\mc{X}(1, h) : h \in H, \dist_\mc{Y}(1, h) \leq l \} \subseteq \{\dist_{\mc{X}'}(1, h) : h \in p(H), \dist_{\mc{Y}'}(1, h) \leq l\}.$$  It follows that $\Delta(l) \leq \Delta'(l)$.
\end{proof}

\section{The Bounded Noise Lemma} \label{sec14}

Let $\mc{P} = \langle \mc{A} \, | \, \mc{R} \rangle$ be a finite presentation with area-radius pair $(\alpha, \rho)$ and define $L = \max \{ |r| : r \in \mc{R} \}$.  If $w$ is a null-homotopic word over $\mc{P}$ with $|w| \leq n$ then there exists a $\mc{P}$-expression $\mc{E}$ for $w$ with area $\leq \alpha(n)$ and radius $\leq \rho(n)$.  Thus $|\partial \mc{E}| \leq (2 \rho(n) + L)\alpha(n)$.  The Bounded Noise Lemma shows that $\mc{E}$ can be chosen so that the free reduction of the word $\partial \mc{E}$ is bounded only in terms of $\alpha$.  This lemma is not original, but a proof of it does not appear to exist in the literature.  Recall that we write $|w|$ for the length of a word $w$, and $\|w\|$ for the length of the free reduction of $w$.

\begin{lem}[The Bounded Noise Lemma] \label{lem28}
  Let $w$ be a null-homotopic word over the presentation $\mc{P}$ with area $N$.  Then there exists a $\mc{P}$-expression $(u_i, r_i)_{i=1}^N$ for $w$ with $$\|u_1\| + \sum_{i=1}^{N-1} \|u_i^{-1} u_{i+1}\| + \|u_N\| \leq |w| + 2LN.$$
\end{lem}

The proof of this result makes use of the following notions concerning van Kampen diagrams.  Say the anticlockwise boundary cycle of a van Kampen diagram $D$, read from the base vertex, is given by the edge path $e_1 \cdot \ldots \cdot e_k$, where $e_1, \ldots, e_k$ are edges of $D$ (possibly with repetition) and $\cdot$ denotes concatenation.  Let $e_i$ be the first edge lying in the boundary of some $2$-cell of $D$.  Then we call the edge $e_i$ the \emph{first thick boundary edge} of $D$ and the edge path $e_1 \cdot \ldots \cdot e_{i-1}$ the \emph{initial boundary segment} of $D$.

\begin{proof}
  We actually prove the following:

  \begin{clm}
    Let $\Delta$ be a $\mc{P}$-van Kampen diagram for the word $w$ with area $N$.  Then there exist words $s_1, \ldots, s_N \in \mathcal{A}^{\pm\bast}$ labelling $2$-cells of $\Delta$, each read anticlockwise from some vertex, and there exist words $v_1, \ldots, v_N \in \mathcal{A}^{\pm\bast}$ with $v_1$ the label on the initial boundary segment of $\Delta$, such that $(v_i, s_i)_{i=1}^N$ is a $\mc{P}$-expression for $w$ and $$|v_1| + \sum_{i=1}^{N-1} \|v_i^{-1} v_{i+1} \| + \|v_N\| \leq |w| + LN.$$
  \end{clm}

  The lemma as stated follows from the claim since each $s_i$ is a cyclic conjugate of some relator $r_i \in \mathcal{R}^{\pm1}$ and so is freely equal to a word $x_i r_i x_i^{-1}$ for some $x_i \in \mathcal{A}^{\pm\bast}$ with $|x_i| \leq |r_i|/2 \leq L/2$.  It follows that we can set $u_i = v_i x_i$ and then $(u_i, r_i)_{i=1}^N$ is a $\mc{P}$-expression for $w$ and \begin{align*}
    &\|u_1\| + \sum_{i=1}^{N-1} \|u_i^{-1} u_{i+1}\| + \|u_N\| \\
    &\quad = \|v_1x_1\| + \sum_{i=1}^{N-1} \|x_i^{-1} v_i^{-1} v_{i+1} x_{i+1}\| + \|v_N x_N\| \\
    &\quad \leq \|v_1\| + \|x_1\| + \sum_{i=1}^{N-1} \left( \|v_i^{-1} v_{i+1}\| + \|x_i\| + \|x_{i+1}\| \right) + \|v_N\| + \|x_N\| \\
    &\quad \leq |v_1| + |x_1| + \sum_{i=1}^{N-1} \left( \|v_i^{-1} v_{i+1}\| + |x_i| + |x_{i+1}| \right) + \|v_N\| + |x_N| \\
    &\quad \leq |w| + LN + L/2 + L/2 + (N-1)(L/2 + L/2) \\
    &\quad = |w| + 2LN.
  \end{align*}

  The claim is proved by induction on the area of $\Delta$.  If $\Delta$ has area $0$ the conclusion is trivial.  Now suppose that $\Delta$ has area $N \geq 1$ and that the claim is true for diagrams with smaller area.  Say $\Delta$ has boundary label $w$ and initial boundary segment labelled by the word $v_1$.  Let $e$ be the first thick boundary edge of $\Delta$ and let $c$ be the unique $2$-cell of $\Delta$ that contains $e$ in its boundary.  The anticlockwise orientation of the boundary cycle of $\Delta$ induces an orientation on the edge $e$.  Say $c$ has boundary label $s_1$ read anticlockwise from the origin of $e$.

  Let $\Delta'$ be the van Kampen diagram of area $N-1$ formed from $\Delta$ by deleting the (interior of the) $2$-cell $c$ and the (interior of the) edge $e$.  Say $\Delta'$ has boundary label $w'$.  Observe that $w$ is freely equal to the word $v_1 s_1 v_1^{-1} w'$ and that $|w'| \leq |w| + L$. Applying the induction hypothesis to $\Delta'$ gives that there exist $v_2, \ldots, v_N \in \mathcal{A}^{\pm\bast}$ with $v_2$ the label on the initial boundary segment of $\Delta'$ and there exist $s_2, \ldots, s_N \in \mathcal{A}^{\pm\bast}$ labelling $2$-cells of $\Delta'$ such that $$w' \FreeEq \prod_{i=2}^N v_i s_i v_i^{-1}$$ and $$|v_2| + \sum_{i=2}^{N-1} \|v_i^{-1}v_{i+1}\| + \|v_N\| \leq |w'| + L(N-1).$$  Thus $$w \FreeEq \prod_{i=1}^N v_i s_i v_i^{-1}.$$

  By construction the initial boundary segment of $\Delta'$ is formed by concatenating the initial boundary segment of $\Delta$ with a (possibly empty) edge path $\gamma$.  Let $\alpha$ be the label on $\gamma$, so $v_2 \equiv v_1 \alpha$.  Then $|v_2| = |v_1| + |\alpha| \geq |v_1| + \|\alpha\| = |v_1| + \|v_1^{-1} v_2\|$ and so \begin{align*}
    |v_1| + \sum_{i=1}^{N-1} \|v_i^{-1} v_{i+1} \| + \|v_N\| &\leq |v_2| -
    \|v_1^{-1} v_2\| + \sum_{i=1}^{N-1} \|v_i^{-1} v_{i+1} \| + \|v_N\| \\
    &= |v_2| + \sum_{i=2}^{N-1} \|v_i^{-1} v_{i+1} \| + \|v_N\| \\
    &\leq |w'| + L(N-1) \\
    &\leq |w| + LN.
  \end{align*}
\end{proof}

\section{Infinite presentations} \label{sec1}

In the process of deriving a finite presentation for a group, we will sometimes find it useful to first produce, as an intermediate stage, a presentation with infinitely many relations.  Care must be taken when dealing with the isoperimetry of such non-finite presentations.  The Dehn functions of different finite presentations of a fixed group all have the same asymptotic behaviour.  However, the same is not true for presentations with an infinite number of relators, where the behaviour of the Dehn functions may differ markedly.  Indeed, for any group, if we take the set of relators to consist of all null-homotopic words then we obtain a presentation whose Dehn function is constant.  In order to regain some control over how the Dehn function changes when changing between (possibly non-finite) presentations, we introduce the following notions.

\begin{defn} \label{def1}
  An \emph{index} on a set $\mc{X}$ is a function $\| \cdot \| :   \mc{X} \rightarrow \N$.  This is extended to an index on the set   $\mc{X}^{\pm1}$ by setting $\| x^{-1} \| = \| x \|$.  An   \emph{indexed presentation} is a pair $(\mc{P}, \| \cdot \|)$   where $\mc{P} = \langle \mc{A} \, | \, \mc{R} \rangle$ is a presentation   and $\| \cdot \|$ is an index on $\mc{R}$.

  Let $(\mc{P}, \| \cdot \|)$ be an indexed presentation whose set of   generators $\mc{A}$ is finite.  A pair $(\alpha, \pi)$ of   functions $\alpha, \pi : \N \rightarrow \N$ is said to be an \emph{area-penetration pair} for $(\mc{P}, \| \cdot \|)$ if for all null-homotopic words   $w \in \mc{A}^{\pm\bast}$ with $|w| \leq n$ there exists a   null-$\mc{P}$-expression $(x_i, r_i)_{i=1}^m$ for $w$ with area   $m \leq \alpha(n)$ and with $\|r_i\| \leq \pi(n)$ for each $i$.

Let $\mc{Q} = \langle \mc{A} \, | \, \mc{S} \rangle$ be a presentation with each $s \in \mc{S}$ null-homotopic over $\mc{P}$ and each $r \in \mc{R}$ null-homotopic over $\mc{Q}$.  Thus $\mc{P}$ and $\mc{Q}$ present the same group.    The \emph{relational area   function} of $(\mc{P}, \| \cdot \|)$ over   $\mc{Q}$ is defined to be the function $\N \rightarrow \N \cup \{\infty\}$ given   by $$\RArea (n) = \max \{ \Area_\mc{Q} (r) \,   : \, r \in \mc{R}, \|r\| \leq n \}.$$
\end{defn}

\begin{prop} \label{prop1}
  Let $(\mc{P}, \| \cdot \|)$ and $\mc{Q}$ be as in definition   \ref{def1}.  Let $(\alpha, \pi)$ be an area-penetration   pair for $(\mc{P}, \| \cdot \|)$ and let $\RArea$ be the relational   area function of $(\mc{P}, \| \cdot \|)$ over   $\mc{Q}$.  Then the Dehn function $\delta_\mc{Q}$ of the   presentation $\mc{Q}$ satisfies $$\delta_\mc{Q}(n) \leq   \alpha(n) \RArea(\pi(n)).$$ \end{prop}

\begin{proof}
  Let $w \in \fm{A}$ be a null-homotopic word with $|w| \leq   n$.  Then there exist $\sigma_1, \ldots, \sigma_N \in \fm{A}$   and $r_1, \ldots, r_N \in \mc{R}^{\pm1}$ with $N \leq   \alpha(n)$ and $\|r_i\| \leq \pi(n)$   for each $i$ such that $$w \FreeEq \prod_{i=1}^N \sigma_i r_i   \sigma_i^{-1}.$$  For each $i$ we have that   $\Area_\mc{Q}(r_i) \leq \RArea(\|r_i\|) \leq \RArea(\pi(n))$ and therefore   there exist $\tau_{i1}, \ldots, \tau_{i M_i} \in   \fm{A}$ and $s_{i 1}, \ldots, s_{i M_i} \in   \mc{S}^{\pm1}$ with $M_i \leq \RArea(\pi(n))$ such that   $$r_i \FreeEq \prod_{j=1}^{M_i} \tau_{ij} s_{ij} \tau_{ij}^{-1}.$$  Hence $$w   \FreeEq \prod_{i=1}^N \prod_{j=1}^{M_i} (\sigma_i \tau_{ij}) s_{ij}   (\sigma_i \tau_{ij})^{-1}$$  and so $\Area_\mc{Q}(w) \leq   \sum_{i=1}^N M_i  \leq \alpha(n) \, \RArea(\pi(n))$.
\end{proof}

Section~\ref{sec2} contains a result, Theorem~\ref{thm1}, concerning area-penetration pairs and cyclic extensions.  Although we give a full algebraic proof of this theorem, the intuition behind it derives from the pictorial context and so we will sketch a proof of the slightly weaker Theorem~\ref{thm2} in this language.  We will thus need the pictorial analogue of area-penetration pairs.

\begin{defn}
  Let $(\mc{P}, \| \cdot \|)$ be an indexed presentation whose set of   generators $\mc{A}$ is finite.  A pair $(\alpha, \pi)$ of   functions $\alpha, \pi : \N \rightarrow \N$ is said to be a   \emph{pictorial area-penetration pair} for $(\mc{P}, \| \cdot \|)$ if for all   null-homotopic words $w \in \mc{A}^{\pm\bast}$ with $|w| \leq n$ there exists a   picture $\mathbb{P}$ with boundary label $w$ such that $\Area \mathbb{P} \leq   \alpha(n)$ and $\|r\| \leq \pi(n)$ for each relator $r$ of $\mc{P}$ labelling a relator   disc of $\mathbb{P}$. \end{defn}

\begin{prop}\label{prop2}
  A pair $(\alpha, \pi)$ of functions $\alpha, \pi : \N \rightarrow   \N$ is an area-penetration pair for a presentation   $\mc{P}$ if and only if it is a pictorial area-penetration pair   for $\mc{P}$. \end{prop}

Since we do not rely on this proposition for the proof of Theorem~\ref{thm1}, we omit its proof.

\section{Cyclic extensions} \label{sec2}

Let $1 \rightarrow K \rightarrow \Gamma \rightarrow \Z \rightarrow 1$ be a cyclic extension with $K$ (and hence $\Gamma$) finitely generated.  In all of the applications presented in this thesis, $\Gamma$ will be finitely presented, but we do not need to make this assumption.  In the principal result of this section (Theorem~\ref{thm1}) we show how a presentation $\mc{P}_\Gamma$ of $\Gamma$ (of a certain form) gives rise to an infinite presentation $\mc{P}_K^\infty$ for $K$.  The relators of $\mc{P}_K^\infty$ come equipped with an index $\| \cdot \|$ and we prove that an area-radius pair for $\mc{P}_\Gamma$ is actually an area-penetration pair for $(\mc{P}_K^\infty, \| \cdot \|)$.  However, before we introduce this new material, we first recall a result of Baik-Harlander-Pride.

Let $\mc{A}$ be a finite generating set for $K$ and let $t \in \Gamma$ be an element whose image generates $\Gamma / K \cong \Z$.  Let $\theta$ be the automorphism of $K$ induced by conjugation by $t$.  For each $a \in \mc{A}$ and $\epsilon \in \{\pm 1\}$, let $w_a^\epsilon \in \fm{A}$ be a word representing $t^\epsilon a t^{-\epsilon}$ in $K$.  For each $\epsilon \in \{\pm 1\}$, define $\mc{S}^{\epsilon} = \{ t^\epsilon a t^{-\epsilon} (w_a^\epsilon)^{-1} : a \in \mc{A} \}$.  Furthermore, define an endomorphism $\Phi^\epsilon : \fm{A} \rightarrow \fm{A}$, commuting with the inversion automorphism, by mapping $a \mapsto w_a^\epsilon$.

\begin{thm}[Baik-Harlander-Pride {\cite[Theorem~6.1]{Baik1}}] \label{thm15}
  Let $\langle \mc{A}, t \, | \, \mc{R}, \mc{S}^+, \mc{S}^- \rangle$ be a presentation for $\Gamma$ with $\mc{R} \subseteq \fm{A}$.  Suppose that all the relations in the sets $\{ a \Phi^-(\Phi^+(a))^{-1} : a \in \mc{A} \}$ and $\{ \Phi^\epsilon(r) : \epsilon \in \{\pm1\}, r \in \mc{R} \}$ are null-homotopic over the presentation $\langle \mc{A} \, | \, \mc{R} \rangle$.  Then $K$ is presented by $\langle \mc{A} \, | \, \mc{R} \rangle$.
\end{thm}

We will apply Theorem~\ref{thm15} in Section~\ref{sec10} to derive finite presentations for certain subdirect products of free groups.  However, the proof of this result in \cite{Baik1} is based on successively removing $t$-rings from van Kampen diagrams over the presentation $\langle \mc{A}, t \, | \, \mc{R}, \mc{S}^+, \mc{S}^- \rangle$, a method which will in general only give an exponential isoperimetric function for $K$.  Since we will be interested in producing polynomial isoperimetric inequalities we adopt a different approach, which essentially involves removing all $t$-rings simultaneously.  We begin with a minor technicality.

\begin{defn}
  A presentation $\langle \mc{A}, t \, | \, \mc{T} \rangle$ for $\Gamma$ is said to be in \emph{positive normal form} if, for each $a \in \mc{A}$, there is precisely one relator in $\mc{T}$ of the form $tat^{-1} w$ with $w \in \fm{A}$, \emph{and}, all the relators in $\mc{T}$ involving $t$ are of this form.
\end{defn}

Thus, given words $w_a^+$ as defined above, a presentation $\langle \mc{A}, t \, | \, \mc{R}, \mc{S}^+ \rangle$ for $\Gamma$ with $\mc{R} \subseteq \fm{A}$ is in positive normal form.  In particular, if $\langle \mc{A} \, | \, \mc{R} \rangle$ is a presentation for $K$, then $\langle \mc{A}, t \, | \, \mc{R}, \mc{S}^+ \rangle$ is in positive normal form.  The following lemma shows that restricting our attention to positive normal form presentations does not impinge on the generality of our results.

\begin{lem}
  If $\Gamma$ is finitely presented then it is presented by some finite presentation in positive normal form.
\end{lem}

\begin{proof}
  Let $\langle \mc{A} \, | \mc{R} \rangle$ be an arbitrary (not necessarily finite) presentation for $K$.  Then $\Gamma$ is presented by the positive normal form presentation $\langle \mc{A}, t \, | \, \mc{R}, \mc{S}^+ \rangle$.  Since $\Gamma$ is finitely presented there is some finite subcollection of $\mc{R} \cup \mc{S}^+$ which suffice as a set of defining relators.  In particular, there exists a finite subset $\mc{R}' \subseteq \mc{R}$ so that $\Gamma$ is finitely presented by $\langle \mc{A}, t \, | \, \mc{R}', \mc{S}^+ \rangle$.
\end{proof}

Now let $\mc{P}_\Gamma = \langle \mc{A}, t \, | \, \mc{R}, \mc{S} \rangle$  be a positive normal form presentation for $\Gamma$ with $\mc{R} \subseteq \fm{A}$ and $\mc{S} = \{ tat^{-1}w_a^{-1} : a \in \mc{A} \}$ for some words $w_a \in \fm{A}$.  For each $k \in \Z$, let $\Phi_k : \mc{A}^{\pm\bast} \rightarrow \mc{A}^{\pm\bast}$ be an endomorphism that lifts $\theta^k : K \rightarrow K$ and commutes with the inversion involution of $\mc{A}^{\pm\bast}$.  We take $\Phi_0$ to be the identity.  Define the following collections of words in $\fm{A}$: \begin{align*}\overline{\mc{R}} &= \{ \Phi_k(r) \, : \, r \in \mc{R}, k \in \Z \}\\ \overline{\mc{S}} &= \{ \Phi_{k+1}(a) \Phi_k(w_{a})^{-1} \, : \, a \in \mc{A}, k \in \Z\}.\end{align*}  Note that each word in $\overline{\mc{R}} \cup \overline{\mc{S}}$ is null-homotopic in $K$.  Define $\mc{P}_K^\infty = \langle \mc{A} \, | \, \ol{\mc{R}}, \ol{\mc{S}} \rangle$ and define an index $\| \cdot \|$ on $\overline{\mc{R}} \cup \overline{\mc{S}}$ by setting $\|\omega \|$ to be the minimal value of $|k|$ such that either $\omega \equiv \Phi_k(r)$ for some $r \in \mc{R}$ or $\omega \equiv \Phi_{k+1}(a) \Phi_k(w_{a})^{-1}$ for some $a \in \mc{A}$.

\begin{thm} \label{thm1}
  $K$ is presented by $\mc{P}_K^\infty$.  Furthermore, if $(\alpha, \rho)$ is an area-radius pair for $\mc{P}_\Gamma$ then it is also an area-penetration pair
  for the indexed presentation $(\mc{P}_K^\infty, \| \cdot \|)$.
\end{thm}

The utility of Theorem~\ref{thm1} is that if one can demonstrate that each word in $\ol{\mc{R}} \cup \ol{\mc{S}}$ is null-homotopic over some finite presentation $\mc{P}_K$, then it will follow that $\mc{P}_K$ presents $K$.  Furthermore, by applying Proposition~\ref{prop1} one can obtain an upper bound on the Dehn function of $\mc{P}_K$.

The following slightly weaker version of Theorem~\ref{thm1} will actually be sufficient for our purposes.  This result also has the advantage that its proof can be seen intuitively in the language of pictures.  However, we wish to avoid having to prove the equivalence given in Propositions~\ref{prop12} and \ref{prop2} between algebraically and pictorially defined area-radius and area-penetration pairs. We thus give a proof of Theorem~\ref{thm2} in the language of pictures and follow this with an algebraic proof of Theorem~\ref{thm1}.

\begin{thm} \label{thm2}
  $K$ is presented by $\mc{P}_K^\infty$.  Furthermore, if $(\alpha, \rho)$ is an area-radius pair for   $\mc{P}_\Gamma$ then there exist functions $\alpha', \rho' :   \N \rightarrow \N$ with $\alpha \simeq \alpha'$ and $\rho \simeq \rho'$ such that $(\alpha', \rho')$ is an area-penetration pair   for the indexed presentation $(\mc{P}_K^\infty, \| \cdot \|)$. \end{thm}

\begin{proof}
  Let $w \in \fm{A}$ be a null-homotopic word of length at most $n$. By Proposition~\ref{prop12} there exists a pictorial area-radius pair $(\alpha', \rho')$ for $\mc{P}_\Gamma$ such that $\alpha \simeq \alpha'$ and $\rho \simeq \rho'$.  Let $\mathbb{P}$ be a $\mc{P}_\Gamma$-picture with boundary word $w$ such that $\Area \mathbb{P} \leq \alpha'(n)$ and $\Rad \mathbb{P} \leq \rho'(n)$.  Say $\mathbb{P}$ has ambient disc $D$, basepoint $b$, relator discs
  $D_1, \ldots, D_m$ and arcs $\gamma_1, \ldots, \gamma_l$.

  We now describe how to assign to each complementary region $C$ of $\mathbb{P}$ an element $g(C)$ of $\Gamma$.  If $\sigma$ is a transverse path between points in $\Back \mathbb{P}$ then reading along $\sigma$ defines a word $W(\sigma) \in (\mc{A} \cup \{t\})^{\pm\bast}$, where we understand that if $\sigma$ crosses an arc labelled $x$ in the direction of its normal orientation then we read $x$, and if $\sigma$ crosses the arc in the opposite direction to its normal orientation then we read $x^{-1}$.  By \cite[Theorem 2.3]{Pride1} if $p_1, p_2 \in \Back \mathbb{P}$ and $\tau$ and $\tau'$ are transverse paths from $p_1$ to $p_2$ then $W(\tau)$ and $W(\tau')$ represent the same element in $\Gamma$.  Given a point $p \in \Back \mathbb{P}$ define $g(p)$ to be the element of $\Gamma$ represented by a transverse path $\sigma$ from $b$ to $p$.  If $p'$ lies in the same complementary region $C$ as $p$ then we can adjoin to $\sigma$ a path from $p$ to $p'$ lying wholly in $C$ to obtain a transverse path $\sigma'$ from $b$ to $p'$ with $W(\sigma) = W(\sigma')$.  Thus $g(p) = g(p')$ and we
  can define $g(C)$ to be this element of $\Gamma$.

  By an $\mc{A}$-arc of $\mathbb{P}$ we mean an arc labelled by a letter in $\mc{A}$.  We now show how to assign a height $h(\gamma) \in \Z$ to each $\mc{A}$-arc $\gamma$.  Let $\overline{t}$ be the image of $t$ under the quotient homomorphism $q : \Gamma \rightarrow \Gamma/K \cong \Z$ and define the height $h(C)$ of a complementary region $C$ of $\mathbb{P}$ to be the exponent of $\overline{t}$ in $q(g(C))$.  Now suppose that $\gamma$ is an arc of $\mathbb{P}$ labelled by the letter $a \in \mc{A}$.  Say $\gamma$ lies in the boundary of the complementary regions $C_1$ and $C_2$, which may or may not be distinct.  We will show that $h(C_1) = h(C_2)$ and define $h(\gamma)$ to be this number.  Let $\sigma_1$ be a transverse path from $b$ to a point $p_1 \in C_1$ and let $\tau$ be a transverse path from $p_1$ to a point $p_2 \in C_2$ which intersects $\gamma$ exactly once and intersects no other arcs of $\mathbb{P}$.  Then the composition $\sigma_2$ of $\sigma_1$ and $\tau$ is a transverse path from $b$ to $p_2$ with $W(\sigma_2) = W(\sigma_1)W(\tau) = W(\sigma_1)a^{\pm1}$ in $(\mc{A} \cup \{t\})^{\pm\bast}$.  Thus $g(C_2) =
  g(C_2)a^{\pm1}$ in $\Gamma$ and so $h(C_2) = h(C_1)$.

  Note that for each complementary region $C$ we can choose a transverse path from $b$ to a point in $C$ with intersection number at most $\Rad \mathbb{P}$ and so $|h(C)| \leq \Rad \mathbb{P}$.  It follows that for all $\mc{A}$-arcs $\gamma$ one
  similarly has $|h(\gamma)| \leq \Rad \mathbb{P}$.

  We now modify $\mathbb{P}$ to produce a $\mc{P}_K^\infty$-picture $\ol{\mathbb{P}}$ for the word $w$.  This is done by deleting each $\mc{A}$-arc $\gamma_i$ labelled by a letter $a$ and replacing it by a collection of $l_i := \left| \Phi_{h(\alpha_i)}(a) \right|$ parallel arcs $\gamma_i^1, \ldots, \gamma_i^{l_i}$ labelled by the letters of the word $\Phi_{h(\gamma_i)}(a)$. We now describe precisely what we mean by this. Say $\gamma_i$ joins $\partial \Lambda_i^\iota$ to $\partial \Lambda_i^\tau$, where $\Lambda_i^\iota, \Lambda_i^\tau \in \{D, D_1, \ldots, D_m\}$. Let $N_i^\iota$ and $N_i^\tau$ be neighbourhoods of $\gamma_i \cap \Lambda_i^\iota$ and $\gamma_i \cap \Lambda_i^\tau$ in $\partial \Lambda_i^\iota$ and $\partial \Lambda_i^\tau$ respectively.  We choose $N_i^\iota$ and $N_i^\tau$ to be homoeomorphic to the unit interval and to be disjoint from all basepoints and all other arcs of $\mathbb{P}$.  Each $\gamma_i^j$ joins $N_i^\iota$ to $N_i^\tau$ and we choose them so as they are all disjoint and their interiors are disjoint from $\cup_{k=1}^m D_k$.  We orientate and label the arcs $\gamma_i^j$ so as reading along $N_i^\iota$ in the direction of the orientation of $\gamma_i$ gives the word $\Phi_{h(\gamma_i)}(a)$.  The picture $\ol{\mathbb{P}}$ is now completed by deleting all the arcs $\gamma_i$ labelled by the
  letter $t$.

  If a disc $D_i$ had label $r \in \mc{R}^{\pm1}$ in $\mathbb{P}$ then all the arcs incident with $D_i$ in $\mathbb{P}$ had the same height $h$.  Thus the corresponding disc in $\ol{\mathbb{P}}$ has label $\Phi_h(r) \in \overline{R}^{\pm1}$ for some $h$ with $|h| \leq \Rad \mathbb{P}$.  If the disc $D_i$ had the label $\left(tat^{-1}w_{a}^{-1}\right)^{\pm1} \in \mc{S}^{\pm1}$ in $\mathbb{P}$ then the incident arc labelled $a$ had height $h$ and the incident arcs labelled by the letters of $w_{a}$ had height $h-1$, for some $h \in \Z$.  Thus the corresponding disc in $\ol{\mathbb{P}}$ has label $\left(\Phi_{h}(a) \Phi_{h-1}(w_{a})^{-1}\right)^{\pm1} \in \overline{\mc{S}}^{\pm1}$ for some $h$ with $|h|$ and $|h-1|$ at most $\Rad \mathbb{P}$.

  By a boundary arc of $\mathbb{P}$ we will mean an arc with at least one of its endpoints lying in $\partial D$.  Note that all boundary arcs of $\mathbb{P}$ are $\mc{A}$-arcs.  If $C$ is a complementary region of $\mathbb{P}$ with the boundary of its closure intersecting $\partial D$ non-trivially, then there exists a transverse path in $\mathbb{P}$ from $b$ to $C$ which intersects only boundary arcs.  Thus $C$ has zero height.  It follows that all the boundary arcs of $\mathbb{P}$ have zero height and hence that the boundary label of $\ol{\mathbb{P}}$ is $\Phi_0(w) \equiv w$.  Thus $\ol{\mathbb{P}}$ is a $\mc{P}_K^{\infty}$-picture for the word $w$, with $\Area \ol{\mathbb{P}} = \Area \mathbb{P}$ and with each relator $z \in (\ol{\mc{R}} \cup \ol{\mc{S}})^{\pm1}$ labelling a disc of $\ol{\mathbb{P}}$ having $\|z\|
  \leq \Rad \mathbb{P}$.

  Since the word $w$ was arbitrary it follows that $\mc{P}_K^\infty$ presents $K$ and has $(\alpha', \rho')$ as a pictorial area-penetration pair.  By Proposition~\ref{prop2} it follows that $(\alpha', \rho')$ is also an area-penetration pair for $\mc{P}_K^{\infty}$.
\end{proof}

\begin{proof}[Proof of Theorem \ref{thm1}]
  Let $w \in \fm{A}$ be a null-homotopic word of length at most $n$ and let $(x_i, z_i)_{i=1}^m$ be a $\mc{P}_\Gamma$-expression for $w$ with $m \leq \alpha(n)$ and with $|x_i| \leq \rho(n)$ for
  each $i$.

  We write $h(u)$ for the exponent sum in the letter $t$ of a word $u \in (\mc{A} \cup \{t\})^{\pm\bast}$ and define $\widetilde{N}$ to be the submonoid of $(\mc{A} \cup \{t\})^{\pm\bast}$ consisting of all those words $u$ with $h(u) = 0$.  Define $\mc{X}$ to be the set of words $\{t^k a t^{-k} \, : \, a \in \mc{A}, k \in \Z \} \leq (\mc{A} \cup \{t\})^{\pm\bast}$.  Let $L$ be the submonoid of $\widetilde{N}$ generated by $\mc{X}^{\pm1}$ and note that $L$ is free on this basis. If $u \in \widetilde{N}$ write $\Lambda(u)$ for the unique word in $L$ which is freely equal to $u$ in $F(\mc{A} \cup \{t\})$ and freely reduced as an element of $F(\mc{X})$.  For each $i \in \{1, \ldots, m\}$, define $\ol{x}_i \equiv \Lambda(x_i t^{-h(x_i)})$ and $\ol{z}_i = \Lambda(t^{h(x_i)} z_i t^{-h(x_i)})$.  Define $\sigma \equiv \prod_{i=1}^m \ol{x}_i \ol{z}_i \ol{x}_i^{-1}$ and note that $w
  \FreeEq \sigma$ in $F(\mc{A} \cup \{t\})$.

  Define a homomorphism $\Psi : L \rightarrow \fm{A}$, which commutes with the inversion involution of $L$, by mapping $t^k a t^{-k} \mapsto \Phi_k(a)$.  Let $N$ be the kernel of the homomorphism $F(\mc{A} \cup \{t\}) \rightarrow \Z$ defined by mapping $t$ to $1$ and each $a \in \mc{A}$ to $0$, and note that $N$ is free with basis the image of $\mc{X}$.  Thus $\Psi$ descends to a homomorphism $N \rightarrow F(\mc{A})$ and since $w \FreeEq \sigma$ in $N$ we have that $\Psi(w) \FreeEq \Psi(\sigma)$ in $F(\mc{A})$.  Observe that $\Psi(\sigma) \equiv \prod_{i=1}^m \Psi(\ol{x}_i) \Psi(\ol{z}_i) \Psi(\ol{x}_i)^{-1}$
  and $\Psi(w) \equiv w$ since $w$ contains no occurrence of the letter $t$.

  If $z_i \equiv a_1 \ldots a_l \in \mc{R}$ then $\ol{z}_i \equiv t^k a_1 t^{-k} \ldots t^k a_l t^{-k}$ for some $k \in \Z$ with $|k| = |h(x_i)| \leq |x_i|$.  Thus $\Psi(\ol{z}_i) \equiv \Phi_k(z_i)$ where $|k| \leq \rho(n)$.  If $z_i \equiv t a t^{-1} a_1 \ldots a_l \in \mc{S}$ then $\ol{z}_i \equiv t^{k+1} a t^{-k-1} t^k a_1 t^{-k} \ldots t^{k} a_l t^{-k}$ for some $k \in \Z$ with $|k| = |h(x_i)| \leq |x_i|$.  Thus $\Psi(\ol{z}_i) \equiv \Phi_{k+1}(a) \Phi_k(w_{a})^{-1}$ where $\min\{|k+1|, |k|\} \leq |k| \leq \rho(n)$.  In either case we have that $\Psi(\ol{z}_i) \in \ol{\mc{R}} \cup \ol{\mc{S}}$ and $\|\Psi(\ol{z}_i)\| \leq \rho(n)$.  Thus $(\Psi(\ol{x}_i), \Psi(\ol{z}_i))_{i=1}^m$ is a $\mc{P}_K^{\infty}$-expression for $w$ and, since $w$ was arbitrary, we see that $\mc{P}_K^\infty$ presents $K$ and that $(\alpha, \rho)$ is an area-penetration pair for $\mc{P}_K^{\infty}$.
\end{proof}

\section{Amalgamated products}

In this section we present a method for giving lower bounds on the Dehn functions of amalgamated products.  Specifically we will be concerned with finitely presented amalgamated products $\Gamma = G_1 \ast_H G_2$ of finitely generated groups $G_1$ and $G_2$ over a finitely generated subgroup $H$ which is proper in each $G_i$.

Suppose each $G_i$ is presented by $\langle\mathcal{A}_i \, | \, \mathcal{R}_i \rangle$, with $\mathcal{A}_i$ finite.  Note that we are at liberty to choose the $\mc{A}_i$ so as each $a \in \mc{A}_i$ represents an element of $G_i \smallsetminus H$.  Indeed, since $H$ is proper in $G_i$, there exists some $a' \in \mc{A}_i$ representing an element of $G_i \smallsetminus H$ and we can replace each other element $a \in \mc{A}_i$ by $a' a$ if necessary.  Let $\mathcal{B}$ be a finite generating set for $H$ and for each $b \in \mathcal{B}$ choose words $u_b \in \mathcal{A}_1^{\pm\bast}$ and $v_b \in \mathcal{A}_2^{\pm\bast}$ which equal $b$ in $\Gamma$. Define $\mc{E} \subseteq (\mc{A}_1 \cup \mc{A}_2 \cup \mc{B})^{\pm\bast}$ to be the finite collection of words $\{bu_b^{-1}, bv_b^{-1} \, : \, b \in \mc{B}\}$. Then, since $\Gamma$ is finitely presented, there exist finite subsets $\mathcal{R}_1' \subseteq \mathcal{R}_1$ and $\mathcal{R}_2' \subseteq \mathcal{R}_2$ such that $\Gamma$ is finitely presented by $$\mathcal{P} = \langle \mathcal{A}_1, \mathcal{A}_2, \mathcal{B} \, | \, \mathcal{R}_1', \mathcal{R}_2', \mc{E} \rangle.
$$

\begin{thm} \label{thm13}
Let $w \in \mc{A}_1^{\pm\bast}$ be a word representing an element $h
\in H$ and let $u \in \mathcal{A}_1^{\pm\bast}$ and $v \in
\mathcal{A}_2^{\pm\bast}$ be words representing elements $\alpha \in
G_1 \smallsetminus H$ and $\beta \in G_2 \smallsetminus H$ respectively.  If
$[\alpha, h] = [\beta , h] = 1$ then
$$\Area_\mathcal{P}([w, (uv)^n]) \geq 2n \, {\rm d}_\mathcal{B}(1, h)$$ where ${\rm d}_\mathcal{B}$ is
the word metric on $H$ associated to the generating set
$\mathcal{B}$.
\end{thm}

\begin{proof}
Let $\Delta$ be a $\mc{P}$-van Kampen diagram for the null-homotopic
word $[w, (uv)^n]$.  For each $i = 1 , 2, \ldots, n$ define $p_i$ to
be the vertex in $\partial \Delta$ such that the anticlockwise path
in $\partial \Delta$ from the basepoint around to $p_i$ is labelled
by the word $w(uv)^{i-1}u$.  Similarly define $q_i$ to be the vertex
in $\partial \Delta$ such that the anticlockwise path in $\partial
\Delta$ from the basepoint around to $q_i$ is labelled by the word
$w (uv)^n w^{-1} (uv)^{i-n} v^{-1}$.  We will show that for each $i$
there is a $\mathcal{B}$-path (i.e. an edge path in $\Delta$
labelled by a word in the letters $\mathcal{B}$) from $p_i$ to
$q_i$.

\begin{figure}[htbp]
  \psfrag{u}{$u$}
  \psfrag{v}{$v$}
  \psfrag{p1}{$p_1$}
  \psfrag{p2}{$p_2$}
  \psfrag{pn}{$p_n$}
  \psfrag{q1}{$q_1$}
  \psfrag{q2}{$q_2$}
  \psfrag{qn}{$q_n$}
  \psfrag{w}{$w$}
  \centering \includegraphics{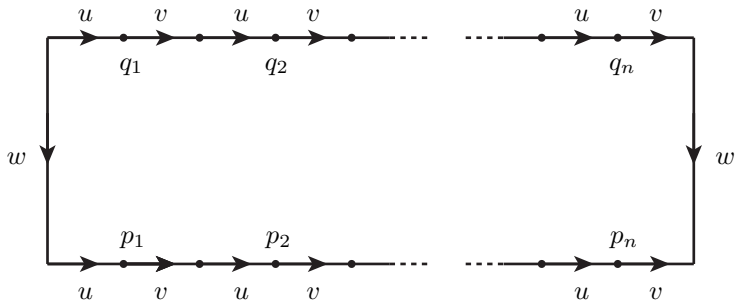}
  \caption{The van Kampen diagram $\Delta$} \label{fig12}
\end{figure}

We assume that the reader is familiar with Bass-Serre theory, as
exposited in \cite{Serre1}.  Let $T$ be the Bass-Serre tree
associated to the splitting $G_1 \ast_H G_2$. This consists of an
edge $gH$ for each coset $\Gamma / H$ and a vertex $gG_i$ for each
coset $\Gamma / G_i$.  The edge $gH$ has initial vertex $gG_1$ and
terminal vertex $gG_2$.  We will construct a continuous (but
non-combinatorial) map $\Delta \rightarrow T$ as the composition of
the natural map $\Delta \rightarrow Cay^2(\mc{P})$ with the map $f :
Cay^2(\mc{P}) \rightarrow T$ defined below.

There is a natural left action of $\Gamma$ on each of
$Cay^2(\mc{P})$ and $T$ and we construct $f$ to be equivariant with
respect to this as follows.  Let $m$ be the midpoint of the edge $H$
of $T$ and define $f$ to map the vertex $g \in Cay^2(\mc{P})$ to the
point $g \cdot m$, the midpoint of the edge $gH$.  Define $f$ to map
the edge of $Cay^2(\mc{P})$ labelled $a \in \mathcal{A}_i$ joining
vertices $g$ and $ga$ to the geodesic segment joining $g \cdot m$ to
$ga \cdot m$. Since $a \not\in H$ this segment is an embedded arc of
length $1$ whose midpoint is the vertex $gG_i$.  Define $f$ to
collapse the edge in $Cay^2(\mc{P})$ labelled $b \in \mathcal{B}$
joining vertices $g$ and $gb$ to the point $g \cdot m = gb \cdot m$.
This is well defined since $gH = gbH$. This completes the definition
of $f$ on the $1$-skeleton of $\Delta$; we now extend $f$ over the
$2$-skeleton.

Let $c$ be a $2$-cell in $Cay^2(\mc{P})$ and let $g$ be some vertex
in its boundary.  Assume that $c$ is metrised so as to be convex and
let $l$ be some point in its interior.  The form of the relations in
$\mc{P}$ ensures that the boundary label of $c$ is a word in the
letters $\mathcal{A}_i \cup \mathcal{B}$ for some $i$ and so every
vertex in $\partial c$ is labelled $gg'$ for some $g' \in G_i$. Thus
$f$ as so far defined maps $\partial c$ into the ball of radius
$1/2$ centred on the vertex $g G_i$; we extend $f$ to the interior
of $c$ by defining it to map the geodesic segment $[l, p]$, where $p
\in \partial c$, to the geodesic segment $[gG_i, f(p)]$.  This is
independent of the vertex $g \in \partial c$ chosen and makes $f$
continuous since geodesics in a tree vary continuously with their
endpoints.  We now define $\bar{f} : \Delta \rightarrow T$ to be the
map given by composing $f$ with the label-preserving map $\Delta
\rightarrow Cay^2(\mc{P})$ which sends the basepoint of $\Delta$ to
the vertex $1 \in Cay^2(\mc{P})$.

Since $w$ commutes with $u$ and $v$ we have that $\bar{f}(p_i) =
\mbox{$w(uv)^{i-1}u \cdot m$} = \mbox{$(uv)^{i-1} u \cdot m$} =
\bar{f} (q_i)$; define $S$ to be the preimage under $\bar f$ of this
point. By construction, the image of the interior of each $2$-cell
in $\Delta$ and the image of the interior of each
$\mathcal{A}_i$-edge is disjoint from $\bar f(p_i)$.  Thus $S$
consists of vertices and $\mathcal{B}$-edges and so finding a
$\mathcal{B}$-path from $p_i$ to $q_i$ reduces to finding a path in
$S$ connecting these vertices. Let $s_i$ and $t_i$ be the vertices
of $\partial \Delta$ immediately preceding and succeeding $p_i$ in
the boundary cycle.  Unless $h=1$, in which case the theorem is
trivial, the form of the word $[w, (uv)^n]$, together with the
normal form theorem for amalgamated products, implies that all the
vertices $p_i$, $s_i$ and $t_i$ lie in the boundary of the same disc
component $D$ of $\Delta$. Furthermore since $u$ and $v$ are words
in the letters $\mathcal{A}_1$ and $\mathcal{A}_2$ respectively the
points $f(s_i)$ and $f(t_i)$ are separated in $T$ by $f(p_i)$. Thus
$s_i$ and $t_i$ are separated in $D$ by $S$ and so there exists an
edge path $\gamma_i$ in $S$ from $p_i$ to some other vertex $r_i \in
\partial D$. Since $\gamma_i$ is a $\mathcal{B}$-path it follows
that the word labelling the sub-arc of the boundary cycle of
$\Delta$ from $p_i$ to $r_i$ represents an element of $H$, and, by
considering subwords of $[w, (uv)^n]$, we see that the only
possibility is that $r_i = q_i$. Thus for each $i = 1, \ldots, n$
the path $\gamma_i$ gives the required $\mathcal{B}$-path connecting
$p_i$ to $q_i$.  We choose each $\gamma_i$ to contain no repeated
edges.

For $i \neq j$ the two paths $\gamma_i$ and $\gamma_j$ are disjoint
since if they intersected there would be a $\mathcal{B}$-path
joining $p_i$ to $p_j$ and thus the word labelling the subarc of the
boundary cycle from $p_i$ to $p_j$ would represent an element of
$H$. Observe that no two edges in any of the paths $\gamma_1, \ldots
, \gamma_n$ lie in the boundary of the same $2$-cell in $\Delta$
since each relation in $\mathcal{P}$ contains at most one occurrence
of a letter in $\mathcal{B}$. Because the word labelling $\partial
\Delta$ contains no occurrences of a letter in $\mathcal{B}$ the
interior of each edge of a path $\gamma_i$ lies in the interior of
$\Delta$ and thus in the boundary of two distinct $2$-cells.  Since
each path $\gamma_i$ contains no repeated edges we therefore obtain
the bound $\Area(\Delta) \geq \sum_{i=1}^n 2|\gamma_i|$. But the
word labelling each $\gamma_i$ is equal to $h$ in $\Gamma$ and so
the length of $\gamma_i$ is at least ${\rm d}_\mathcal{B}(1, h)$
whence we obtain the required inequality.
\end{proof}

\section{Fibre products}

\begin{defn}
  Given a homomorphism $p : \Gamma \rightarrow Q$, the \emph{(untwisted) fibre product} of $p$ is defined to be the subgroup $\{(\gamma_1, \gamma_2) : p(\gamma_1) = p(\gamma_2) \} \leq \Gamma \times \Gamma$.
\end{defn}

Recall the following result of Baumslag, Bridson, Miller and Short.

\begin{thm}[The 1-2-3 Theorem \cite{baum00}] \label{thm5}
  Let $1 \rightarrow N \rightarrow \Gamma \xrightarrow{p} Q \rightarrow 1$ be a short exact sequence of groups.  Suppose that $N$ is finitely generated, $\Gamma$ is finitely presented and $Q$ is of type $\rm{F}_3$.  Then the fibre product of $p$ is finitely presented.
\end{thm}

\begin{defn}
  Given a pair of homomorphisms $p_i: \Gamma_i \rightarrow Q$, $i=1,2$, the \emph{(twisted) fibre product} of $p_1$ and $p_2$ is defined to be the subgroup $\{(\gamma_1, \gamma_2) : p_1(\gamma_1) = p_2(\gamma_2) \} \leq \Gamma_1 \times \Gamma_2$.
\end{defn}

In this section we prove a generalisation of the 1-2-3 theorem which covers twisted fibre products.

\begin{thm} \label{thm4}
  For each $i=1,2$, let $1 \rightarrow N_i \rightarrow \Gamma_i \xrightarrow{p_i} Q \rightarrow 1$ be a short exact sequence of groups.  Suppose that $N_1$ is finitely generated, $\Gamma_1$ and $\Gamma_2$ are finitely presented, and $Q$ is of type $\rm{F}_3$.  Then the fibre product of $p_1$ and $p_2$ is finitely presented.
\end{thm}

Note that we do not need to make any assumptions about $N_2$.  The proof of Theorem~\ref{thm4} given below closely follows the proof of Theorem~\ref{thm5} given in \cite{baum00}.  We will require the following lemma.

\begin{lem} \label{lem9}
  For each $i=1,2$, let $p_i : \Gamma_i \rightarrow Q$ be a surjective homomorphism.  Suppose that $\Gamma_1$ and $\Gamma_2$ are finitely generated and that $Q$ is finitely presented.  Then the fibre product $P$ of $p_1$ and $p_2$ is finitely generated.  If $\alpha$ is an isoperimetric function for some finite presentation of $Q$ then the distortion function $\Delta$ of $P$ in $\Gamma_1 \times \Gamma_2$ satisfies $\Delta \preceq \alpha$.

  More specifically, let $\mc{X}_1$ be a finite generating set for $\Gamma_1$ and let $\mc{X}$ be the image of $\mc{X}_1$ in $Q$.  Let $\mc{X}_2$ be a choice of lifts of the elements of $\mc{X}$ under $p_2$ and let $\mc{A} \subseteq \ker p_2$ be a finite collection of elements such that $\Gamma_2$ is generated by $\mc{A} \cup \mc{X}_2$.  Let $\langle \mc{X} \, | \, \mc{R} \rangle$ be a finite presentation for $Q$.
  Then $P$ is generated by the union of the following sets of elements: \begin{align*}
    \omc{X} &= \{ (x_1, x_2) : x_i \in \mc{X}_i, p_1(x_1) = p_2(x_2) \}; \\
    \omc{A} &= \{ (1, a) : a \in \mc{A} \}; \\
    \omc{R} &= \{ (r(\mc{X}_1),1) : r(\mc{X}) \in \mc{R} \}.
  \end{align*}
\end{lem}

\begin{rem}
  Note that the bound on $\Delta$ is only defined up to $\simeq$-equivalence, not the stronger $\approx$-equivalence usually used with distortion functions.
\end{rem}

\begin{proof}[Proof of Lemma~\ref{lem9}]
  Fix compatible orderings on $\mc{X}$, $\mc{X}_1$, $\mc{X}_2$ and $\ol{\mc{X}}$; and on $\mc{A}$ and $\ol{\mc{A}}$.  By Lemma~\ref{lem27}, the Dehn function $\delta$ of $\langle \mc{X} \, | \, \mc{R} \rangle$ satisfies $\delta \preceq \alpha$.

  Let $w = w(\mc{X}_1, \mc{X}_2, \mc{A})$ be a word representing an element $\gamma$ of $P$.  Then $w \stackrel{P}{=} w(\mc{X}_1, \emptyset, \emptyset) w(\emptyset, \mc{X}_2, \mc{A}) \FreeEq w(\mc{X}_1, \emptyset, \emptyset) w^{-1}(\emptyset, \mc{X}_1, \emptyset) w(\emptyset, \mc{X}_1, \emptyset) w(\emptyset, \mc{X}_2, \mc{A}) \stackrel{P}{=} w(\mc{X}_1, \emptyset, \emptyset) w^{-1}(\emptyset, \mc{X}_1, \emptyset) w(\emptyset, \ol{\mc{X}}, \ol{\mc{A}})$.  Define $w_1(\mc{X}_1) \equiv w(\mc{X}_1, \emptyset, \emptyset) w^{-1}(\emptyset, \mc{X}_1, \emptyset)$ and note that $|w_1| \leq |w|$.  Furthermore $p_1(w_1(\mc{X}_1))$ is trivial in $Q$ and so $w_1(\mc{X})$ is null-homotopic.  Define $L = \max \{ |r| : r \in \mc{R} \}$.  By Lemma~\ref{lem28}, there exist words $r_1, \ldots r_n \in \mc{R}^{\pm1}$ and words $x_0, \ldots, x_n \in \fm{X}$ with $n \leq \delta(|w_1|)$ and $\sum |x_i| \leq |w_1| + 2Ln$ so that $w_1(\mc{X}) \FreeEq x_0 r_1 x_1 \ldots r_n x_n$ and the word $x_0 \ldots x_n \FreeEq \emptyset$.  Thus $w_1(\mc{X}_1) \FreeEq x_0(\mc{X}_1) r_1(\mc{X}_1) \ldots r_n(\mc{X}_1) x_n(\mc{X}_1) \stackrel{P}{=} x_0(\omc{X}) (r_1(\mc{X}_1), 1) \ldots (r_n(\mc{X}_1), 1) x_n(\omc{X})$ and so $\gamma$ is represented by a word in the letters $\ol{\mc{X}}$, $\ol{\mc{A}}$ and $\ol{\mc{R}}$ of length at most $(2L+1)\delta(|w|) + 2|w|$.  Thus $\Delta \preceq \delta \preceq \alpha$.
\end{proof}

\begin{proof}[Proof of Theorem~\ref{thm4}]
  Let $\mc{X}_1$ be a finite ordered generating set for $\Gamma_1$ and let $\mc{X}$ be the image of $\mc{X}_1$ in $Q$.  Then there is an induced ordering on $\mc{X}$ and $\mc{X}$ generates $Q$.  Let $\mc{A}_1$ be a finite ordered generating set for $N_1$.  For each $a \in \mc{A}_1$, $x \in \mc{X}_1$ and $\epsilon \in \{\pm1\}$, choose a word $w_{ax\epsilon} \in \mc{A}_1^{\pm\bast}$ such that $x^\epsilon a x^{-\epsilon} = w_{ax\epsilon}$ in $\Gamma_1$.  Let $\langle \mc{X} \, | \, \mc{R} \rangle$ be a finite presentation for $Q$ and for each $r = r(\mc{X}) \in \mc{R}$ choose a word $w_r \in \mc{A}_1^{\pm\bast}$ such that $r(\mc{X}_1) = w_r$ in $\Gamma_1$.  Define $$\mc{R}_1 = \big\{x^\epsilon a x^{-\epsilon}w_{ax\epsilon}^{-1} \, : \, a \in \mc{A}_1, x \in \mc{X}_1, \epsilon \in \{\pm1\}\big\}$$ and $$\mc{R}_2 = \big\{r(\mc{X}_1)w_r^{-1} \, : \, r \in \mc{R} \big\}.$$  If $w = w(\mc{A}_1, \mc{X}_1)$ is null-homotopic in $\Gamma_1$ then, modulo relators in $\mc{R}_1$, $w$ is equal to a word of the form $u(\mc{A}_1)v(\mc{X}_1)$.  The word $v(\mc{X})$ is null-homotopic in $Q$ and hence there is a free equality $v(\mc{X}_1) = \prod \rho_i(\mc{X}_1) r_i(\mc{X}_1) \rho_i(\mc{X}_1)^{-1}$ for some $r_i = r_i(\mc{X}) \in \mc{R}$ and some words $\rho_i$.  Thus, modulo relators in $\mc{R}_1$ and $\mc{R}_2$, $w$ is equal to a word in the letters $\mc{A}_1$.  It follows that there exists a finite collection of relations $\mc{R}_3 \subset \mc{A}_1^{\pm\bast}$ such that $\Gamma_1$ is presented by $\langle \mc{A}_1, \mc{X}_1 \, | \, \mc{R}_1, \mc{R}_2, \mc{R}_3 \rangle$.

  Let the finite set $\mc{X}_2 \subset \Gamma_2$ be a choice of lifts of the elements of $\mc{X}$ under $p_2$ ordered compatibly with $\mc{X}$.  Then there exists a finite ordered collection of elements $\mc{A}_2 \subset N_2$ so that $\mc{X}_2 \cup \mc{A}_2$ generates $\Gamma_2$.  Note that $\mc{A}_2$ may not generate $N_2$.  Let $\langle \mc{A}_2, \mc{X}_2 \, | \, \mc{R}_4\rangle$ be a finite presentation for $\Gamma_2$.

  By the argument in the proof of Lemma~\ref{lem9}, the fibre product $P$ of $p_1$ and $p_2$ is generated by the union of the following sets of elements: \begin{align*}
    \omc{X} &= \{(x_1, x_2) \, : \, x_i \in \mc{X}_i, p_1(x_1) =
    p_2(x_2) \};\\
    \omc{A}_1 &= \{(a,1) \, : \, a \in \mc{A}_1\};\\
    \omc{A}_2 &= \{(1,a) \, : \, a \in \mc{A}_2\}.
  \end{align*}  Order the elements of $\omc{X}$, $\omc{A}_1$ and $\omc{A}_2$ compatibly with the $\mc{X}_i$, $\mc{A}_1$ and $\mc{A}_2$ respectively.   We now define some relations which hold amongst these generators: \begin{align*}
    \mc{S}_1 &= \big\{ [\bar{a}_1, \bar{a}_2] \, : \, \bar{a}_i \in \omc{A}_i\big\}\\
    \mc{S}_2 &= \big\{ r(\omc{X}, \omc{A}_1) \, : \, r=r(\mc{X}_1, \mc{A}_1) \in \mc{R}_1 \big\}\\
    \mc{S}_3 &= \big\{ r(\omc{A}_1) \, : \, r = r(\mc{A}_1) \in \mc{R}_3 \big\}\\
    \mc{S}_4 &= \big\{ [r(\omc{X}, \omc{A}_1), \bar{a}] \, : \, r=r(\mc{X}_1, \mc{A}_1) \in \mc{R}_2, \bar{a} \in \omc{A}_1 \big\}
  \end{align*}  For each $r = r(\mc{X}_2, \mc{A}_2) \in \mc{R}_4$, choose a word $w_r \in \omc{A}_1^{\pm\bast}$ so that $r(\omc{X}, \omc{A}_2) = w_r$ in $\Gamma_1 \times \Gamma_2$.  Then we can define the set of relations $$\mc{S}_5 = \big\{ r(\omc{X}, \omc{A}_2) w_r(\omc{A}_1)^{-1} \, : \, r(\mc{X}_2, \mc{A}_2) \in \mc{R}_4 \big\}.$$ Let $\Sigma$ be a finite generating set of Peiffer sequences for $\pi_2(Q)$ as a $Q$-module.  Each $\sigma \in \Sigma$ is a sequence $(u_1 r_1 u_1^{-1}, \ldots, u_n r_n u_n^{-1})$ where each $r_i = r_i(\mc{X}) \in \mc{R}$, each $u_i = u_i(\mc{X})$ is a word in $\mc{X}^{\pm\bast}$ and the word $$\zeta_\sigma(\mc{X}) = \prod_i u_i(\mc{X}) r_i(\mc{X}) u_i(\mc{X})^{-1}$$ is freely equal to the empty word.  Observe that, modulo relations in $\mc{R}_2$, the word $\zeta_\sigma(\mc{X}_1)$ is equal to $$\prod_i u_i(\mc{X}_1) w_{r_i}(\mc{A}_1) u_i(\mc{X}_1)^{-1}$$ and this is equal, modulo relations in $\mc{R}_1$, to a word $Z_\sigma = Z_\sigma(\mc{A}_1)$.  We define $$\mc{S}_6 = \big\{ Z_\sigma(\omc{A}_1) \, : \, \sigma \in \Sigma \big\}.$$

  We claim that $P$ is presented by $\langle \omc{X}, \omc{A}_1, \omc{A}_2 \, | \, \mc{S}_1, \mc{S}_2, \mc{S}_3, \mc{S}_4, \mc{S}_5, \mc{S}_6 \rangle$.  Indeed suppose that $w = w(\omc{X}, \omc{A}_1, \omc{A}_2)$ is null-homotopic.  Then the relations in $\mc{S}_1$ and $\mc{S}_2$ are sufficient to convert $w$ to a word $w_1(\omc{A}_1) w_2(\omc{X}, \omc{A}_2)$.  Projecting onto the factor $\Gamma_2$ demonstrates that the word $w_2(\mc{X}_2, \mc{A}_2)$ is null-homotopic.  There thus exists a free equality $$w_2(\mc{X}_2, \mc{A}_2) \FreeEq \prod_i u_i(\mc{X}_2, \mc{A}_2) r_i(\mc{X}_2, \mc{A}_2) u_i(\mc{X}_2, \mc{A}_2)^{-1}$$ for some words $u_i$ and some relations $r_i \in \mc{R}_4$ and hence a free equality $$w_2(\omc{X}, \omc{A}_2) \FreeEq \prod_i u_i(\omc{X}, \omc{A}_2) r_i(\omc{X}, \omc{A}_2) u_i(\omc{X}, \omc{A}_2)^{-1}.$$  The relations in $\mc{S}_5$ and in $\mc{S}_1$ and $\mc{S}_2$ are sufficient to convert $w_2(\omc{X}, \omc{A}_2)$ to the word $$\prod_i u_i(\omc{X}, \omc{A}_2) w_{r_i}(\omc{A}_1)^{-1} u_i(\omc{X}, \omc{A}_2)^{-1}$$ and thence to some word in the letters $\omc{A}_1$.  The word $w(\omc{X}, \omc{A}_1, \omc{A}_2)$ can thus be converted to a word $w' = w'(\omc{A}_1)$.  We now recall the following result of Baumslag, Bridson, Miller and Short:

  \begin{lem}[\cite{baum00}] \label{lem10}
    A word $v=v(\mc{A}_1)$ is null-homotopic in $\Gamma_1$ if and only if it is freely equal in $F(\mc{A}_1 \cup \mc{X}_1)$ to a product of conjugates of the following relations:
    \begin{itemize}
      \item $\mc{R}_1$
      \item $\mc{R}_3$
      \item $\big\{Z_\sigma(\mc{A}_1) \, : \, \sigma \in \Sigma\big\}$
      \item $\big\{ [r(\mc{A}_1, \mc{X}_1), a] \, : \, r \in \mc{R}_2, a \in \mc{A}_1 \big\}$
    \end{itemize}
  \end{lem}

  Projecting $\Gamma_1 \times \Gamma_2$ onto the first factor demonstrates that $w'(\mc{A}_1)$ is null-homotopic in $\Gamma_1$ and hence there is an equality $$w'(\mc{A}_1) \FreeEq \prod_i u_i(\mc{A}_1, \mc{X}_1) s_i(\mc{A}_1, \mc{X}_1) u_i(\mc{A}_1, \mc{X}_1)^{-1}$$ for some words $u_i$ and some relations $s_i$ from the sets given in Lemma~\ref{lem10}.  It follows that there is an equality $$w'(\omc{A}_1) \FreeEq \prod_i u_i(\omc{A}_1, \omc{X}) s_i(\omc{A}_1, \omc{X}) u_i(\omc{A}_1, \omc{X})^{-1}$$ where the $s_i =s_i(\omc{A}_1, \omc{X})$ are relations in $\mc{S}_2 \cup \mc{S}_3 \cup \mc{S}_4 \cup \mc{S}_6$. This completes the proof of the claim.
\end{proof}

\section{Close fillings} \label{sec8}

Let $H$ be a subgroup of a group $G$.  In this section we establish criteria for $H$ to be finitely generated or to be finitely presented.  The utility of these  criteria is that they are phrased entirely in terms of properties of generating sets (respectively presentations) for $G$, and so one avoids having to explicitly determine a generating set (respectively a presentation) for $H$.  In the language of course geometry, the criteria amount to showing that $H$ is coarsely connected (respectively coarsely simply connected) in $G$.

Suppose that $G$ is finitely generated, and consider the vertices in the Cayley graph of $G$ that represent elements of $H$.  We will show that $H$ is finitely generated if this set is coarsely connected.   More explicitly, the criterion amounts to showing that every element of $H$ can be represented by a word in the generators of $G$ that, considered as a path in the Cayley graph of $G$, lies uniformly close to $H$.  By considering the length of such words, one obtains a bound on the distortion of $H$ in $G$.

If $G$ is finitely presented then an analogous criterion will establish that $H$ is itself finitely presented: this amounts to showing that an embedding of the Cayley graph of $H$ in the Cayley complex of $G$ is coarsely simply connected.  In the language of van Kampen diagrams one demonstrates that every null-homotopic edge loop in the Cayley $2$-complex of $G$ which lies close to $H$ can be filled by a diagram which lies close to $H$.  We translate this notion into the language of $\mc{P}$-expressions.  By considering the areas of such expressions one obtains an isoperimetric function for $H$.

\begin{defn}
  Let $\mc{X}$ be a generating set for $G$.

  Given $g \in G$ define $$d_\mc{X}(g, H) = \min_{h \in H} d_\mc{X}(g,h),$$ where $d_\mc{X}$ is the word metric on $G$
  associated to $\mc{X}$.  Define the \emph{departure} from $H$ of a word $w \in
  \fm{X}$ by $$\Dep_\mc{X}(w, H) = \max_{0 \leq i \leq |w|} d_\mc{X} (w[i], H).$$
\end{defn}

\begin{prop} \label{prop3}
  Let $\mc{X}$ be a finite generating set for the group $G$.  Suppose that there exists a constant $K \geq 0$ such that for all $h \in H$ there exists a word $w_h \in \fm{X}$ representing $h$ in $G$ with $\Dep_\mc{X}(w_h, H) \leq K$.  Then there exists a finite generating set $\mc{Y}$ for $H$ and the distortion function $\Delta$ of $H$ in $G$ with respect to $\mc{Y}$ and $\mc{X}$ satisfies $$\Delta(l) \leq \max \{ |w_h| \, : \, d_\mc{X}(1, h) \leq l \}.$$
\end{prop}

\begin{proof}
  For each $g \in G$, choose an element $\gamma_g \in G$ such that $g \gamma_g^{-1} \in H$ and $d_\mc{X}(1, \gamma_g) = d_\mc{X}(g, H)$.  Define a function $\Pi : G \times \mc{X}^{\pm1} \rightarrow H$ by $\Pi(g, x) = \gamma_g x \gamma_{gx}^{-1}$.  Define a function $\Psi : \fm{X} \rightarrow H^{\pm\bast}$ by $$\Psi(x_1 \ldots x_n) = \Pi(1, x_1) \, \Pi(x_1, x_2) \, \Pi(x_1 x_2, x_3) \, \ldots \, \Pi(x_1 \ldots x_{n-1}, x_n)$$ and note that if $w \in \mc{X}^{\pm\bast}$ represents an element of $H$ then $\Psi(w) = w$ in $H$.

  Given $r \in \N$, define $N_r = \{g \in G \, : \, d_\mc{X}(g, H) \leq r\}$.  Define $\mc{Y} = \Pi(N_K \times \mc{X}^{\pm1}) \subseteq H$ and note that $\mc{Y}$ is finite since it is contained in the finite set $\{h \in H \, : \, d_\mc{X}(1, h) \leq 2K + 2 \}$.  Observe that, for every $h \in H$, the word $\Psi(w_h) \in \fm{Y}$ represents $h$ and hence $\mc{Y}$ generates $H$.  Furthermore $d_\mc{Y}(1, h) \leq |\Psi(w_h)| = |w_h|$ so $\Delta$ satisfies the given inequality.
\end{proof}

\begin{defn}
  Let $\mc{P}$ be a presentation of the group $G$.  Define the \emph{departure} from $H$ of a $\mc{P}$-expression $\mc{E} = (x_i, r_i)_{i=1}^m$ to be $$\Dep_\mc{X}(\mc{E}, H)= \max_{1 \leq i \leq m} \Dep_\mc{X}(x_i, H).$$
\end{defn}

\begin{prop} \label{prop4}
  Let $\mc{P} = \langle \mc{X} \, | \, \mc{R} \rangle$ be a finite
  presentation of the group $G$ and let $H \leq G$ be a finitely
  generated subgroup with finite generating set $\mc{Y}$. \begin{enumerate}
    \item Suppose that there exists a function $K : \N \rightarrow \N$ such that, for each null-homotopic word $w \in \fm{X}$, there exists a $\mc{P}$-expression $\mc{E}_w$ for $w$ with $\Dep_\mc{X}(\mc{E}_w, H) \leq K(\Dep_\mc{X}(w, H))$.  Then there exists a finite set of words $\mc{S} \subseteq \fm{Y}$ so that $H$ is presented by $\mc{Q} = \langle \mc{Y} \, | \, \mc{S} \rangle$.

    \item Suppose, in addition, that there exists a function $\alpha : \N \rightarrow \N$ so that $\Area(\mc{E}_w) \leq \alpha(|w|)$ for each $w$.  Then $\alpha$ is an isoperimetric function for $H$.

    \item Suppose, in addition, that there exists a function $\rho : \N \rightarrow \N$ so that $\Rad(\mc{E}_w) \leq \rho(|w|)$ for each $w$.  Then there exist functions $\alpha', \rho' : \N \rightarrow \N$ with $\alpha' \simeq \alpha$ and $\rho' \simeq \rho$ so that $(\alpha', \rho')$ is an area-radius pair for $\mc{Q}$.
  \end{enumerate}
\end{prop}

\begin{proof}
  For each $y \in \mc{Y}$ choose a word $u_y \in \fm{X}$ with $u_y = y$ in $G$.  Define $L = \max \{\Dep_\mc{X}(u_y, H) \, : \, y \in \mc{Y} \}$.

  For each $g \in G$, choose an element $\gamma_g \in G$ with $g \gamma_g^{-1} \in H$ and $d_\mc{X}(1, \gamma_g) = d_\mc{X}(g, H)$.  Choose $\xi$ to be a function $H \times H \rightarrow \mc{Y}^{\pm \bast}$ such that $\xi(h_1, h_2)$ represents $h_1^{-1} h_2$ in $H$, $|\xi(h_1, h_2)| = d_\mc{Y}(h_1, h_2)$ and $\xi(h_2, h_1) = \xi(h_1, h_2)^{-1}$.  Geometrically $\xi$ is a choice of a preferred edge path connecting each pair of vertices in the Cayley graph of $H$ that is compatible with reversing orientation.  Define a function $\Omega : G \times \mathcal{X}^{\pm1} \rightarrow \mc{Y}^{\pm \bast}$ by $\Omega(g, x) = \xi(g \gamma_g^{-1}, gx \gamma_{gx}^{-1})$.  Then $\Omega(g, x)$ represents the element $\gamma_g x \gamma_{gx}^{-1}$ of $H$ and $\Omega(gx, x^{-1}) = \Omega(g, x)^{-1}$.  Extend $\Omega$ to a function $G \times \fm{X} \rightarrow \fm{Y}$ by setting $$\Omega(g, x_1 \ldots x_n) = \Omega(g, x_1) \, \Omega(gx_1, x_2) \, \ldots \, \Omega(g x_1 \ldots x_{n-1}, x_n).$$  Geometrically, we can think of $\Omega$ as a map from edge paths in the Cayley graph of $G$ to edge paths in the Cayley graph of $H$ which is compatible with reversing the orientation of paths.

  Note that $\Omega(g, w) = \gamma_g w \gamma_{gw}^{-1}$ in $G$ for any $w \in \fm{X}$.  Thus if $w$ is null-homotopic then so is $\Omega(g, w)$.  Given $r \in \N$, define $N_r = \{ g \in G \, : \, d_\mc{X}(g, H) \leq r \}$ and define $\mc{S}_1$ to be the collection of null-homotopic words $\Omega(N_{K(L)} \times \mc{R}^{\pm1}) \subseteq \fm{Y}$.  We will show that $\mc{S}_1$ is finite by demonstrating that there is a uniform bound on the length of all words in this set.  Indeed, note that, for all $g \in G$ and $x \in \mc{X}^{\pm1}$, one has $d_\mc{X}(g \gamma_g^{-1}, gx \gamma_{gx}^{-1}) = d_\mc{X}(1, \gamma_g x \gamma_{gx}^{-1}) \leq d_\mc{X}(g, H) + 1 + d_\mc{X}(gx, H) \leq 2d_\mc{X}(g, H) + 2$.  Thus $|\Omega(g, x)| = d_\mc{Y}(g \gamma_g^{-1}, g x \gamma_{gx}^{-1}) \leq \Delta^G_H(2 d_\mc{X}(g, H) + 2)$, where $\Delta^G_H$ is the distortion function of $H$ in $G$ with respect to the generating sets $\mc{Y}$ and $\mc{X}$ respectively. It follows that, for any word $w \in \fm{X}$, one has $|\Omega(g, w)| \leq |w| \Delta^G_H(2 \max_{0 \leq i < |w|} d_\mc{X}(gw[i], H) + 2) \leq |w| \Delta^G_H(2 d_\mc{X}(g, H) + 2|w| + 2)$.  Thus $|s| \leq R \Delta^G_H(2K(L) + 2R + 2)$ for all $s \in \mc{S}_1$, where $R = \max_{r \in \mc{R}}|r|$, and hence $\mc{S}_1$ is indeed finite.

  If $w \in \fm{X}$ represents an element of $H$ then as group elements $\Omega(1, w) = \gamma_1 w \gamma_w^{-1} = w$.  Thus, for each $y \in \mc{Y}$, we have that $\Omega(1, u_y) = y$ in $H$. Define $\mc{S}_2$ to be the collection of null-homotopic words $\{y \Omega(1, u_y)^{-1} \, : \, y \in \mc{Y}\} \subseteq \fm{Y}$.  We will show that $H$ is presented by $\mc{Q} = \langle \mc{Y} \, | \, \mc{S}_1, \mc{S}_2 \rangle$.

  Let $\sigma = y_1 \ldots y_n \in \fm{Y}$ be an arbitrary null-homotopic word.  Define $\sigma'$ to be the word $u_{y_1} \ldots u_{y_n} \in \fm{X}$.  Then $\Dep_\mc{X}(\sigma', H) \leq L$ so there exists a null $\mc{P}$-expression $\mc{E} = (x_i, r_i)_{i=1}^m$ for $\sigma'$ with $\Dep_\mc{X}(\mc{E}, H) \leq K(L)$. Define $\ol{\mc{E}}$ to be the null $\mc{Q}$-expression $\big(\Omega(1, x_i), \, \Omega(x_i, r_i) \big)_{i=1}^m$.  The relationship between $\mc{E}$ and $\ol{\mc{E}}$ is represented schematically in Figure~\ref{fig2}

  \begin{figure}[htbp]
  \psfrag{a}{$x_i$}
  \psfrag{b}{$\Omega(1,x_i)$}
  \psfrag{c}{$r_i$}
  \psfrag{d}{$\Omega(x_i, r_i)$}
  \psfrag{e}{$\Omega(1, \sigma')$}
  \psfrag{f}{$\sigma'$}
  \psfrag{g}{$H$}
  \psfrag{h}{$\mc{E}$}
  \psfrag{i}{$\ol{\mc{E}}$}
  \centering \includegraphics{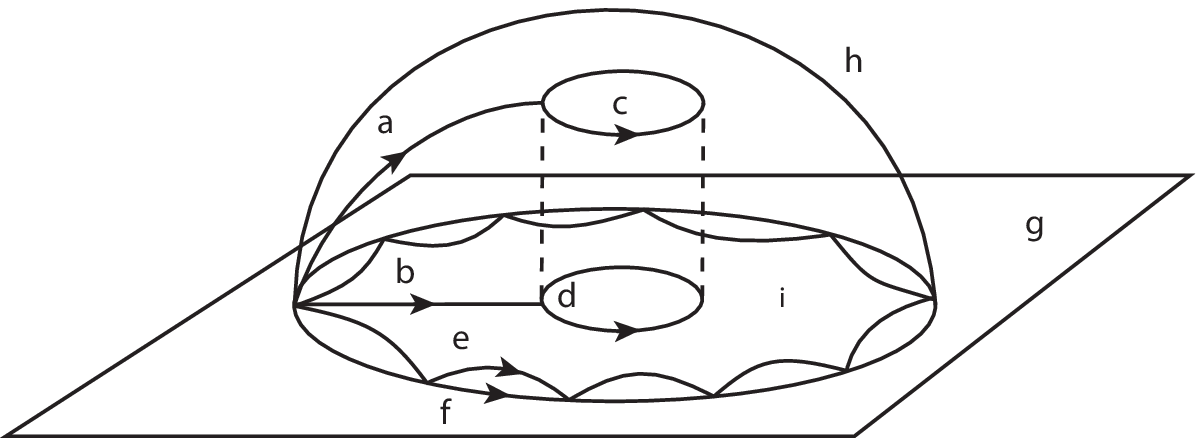}
  \caption{The relationship between $\mc{E}$ and $\ol{\mc{E}}$.} \label{fig2}
  \end{figure}

   Recall that if $g \in G$ and $x \in \mc{X}^{\pm1}$ then $\Omega(g, x)^{-1} \equiv \Omega(gx, x^{-1})$.  Thus if $w_1 \in \fm{X}$ and $w_2 \in \fm{X}$ are freely equal then for any $g \in G$ one has that $\Omega(g, w_1) \in \fm{Y}$ and $\Omega(g, w_2) \in \fm{Y}$ are freely equal.  In particular $\Omega(1, \sigma')$ is freely equal to $\Omega(1, \partial \mc{E})$.  Also \begin{align*} \Omega(1, x_i r_i x_i) &\equiv \Omega(1, x_i) \Omega(x_i, r_i) \Omega(x_i r_i, x_i^{-1}) \\ &\equiv \Omega(1, x_i) \Omega(x_i, r_i) \Omega(x_i r_i x_i^{-1}, x_i)^{-1} \\ &\equiv \Omega(1, x_i) \Omega(x_i, r_i) \Omega(1, x_i)^{-1} \end{align*} and so $\Omega(1, \partial \mc{E}) \equiv \partial \ol{\mc{E}}$.  Thus $\ol{\mc{E}}$ is a $\mc{Q}$-expression for $\Omega(1, \sigma')$.

  For each $i = 1, \ldots, n$, define $\mc{E}_i$ to be the area $1$ $\mc{Q}$-expression $(y_1 \ldots y_{i-1}, y_i \Omega(1, u_{y_i})^{-1})$.  Then $\partial(\mc{E}_n \ldots \mc{E}_1)$ is freely equal to $y_1 \ldots y_n \Omega(1, u_{y_n})^{-1} \Omega(1, u_{y_1})^{-1} \equiv \sigma \Omega(1, \sigma')^{-1}$ and so $\mc{E}_n \ldots \mc{E}_1 \ol{\mc{E}}$ is a $\mc{Q}$-expression for $\sigma$.

  Now suppose that there exists a function $\alpha$ as in assertion (2).  If we define $C = \max \{|u_y| \, : \, y \in \mc{Y}\}$ then $|\sigma'| \leq C |\sigma|$ and so we can choose $\mc{E}$ so that $\Area(\mc{E}) \leq \alpha(C|\sigma|)$.  Hence $\Area_\mc{Q}(\sigma) \leq \Area(\mc{E}_n \ldots \mc{E}_1 \ol{\mc{E}}) \leq \alpha(C |\sigma|) + |\sigma|$.  Define $\alpha'$ by $\alpha'(l) = \alpha(Cl) +l$.  Then $\alpha'$, and hence $\alpha$, is an isoperimetric function for $H$.

  If furthermore there exists a function $\rho$ as in assertion (3) then we can choose $\mc{E}$ so that additionally $\Rad(\mc{E}) \leq \rho(C|\sigma|)$.  Then $\Rad(\mc{E}_n \ldots \mc{E}_1 \ol{\mc{E}}) \leq \max \{\rho(C|\sigma|), |\sigma|-1 \} \leq \rho(C|\sigma|) + |\sigma|$.  Define $\rho'$ by $\rho'(l) = \rho(Cl) + l$.  Then $(\alpha', \rho')$ is an area-radius pair for $\mc{Q}$.
\end{proof}

\begin{proof}[Proof of Lemma~\ref{lem22}~(2)]
    Say $\mc{Q} = \langle \mc{X} \, | \, \mc{R} \rangle$ and that $H$ is finitely generated by $\mc{Y}$.  Let $\mc{C}$ be a finite set of right coset representatives for $H$ in $G$.  For each $c \in \mc{C}$, choose a word $w_c \in \fm{X}$ representing $c$ in $G$.  Define $L = \max_{c \in \mc{C}} \{ |w_c| \}$.  Then for each $g \in G$, there exists $c \in \mc{C}$ so that $gc^{-1} \in H$ and hence $d_\mc{X}(g, H) \leq L$.  Thus, for any $\mc{Q}$-expression $\mc{E}$, one has that $\Dep_\mc{X}(\mc{E}, H) \leq L$.  Proposition~\ref{prop4} therefore gives a finite collection of words $\mc{S} \in \fm{Y}$ and functions $\bar{\alpha}, \bar{\rho} : \N \rightarrow \N$ with $\bar{\alpha} \simeq \alpha$ and $\bar{\rho} \simeq \rho$ so that $H$ is presented by $\bar{\mc{P}} = \langle \mc{Y} \, | \, \mc{S} \rangle$ and $(\bar{\alpha}, \bar{\rho})$ is an area-radius pair for $\bar{\mc{P}}$.  The result then follows by Proposition~\ref{prop11}.
\end{proof}

\section{Full coabelian subdirect products} \label{sec7}

\subsection{The main theorem} \label{sec16}

\begin{defn}
   Let $H$ be a subgroup of a group $G$.  If $[G, G] \leq H$, then we say that $H$ is \emph{coabelian} in $G$.  If there exists a finite index subgroup $G' \leq G$ so that $[G', G'] \leq H$, then we say that $H$ is \emph{virtually-coabelian} in $G$.  In this situation, we define the \emph{corank} of $H$ in $G$ to be $\dim \left(\frac{G'}{G' \cap H} \otimes \Q\right)$.  Note that this is independent of the finite index subgroup $G' \leq G$ chosen.
\end{defn}

\begin{defn}
  Let $H$ be a subgroup of a direct product $D = \Gamma_1 \times \ldots \times \Gamma_n$.  If $\Gamma_i H = D$ for each $i$, then we say that $H$ is \emph{full} in $D$.  If $[D : \Gamma_iH] < \infty$ for each $H$, then we say $H$ is \emph{virtually-full} in $D$.  Note that these definitions are dependent upon a choice of a particular decomposition of $D$ as a direct product.
\end{defn}

Section~\ref{sec7} of this thesis is dedicated to proving the following result.

\begin{thm} \label{thm6}
  Let $H$ be a virtually-full, virtually-coabelian subgroup of a direct product $D = \Gamma_1 \times \ldots \times \Gamma_n$, with corank $r$. \begin{enumerate}
    \item Suppose each $\Gamma_i$ is finitely generated and $n \geq 2$.  Then $H$ is finitely generated and the distortion function $\Delta$ of $H$ in $D$ satisfies $\Delta(l) \preccurlyeq l^2$.

    \item Suppose each $\Gamma_i$ is finitely presented and $n \geq 3$.  Then $H$ is finitely presented.

    \item Suppose each $\Gamma_i$ is finitely presented and $n \geq 3$.  For each $i$, let $(\alpha_i, \rho_i)$ be an area-radius pair for some finite presentation of $\Gamma_i$.  Define $$\alpha(l) = \max ( \{l^2\} \cup \{ \alpha_i(l) \, : \, 1 \leq i \leq n\})$$ and $$\rho(l) = \max ( \{l\} \cup \{ \rho_i(l) \, : \, 1 \leq i \leq n\}).$$  Then $\rho^{2r} \alpha$ is an isoperimetric function for $H$

    \item Suppose that each $\Gamma_i$ is finitely presented and that $n \geq \max\{3, 2r\}$.  Let $\beta_1$ and $\beta_2$ be the Dehn functions of some finite presentations of $\Gamma_1 \times \ldots \times \Gamma_{n-r}$ and $\Gamma_{n-r+1} \times \ldots \times \Gamma_n$ respectively.  Then the function $\beta$ defined by $$\beta(l) = l\beta_1(l^2) + \beta_2(l)$$ is an isoperimetric function for $H$.
  \end{enumerate}
\end{thm}

Note that the finite generation of the $\Gamma_i$ ensures that $H$ has finite corank in $D$.  Furthermore, for a fixed $D$, the corank of a virtually-coabelian subgroup $H \leq D$ is bounded by the corank of $[D,D]$.  It follows that there is a uniform polynomial isoperimetric function for all virtually-full, virtually-coabelian subgroups of $D$.  Also observe that the finiteness properties of the Stallings-Bieri groups $\mathrm{SB}_1$ and $\mathrm{SB}_2$ demonstrate the necessity of the conditions $n \geq 2$ and $n \geq 3$ respectively.

\subsection{Reductions of the main theorem}

\renewcommand{\labelenumi}{(\roman{enumi})}

\begin{prop} \label{prop15}
  Theorem~\ref{thm6} is true if and only if it holds under the following additional hypotheses: \begin{enumerate}
    \item $H$ is full in $D$.

    \item $H$ is coabelian in $D$.

    \item $D/H$ is finitely generated free abelian.
  \end{enumerate}
\end{prop}

\renewcommand{\labelenumi}{(\arabic{enumi})}

Note that these stronger hypotheses hold precisely when $H$ is the kernel of a homomorphism $\theta : \Gamma_1 \times \ldots \times \Gamma_n \rightarrow \Z^r$ with the restriction of $\theta$ to each factor $\Gamma_i$ surjective.  In order to perform the reduction of Proposition~\ref{prop15} we will need the following two lemmas.

\begin{lem} \label{lem29}
  Let $H$ be a virtually-full, virtually-coabelian subgroup of the direct product $D = \Gamma_1 \times \ldots \times \Gamma_n$.  Then there exists a finite index subgroup $D' \leq D$ so that $H \cap D'$ is full and coabelian in $D' = (D' \cap \Gamma_1) \times \ldots \times (D' \cap \Gamma_n)$.
\end{lem}

\begin{proof}
  Since $H$ is virtually-coabelian in $D$, there exists a finite index subgroup $\bar{D} \leq D$ with $[\bar{D}, \bar{D}] \leq H$.  Define $\bar{H}$ to be the finite index subgroup $H \cap \bar{D} \leq H$, and, for each $i$, define $\bar{\Gamma}_i$ to be the finite index subgroup $\Gamma_i \cap \bar{D} \leq \Gamma_i$.

Since $H$ is virtually-full in $D$, $[D: \Gamma_i H] < \infty$ for each $i$.  Thus each $\bar{\Gamma}_i \bar{H}$ has finite index in $\bar{D}$, since $[\bar{D}: \bar{\Gamma}_i \bar{H}] \leq [D: \bar{\Gamma}_i \bar{H}] = [D : \Gamma_iH] [ \Gamma_iH : \bar{\Gamma}_i \bar{H} ] < \infty$.  Define $D'$ to be the finite index subgroup $\cap_{i=1}^n \bar{\Gamma}_i \bar{H} \leq D$ and, for each $i$, define $\Gamma_i' = \bar{\Gamma}_i \cap D'$.  Note that $\bar{H} \leq D'$ and hence that $\bar{H} = H \cap D'$.  For each $k$, $\Gamma_k' \bar{H} = (\bar{\Gamma}_k \cap D') \bar{H} = \bar{\Gamma}_k \bar{H} \cap D' = D'$ and so $\bar{H} = H \cap D'$ is full in $D'$.  Furthermore, $[D', D'] \leq [\bar{D}, \bar{D}] \leq \bar{H}$ and so $\bar{H}$ is coabelian in $D'$.
\end{proof}

\begin{lem} \label{lem31}
  Let $G$ be a non-hyperbolic, finitely presented group, and let $\delta$ be the Dehn function of some finite presentation of $G$.  Then there exists $C \in \N$ so that $l^2 \leq C\delta(l) + C$.
\end{lem}

\begin{proof}
  Since $G$ is not hyperbolic, the function $\delta$ satisfies $\delta(l) \succeq l^2$ \cite[Theorem~6.1.5]{brid02}.  Hence there exists $K \in \N$ such that $l^2 \leq K \delta(Kl + K) + Kl + K$ whence $l^2 \leq K \delta(2Kl) + 2Kl$.  This implies that $\frac{1}{2} l^2 + \frac{1}{2} l^2 -2Kl \leq K \delta(2Kl)$.  Note that $\frac{1}{2} l^2 -2Kl \geq -2K^2$, so $\frac{1}{2} l^2 -2K^2 \leq K \delta(2Kl)$, which implies that $l^2 \leq 2K \delta(2Kl) + 4K^2$.  We thus have that \begin{align*}
  l^2 &= 4K^2(l/(2K))^2 \\
  &\leq 4K^2 \lfloor l/(2K) \rfloor^2 + 4K^2 \\
  &\leq 4K^2 (2K \delta(2K \lfloor l/(2K) \rfloor ) + 4K^2) + 4K^2 \\
  &\leq 8K^3 \delta(l) + 16K^4 + 4K^2.
\end{align*}
\end{proof}

\begin{proof}[Proof of Proposition~\ref{prop15}]
  Let $H$ be a virtually-full, virtually-coabelian subgroup of a direct product $D = \Gamma_1 \times \ldots \times \Gamma_n$ with corank $r$.  Suppose that each $\Gamma_i$ finitely generated and that Theorem~\ref{thm6} is true under the additional hypotheses (i), (ii) and (iii).

By Lemma~\ref{lem29}, there exists a finite index subgroup $D' \leq D$ so that $H \cap D'$ is full and coabelian in $D'$.  Since $D$ is finitely generated, we may, by replacing $D'$ by a finite index subgroup if necessary, assume that $D' / (H \cap D')$ is free abelian of rank $r$.  Define $H' = H \cap D'$ and, for each $i$, define $\Gamma_i' = \Gamma_i \cap D'$.  Note that $[H : H'] < \infty$ and $[\Gamma_i : \Gamma_i'] < \infty$.  Thus each $\Gamma_i'$ is finitely generated.

  Now suppose that $n \geq 2$.  Since we assumed that Part~(1) of Theorem~\ref{thm6} is true under the additional hypotheses, it follows that $H'$ is finitely generated and that the distortion function $\Delta'$ of $H'$ in $D'$ satisfies $\Delta'(l) \preccurlyeq l^2$.  Since $D'$ has finite index in $D$ it is undistorted.  Thus by Lemma~\ref{lem24}, the distortion function $\Delta$ of $H$ in $D$ satisfies $\Delta(l) \preccurlyeq l^2$.

  Now suppose that $n \geq 3$ and that each $\Gamma_i$ is finitely presented.  Then each $\Gamma_i'$ is finitely presented.  Since we assumed that Part~(2) of Theorem~\ref{thm6} is true under the additional hypotheses, it follows that $H'$, and hence $H$, is finitely presented.  Let $\alpha, \rho, \alpha_i, \rho_i$ be as in the statement of Part~(3) of Theorem~\ref{thm6}.  By Lemma~\ref{lem22}~(2), there exists, for each $i$, functions $\alpha_i', \rho_i' : \N \rightarrow \N$ with $\alpha_i' \simeq \alpha_i$ and $\rho_i' \simeq \rho_i$ so that $(\alpha_i', \rho_i')$ is an area-radius pair for some finite presentation of $\Gamma_i'$.  Define $\alpha'(l) = \max ( \{l^2\} \cup \{ \alpha_i'(l) \, : \, 1 \leq i \leq n\})$ and $\rho'(l) = \max ( \{l\} \cup \{ \rho_i'(l) \, : \, 1 \leq i \leq n\})$.  Then, by the assumption that Part~(3) of Theorem~\ref{thm6} is true under the additional hypotheses, ${\rho'}^{2r} \alpha'$ is an isoperimetric function for $H'$.  By definition of the equivalence $\simeq$, there exists a constant $K$ such that $\alpha_i'(l) \leq K\alpha_i(Kl + K) + Kl + K$ and $\rho_i'(l) \leq K\rho_i(Kl+K) + Kl+K$ for all $i$.  Thus $\alpha'(l) \leq K\alpha(Kl + K) + Kl + K$ and $\rho'(l) \leq K\rho(Kl+K) + Kl+K$.  Since $\rho(l) \geq l$ and $\alpha(l) \geq l^2 \geq l$ we have that $\alpha'(l) \leq (K+1) \alpha(Kl + K)$ and $\rho'(l) \leq (K+1) \rho(Kl + K)$.  Thus $(\rho'(l))^{2r}  \alpha'(l) \leq (K+1)^{2r+1} (\rho(Kl + K))^{2r} \alpha(Kl + K)$ and hence $\rho'^{2r}\alpha' \preceq \rho^{2r} \alpha$.  It follows that $\rho^{2r}\alpha$ is an isoperimetric function for $H'$ and hence, by Lemma~\ref{lem22}~(1), an isoperimetric function for $H$.

  Finally, suppose that $n \geq \max \{3, 2r\}$.  Let $\beta_1, \beta_2, \beta$ be as in the statement of Part~(4) of Theorem~\ref{thm6}.  Let $\beta_1'$ and $\beta_2'$ be the Dehn functions of some finite presentations of $\Gamma_1' \times \ldots \times \Gamma_{n-r}'$ and $\Gamma_{n-r+1}' \times \ldots \times \Gamma_n'$ respectively and define $\beta'(l) = l\beta_1'(l^2) + \beta_2'(l)$.  Then, by the assumption that Part~(4) of Theorem~\ref{thm6} is true under the additional hypotheses, $\beta'$ is an isoperimetric function for $H'$.  By Lemma~\ref{lem22}~(1), $\beta_1' \simeq \beta_1$ and $\beta_2' \simeq \beta_2$ and so, by the definition of $\simeq$-equivalence, there exists a constant $K \in \N$ so that $\beta_1'(l) \leq K\beta_1(Kl+K) + Kl + K$ and $\beta_2'(l) \leq K\beta_2(Kl+K) + Kl + K$.  Then \begin{align*}
     \beta'(l) &\leq l[ K \beta_1(Kl^2 + K) + Kl^2 + K] + K\beta_2(Kl + K) + Kl + K\\
     &= Kl \beta_1(Kl^2 + K) + K\beta_2(Kl + K) + Kl^3 + 2Kl + K.
\end{align*}  By construction, $H'$ is the kernel of a homomorphism $\Gamma_1' \times \ldots \times \Gamma_n' \rightarrow \Z^r$ that is surjective on each factor $\Gamma_i'$.  Theorem~\ref{thm6}~(4) is trivially true when $r=0$, so we may assume that $r \geq 1$.  It follows that each $\Gamma_i'$, and hence each $\Gamma_i$, contains an element of infinite order.  The condition $n \geq \{3, 2r\}$ implies that $n-r \geq 2$, and so $\Gamma_1 \times \ldots \times \Gamma_{n-r}$ contains $\Z^2$ as a subgroup and hence is not hyperbolic.  By Lemma~\ref{lem31} there thus exists $C \in \N$ so that $l^2 \leq C\beta_1(l) + C$.  We now have \begin{align*}
  \beta'(l) &\leq Kl \beta_1(Kl^2 + K) + K \beta_2(Kl+K) + KCl \beta_1(l) + KCl + 2Kl + K \\
  &\leq K(C+1) l \beta_1(Kl^2 + K) + K\beta_2(Kl+K) + K(C+2)l + K \\
  &\leq K(C+1) \beta(Kl+K) + K(C+2)l +K.
\end{align*}  Thus $\beta' \preceq \beta$ and so $\beta$ is an isoperimetric function for $H'$ and hence, by Lemma~\ref{lem22}, an isoperimetric function for $H$.
\end{proof}

\subsection{Finite generation, distortion and finite presentation}

Combined with the reduction of Proposition~\ref{prop15}, the following result proves Parts~(1) and~(2) of Theorem~\ref{thm6}.

\begin{thm} \label{thm7}
  Let $\theta$ be a homomorphism from a direct product $D = \Gamma_1 \times \ldots \times \Gamma_n$ of groups to a finitely generated free abelian group $A$ such that, for each $i$, the restriction of $\theta$ to $\Gamma_i$ is surjective. \begin{enumerate}
    \item If each $\Gamma_i$ is finitely generated and $n \geq 2$ then $\ker \theta$ is finitely generated and the distortion function $\Delta$ of $\ker \theta$ in $D$ satisfies $\Delta(l) \preccurlyeq l^2$.
    \item If each $\Gamma_i$ is finitely presented and $n \geq 3$ then $\ker \theta$ is finitely presented.
  \end{enumerate}
\end{thm}

\begin{proof}
  For Part~(1), observe that $\ker \theta$ is the fibre product of the homomorphisms $\theta |_{\Gamma_1}$ and $-\theta |_{\Gamma_2 \times \ldots \times \Gamma_n}$.  Since $A$ is finitely generated free abelian, it admits a quadratic isoperimetric function.  Thus, by Lemma~\ref{lem9}, $\ker \theta$ is finitely generated and $\Delta(l) \preceq l^2$.  Hence, by definition of the relation $\preceq$, there exists $C \in \N$ so that $\Delta(l) \leq C(Cl+C)^2 + Cl + C \leq (3C^3 + C)l^2 + C^3 + C$.  Thus $\Delta(l) \preccurlyeq l^2$.

  For Part~(2), observe that $\ker \theta$ is the fibre product of the homomorphisms $p_1 := \theta |_{\Gamma_1 \times \Gamma_2}$ and $p_2 := -\theta |_{\Gamma_3 \times \ldots \times \Gamma_n}$.  Since $\ker p_1$ is the fibre product of the homomorphisms $\theta |_{\Gamma_1}$ and $-\theta |_{\Gamma_2}$, it is finitely generated by Lemma~\ref{lem9}.  Thus $\ker \theta$ is finitely presented by Theorem~\ref{thm4}.
\end{proof}

\subsection{Heights} \label{sec5}

Throughout the remainder of Section~\ref{sec7}, we will be considering a homomorphism $\theta$ from a direct product $D = \Gamma_1 \times \ldots \times \Gamma_n$ of groups to a finitely generated free abelian group $A$ such that the restriction of $\theta$ to each $\Gamma_i$ is surjective.  After establishing some notation, which will be maintained throughout Sections~\ref{sec5}--\ref{sec6}, we will define certain \emph{height functions} that measure the departure from $\ker \theta$ of words, expressions and sequences in each of the $r$ directions given by the $\Z$-factors of $A \cong \Z^r$.

Let $t_1, \ldots, t_r$ be a free abelian basis for $A$.  For each $i$, let $\mc{A}_i = \{ a_1^{(i)}, \ldots, a_r^{(i)} \} \subseteq \Gamma_i$ be a collection of elements with $\theta(a_k^{(i)}) = t_k$, and let $\mc{B}_i \subseteq \Gamma_i$ be a collection of elements with $\theta(\mc{B}_i) = \{1\}$ and so that $\mc{X}_i:= \mc{A}_i \cup \mc{B}_i$ generates $\Gamma_i$.  Define $\mc{X}$ to be the generating set $\cup_{i=1}^n \mc{X}_i$ for $D$.

For each $i = 1, \ldots, r$, define $\Z^{(i)}$ to be the quotient of $A$ by the subgroup generated by $\{t_j : j \neq i \}$.  Define $\theta_i : D \rightarrow \Z$ to be the composition of $\theta$ with the quotient homomorphism $A \twoheadrightarrow \Z^{(i)}$.  Abusing notation, we will also write $\theta_i$ for the map $\fm{X} \times \N \rightarrow \Z$ given by $\theta_i(w, l) = \theta_i(w[l])$.

For each $i = 1, \ldots, r$, we define the \emph{$i$-height} of a word to be the departure of the word in the $\Z^{(i)}$-direction, as measured by $\theta_i$.  Specifically, given $w \in\ \fm{X}$, define $$\height_i(w) = \max_{0 \leq j \leq |w|} \{ |\theta_i(w, j)|  \}.$$  If $\mc{P} = \langle \mc{X} \, | \, \mc{S} \rangle$ is a presentation for $D$, then the $i$-heights of $\mc{P}$-sequences and $\mc{P}$-expressions are defined similarly.  Given a $\mc{P}$-expression $\mc{E} = (x_j, s_j)_{j=1}^m$ and a $\mc{P}$-sequence $\Sigma = (\sigma_j)_{j=0}^m$ define \begin{align*}
  \height_i(\mc{E}) &= \max_{1 \leq j \leq m} \{ \height_i(x_j) \}, \\
  \height_i(\Sigma) &= \max_{0 \leq j \leq m} \{ \height_i(\sigma_j) \}.
\end{align*}  The following lemma makes precise the relationship between heights and departures.

\begin{lem} \label{lem12}
  \begin{align*}
    \height_i(w) &\leq \Dep_\mc{X}(w, \ker \theta) \leq \sum_{j=1}^r \height_j(w) \\
    \height_i(\mc{E}) &\leq \Dep_\mc{X}(\mc{E}, \ker \theta) \leq \sum_{j=1}^r \height_j(\mc{E}) \\
    \height_i(\Sigma) &\leq \Dep_\mc{X}(\Sigma, \ker \theta) \leq \sum_{j=1}^r \height_j(\Sigma)
  \end{align*}
\end{lem}

\begin{proof}
  For any prefix $u$ of $w$, there exists a word $v \in \fm{X}$ with $|v| \leq \Dep_\mc{X}(w, \ker \theta)$ such that $uv^{-1} \in \ker \theta$.  Then $| \theta_i(u)| = |\theta_i(v)| \leq |v|$ and so $\height_i(w) \leq \Dep_\mc{X}(w, \ker \theta)$.

  For any word $u \in \fm{X}$, the word $u (a^{(1)}_1)^{-\theta_1(u)} \ldots (a^{(1)}_r)^{-\theta_r(u)}$ represents an element of the kernel, so $\dist_\mc{X}(u, \ker \theta) \leq \sum_{j=1}^r |\theta_j(u)|$.  Taking the maximum over all prefixes $u$ of $w$ gives the inequality $\Dep_\mc{X}(w, \ker \theta) \leq \sum_{j=1}^m \height_j(w)$.

  Maximising over all the words $x_j$ or all the words $\sigma_j$ gives the other inequalities.
\end{proof}

The following lemma asserts that in order to produce a $\mc{P}$-expression for a word $w$ with some bounds on its area and heights, it suffices to produce a null $\mc{P}$-sequence for $w$ satisfying the given bounds.

\begin{lem} \label{lem13}
  Let $\Sigma$ be a $\mc{P}$-sequence converting $\tau$ to $\tau'$.  Then there exists a $\mc{P}$-expression $\mc{E}$ for $\tau (\tau')^{-1}$ with $\Area(\mc{E}) = \Area(\Sigma)$ and $\height_i(\mc{E}) \leq \height_i(\Sigma)$ for each $i$.
\end{lem}

\begin{proof}
  Say $\Sigma = (\sigma_j)_{j=0}^m$, where $\sigma_0 \equiv \tau$ and $\sigma_m \equiv \tau'$.  Define $\Sigma_1$ to be the $\mc{P}$-sequence $(\sigma_j)_{j=0}^{m-1}$.  By induction, there exists a $\mc{P}$-expression $\mc{E}_1$ for $\sigma_0 (\sigma_{m-1})^{-1}$ with $\Area(\mc{E}_1) = \Area(\Sigma_1)$ and $\height_i(\mc{E}_1) \leq \height_i(\Sigma_1) \leq \height_i(\Sigma)$ for each $i$.  If $\sigma_m$ is obtained from $\sigma_{m-1}$ by a free expansion or reduction then $\sigma_0 {\sigma_m}^{-1} \FreeEq \sigma_0 {\sigma_{m-1}}^{-1}$ and the result follows on taking $\mc{E} = \mc{E}_1$.  The other possibility is that $\sigma_m$ is obtained from $\sigma_{m-1}$ by an application-of-a-relator move.  Then $\sigma_{m-1} \equiv \alpha u \beta$ and $\sigma_m \equiv \alpha v \beta$ where $uv^{-1}$ is a cyclic conjugate of a relator $s \in \mc{S}^{\pm1}$ and $\alpha$ and $\beta$ are some words in $\fm{X}$.  Observe that either $uv^{-1} \FreeEq  u' s {u'}^{-1}$ where $u'$ is a prefix of $u$ or else $u v^{-1} \FreeEq v' s {v'}^{-1}$ where $v'$ is a prefix of $v$.

  In the first case, we have that $\sigma_{m-1} \sigma_m^{-1} \FreeEq \alpha u v^{-1} \alpha^{-1} \FreeEq \alpha u' s (\alpha u')^{-1}$.  Note that $\height_i(\alpha u') \leq \height_i(\Sigma)$ since $\alpha u'$ is a prefix of $\sigma_{m-1}$.  If we take $\mc{E}_2 = (\alpha u', s)$ then $\mc{E} = \mc{E}_1 \mc{E}_2$ has the required properties.  In the second case we can take $\mc{E}_2 = (\alpha v', s)$ and the result follows similarly.
\end{proof}

When considering the areas of sequences, one only has to take account of the application-of-a-relator moves, and can ignore the free expansions and contractions.  The following lemma shows that the same is true when one is considering the heights of sequences.

\begin{lem} \label{lem14}
  Let $w_1, w_2 \in \fm{X}$ be freely equal words.  Then there exists a $\mc{P}$-sequence $\Sigma$ converting $w_1$ to $w_2$ with $\Area(\Sigma) = 0$ and $\height_i(\Sigma) \leq \max \{\height_i(w_1), \height_i(w_2) \}$ for each $i$.
\end{lem}

\begin{proof}
  Let $\bar{w}$ be the unique freely reduced word in the free equivalence class of $w_1$ and $w_2$.  For each $k = 1, 2$, let $\Sigma_k = \left(\sigma_j^{(k)}\right)_{j=0}^{m_k}$ be a $\mc{P}$-sequence converting $w_k$ to $\bar{w}$ where each $\sigma_{j+1}^{(k)}$ is obtained from $\sigma_j^{(k)}$ by a free reduction.  Then $\height_i(\sigma_{j+1}^{(k)}) \leq \height_i(\sigma_{j}^{(k)})$ and so $\height_i(\Sigma_k) \leq \height_i(w_k)$.  Define $\Sigma_2'$ to be the $\mc{P}$-sequence $\sigma_{m_2}^{(2)}, \sigma_{m_2 - 1}^{(2)}, \ldots, \sigma_{0}^{(2)}$ converting $\bar{w}$ to $w_2$.  Then $\Sigma = \Sigma_1 \Sigma_2'$ has the required properties.
\end{proof}

The area of a null-homotopic word is equal to the area of its inverse.  The following lemma is the analogous result for heights.

\begin{lem} \label{lem15}
  Let $\Sigma$ be a null $\mc{P}$-sequence for the word $w \in \fm{X}$.  Then there exists a null $\mc{P}$-sequence $\Sigma'$ for $w^{-1}$ with $\Area(\Sigma') = \Area(\Sigma)$ and $\height_k(\Sigma') = \height_k(\Sigma)$ for each $k$.
\end{lem}

\begin{proof}
  For any null-homotopic word $u \in \fm{X}$, one has that $\theta_k(u) = 0$ and hence $\height_k(u) = \height_k(u^{-1})$.  Thus if $\Sigma = (\sigma_i)_{i=0}^m$ then we can take $\Sigma'$ to be $(\sigma_i^{-1})_{i=0}^m$.
\end{proof}

Throughout Section~\ref{sec7} we will frequently wish to assert that there exists a $\mc{P}$-sequence for a word with some stated bounds on its area and heights.  However, for reasons of space and readability we wish to avoid having to present all of the data required to define a particular such sequence.  We therefore redefine the notion of a $\mc{P}$-scheme to additionally take account of heights.   Thus, throughout this section, a $\mc{P}$-scheme is defined to consist of a sequence $(\sigma_i)_{i=1}^m$ of words in $\fm{X}$ and sequences $\Big(\alpha_i\Big)_{i=1}^{m-1}, \left(h^{(1)}_i\right)_{i=1}^{m-1}, \ldots, \left(h^{(r)}_i\right)_{i=1}^{m-1}$ of integers so that, for each $i$, there exists a $\mc{P}$-sequence $\Sigma_i$ converting $\sigma_i$ to $\sigma_{i+1}$ with $\Area(\Sigma) \leq \alpha_i$ and $\height_k(\Sigma) \leq h^{(k)}_i$ for each $k$.  The notion of a null $\mc{P}$-scheme is redefined similarly.

\subsection{Distortion}

In this section we prove the following result.

\begin{thm} \label{thm8}
  Let $\theta$ be a homomorphism from a direct product $D = \Gamma_1 \times \ldots \times \Gamma_n$ of $n \geq 2$ finitely generated groups to a finitely generated free abelian group $A$ such that the restriction of $\theta$ to each $\Gamma_i$ is surjective.  Then $\ker \theta$ is finitely generated and the distortion function $\Delta$ of $\ker \theta$ in $D$ satisfies $\Delta(l) \preccurlyeq l^{r+1}$, where $r = \dim A \otimes \Q$.
\end{thm}

Note that when combined with Proposition~\ref{prop15}, this result provides an alternative proof of the assertion of finite generation in Part~(1) of Theorem~\ref{thm6}.  However, the main purpose of Theorem~\ref{thm8} is to act as a warm up for the proof of Theorem~\ref{thm10}, which is analogous but more involved.

We continue with the notation of the previous section.  Recall that $t_1, \ldots, t_r$ is a free abelian basis for $A$.  For each $i$, $\mc{X}_i = \mc{A}_i \cup \mc{B}_i$ is a generating set for $\Gamma_i$, with $\mc{A}_i = \{ a_1^{(i)}, \ldots, a_r^{(i)} \}$ satisfying $\theta(a_k^{(i)}) = t_k$ and with $\theta(\mc{B}_i) = \{1\}$.  Thus $D$ is generated by $\mc{X} = \cup_{i=1}^n \mc{X}_i$.  For each $i = 1, \ldots, r$, $\Z^{(i)}$ is the infinite cyclic subgroup of $A$ generated by $t_i$, and $\theta_i : D \rightarrow \Z$ is the composition of $\theta$ with the projection homomorphism $A \twoheadrightarrow \Z^{(i)}$.  We also write $\theta_i$ for the map $\fm{X} \times \N \rightarrow \Z$ given by $\theta_i(w, l) = \theta_i(w[l])$.

The proof of Theorem~\ref{thm8} makes use of Proposition~\ref{prop3}: we show that every element $g \in \ker \theta$ can be represented by a word in $\fm{X}$ which has uniformly bounded departure.  We first represent $g$ by an arbitrary geodesic word, representing an edge path in the Cayley graph of $D$, which we then `pull down' until it lies close to the kernel.  Recall the height functions, defined in Section~\ref{sec5}, which measure departure in each of the $r$ different directions given by the $\Z$-factors of $A$. Proposition~\ref{prop5} shows that it is possible to pull down a word in a particular direction without increasing its height in the other directions.  The trade off to this process is that the length of the word is increased.  In Proposition~\ref{prop6} we show that, by applying Proposition~\ref{prop5} repeatedly, an arbitrary word can be pulled down to a word which has small height in every direction.  This word thus has small departure from the kernel.

For each $k = 1, \ldots, r$, we will define a function $\Phi_k$ that will be used to pull down words in the $k^\text{th}$ direction.  The idea is that if $w \in \fm{X}$ represents an element of $\ker \theta$ then $\Phi_k(w)$ will represent the same element as $w$ but will have $\height_k(\Phi_k(w)) \leq 1$.  In actual fact, we will find it useful to define $\Phi_k$ to be a function $\fm{X} \times \Z \rightarrow \fm{X}$, with $\Phi_k(w, h)$ representing the element $\left(a_k^{(1)}\right)^h w \left(a_k^{(1)}\right)^{-h-\theta_k(w)} \in \ker \theta$.  Geometrically, one thinks of the input to $\Phi_k$ as being an edge path in the Cayley graph of $G$ which starts at height $h$ and is labelled by the word $w$.  This pulling down process is represented schematically in Figure~\ref{fig3}.

The reader should note that $\Phi_k$ is only defined when $n \geq 2$, and from now on we assume that this is the case.  For brevity write $e_k$ for $a_k^{(1)}$ and $f_k$ for $a_k^{(2)}$.  Define $\Phi_k$ on $\mc{X}^{\pm1} \times \Z$ by $$\Phi_k(x, h) \equiv \begin{cases}
    (e_k {f_k}^{-1})^h x {f_k}^{-\theta_k(x)} (e_k {f_k}^{-1})^{-h-\theta_k(x)} \quad &\text{if $x \in \mc{X}_1$,}\\
    x {e_k}^{-\theta_k(x)} \quad &\text{if $x \in \mc{X}_2 \cup \ldots \cup \mc{X}_n$}
  \end{cases}$$ and $$\Phi_k(x^{-1}, h) \equiv \begin{cases}
    (e_k {f_k}^{-1})^h {f_k}^{\theta_k(x)} x^{-1} (e_k {f_k}^{-1})^{-h+\theta_k(x)} \quad &\text{if $x \in \mc{X}_1$,}\\
    {e_k}^{\theta_k(x)} x^{-1} \quad &\text{if $x \in \mc{X}_2 \cup \ldots \cup \mc{X}_n$.}
  \end{cases}$$  Extend $\Phi_k$ over $\fm{X} \times \Z$ by setting $$\Phi_k(w, h) \equiv \prod_{j=1}^{|w|} \Phi_k(w(j), \theta_k(w, j-1)+h).$$

\begin{figure}[htpb]
\psfrag{a}{$\theta_k$}
\psfrag{b}{$h'$}
\psfrag{c}{$h$}
\psfrag{d}{$x_1$}
\psfrag{e}{$x_2$}
\psfrag{f}{$x_3$}
\psfrag{g}{$x_4$}
\psfrag{h}{$e_k^h$}
\psfrag{i}{$e_k^h$}
\psfrag{j}{$e_k^{h'}$}
\psfrag{k}{$\Phi_k(x_1,0)$}
\psfrag{l}{$\Phi_k(x_2,h)$}
\psfrag{m}{$\Phi_k(x_3,h)$}
\psfrag{n}{$\Phi_k(x_4,h')$}
\psfrag{o}{$\ker \theta$}
\centering \includegraphics{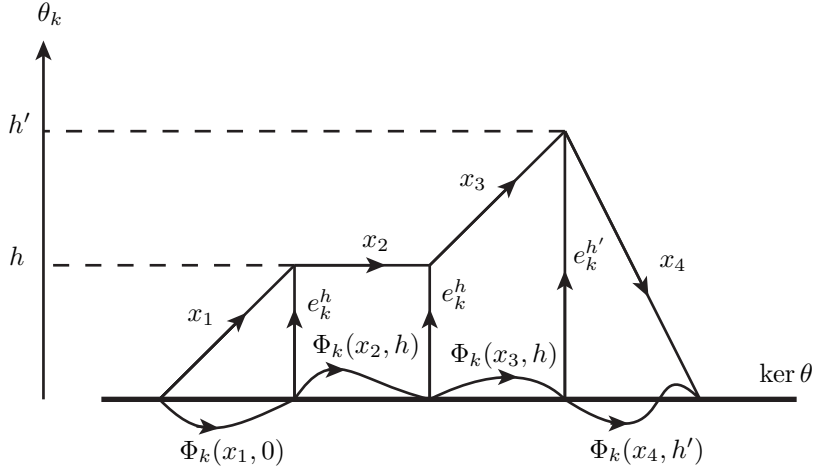}
\caption{Pulling down a word $w \equiv x_1 x_2 x_3 x_4$.} \label{fig3}
\end{figure}

\begin{prop} \label{prop5}
  Let $w, w' \in \fm{X}$, $h \in \Z$ and $k \in \{1, \ldots, r\}$.  Then $\Phi_k$ enjoys the following properties: \begin{enumerate}
    \item $\Phi_k(w, h) = {e_k}^h w {e_k}^{-h - \theta_k(w)}$ in $D$.

    \item $|\Phi_k(w, h)| \leq 4 |w| (\height_k(w) + |h| + 1)$.

    \item $\height_i(\Phi_k(w,h)) \leq \begin{cases}
      1 \quad &\text{if $i=k$,}\\
      \height_i(w) \quad &\text{if $i\neq k$}.
    \end{cases}$

    \item $\Phi_k(w, h)^{-1} \equiv \Phi_k(w^{-1}, \theta_k(w) + h)$.

    \item $\Phi_k(ww', h) \equiv \Phi_k(w, h) \Phi_k(w', \theta_k(w) + h)$.

    \item If $w \FreeEq w'$ then $\Phi_k(w, h) \FreeEq \Phi_k(w', h)$.
  \end{enumerate}
\end{prop}

\begin{proof} \mbox{}
  \begin{enumerate}
    \item If $w \in \mc{X}^{\pm1}$ then one checks directly that property (1) holds.  Thus for an arbitrary $w \in \fm{X}$ \begin{align*}
      \Phi_k(w, h) &\stackrel{D}{=} \prod_{j=1}^{|w|} {e_k}^{\theta_k(w, j-1) + h} w(j) {e_k}^{-\theta_k(w, j-1) - h - \theta_k(w(j))} \\
      &\equiv \prod_{j=1}^{|w|} {e_k}^{\theta_k(w, j-1) + h} w(j) {e_k}^{-\theta_k(w, j) - h} \\
      &\FreeEq {e_k}^{\theta_k(w,0) + h} \left( \prod_{j=1}^{|w|} w(j) \right) {e_k}^{-\theta_k(w, m) - h} \\
      &\equiv {e_k}^h w {e_k}^{-h - \theta_k(w)}.
    \end{align*}

    \item If $x \in \mc{X}^{\pm1}$ then $|\Phi_k(x, h)| \leq 4(|h| + 1)$.  Thus \begin{align*}
      |\Phi_k(w, h)| &\leq |w| \max_{1 \leq j \leq |w|} \big\{ |\Phi_k(w(j), \theta_k(w, j-1) + h)| \big\} \\
      &\leq |w| \max_{1 \leq j \leq |w|} \big\{ 4(|\theta_k(w, j-1)| + |h| + 1) \big\} \\
      &\leq 4 |w| (\height_k(w) + |h| + 1).
    \end{align*}

    \item If $x \in \mc{X}^{\pm1}$ and $u$ is a prefix of $\Phi_k(x, h)$ then $$\theta_i(u) \in \begin{cases}
        \{0, 1\} \quad &\text{if $i = k$,}\\
        \{0, \theta_i(x)\} \quad &\text{if $i \neq k$.}
    \end{cases}$$  Furthermore $$\theta_i(\Phi_k(x, h)) = \begin{cases}
        0 \quad &\text{if $i=k$,}\\
        \theta_i(x) \quad &\text{if $i \neq k$.}
    \end{cases}$$  Thus if $v$ is a prefix of $\Phi_k(w, h)$ then $\theta_k(v) \in \{-1, 0, 1\}$ and so $\height_k(\Phi_k(w, h)) \leq 1$. If $i \neq k$ then $\theta_i(v) = \theta_i(v')$ for some prefix $v'$ of $w$.  Thus $\height_i(\Phi_k(w, h)) \leq \height_i(w)$.

    \item  One checks directly that if $x \in \mc{X}^{\pm1}$ then $\Phi_k(x^{-1}, h) \equiv \Phi_k(x, h - \theta_k(x))^{-1}$.  Thus \begin{align*}
          &\Phi_k\big(w^{-1}, \theta_k(w) + h \big) \\
          &\quad \equiv \prod_{j=1}^{|w|} \Phi_k\big(w^{-1}(j), \theta_k(w^{-1}, j-1) + \theta_k(w)  + h \big) \\
          &\quad \equiv \prod_{j=1}^{|w|} \Phi_k\big(w(|w| - j + 1)^{-1}, \theta_k(w, |w| - j + 1) + h \big) \\
          &\quad \equiv \prod_{j=1}^{|w|} \Phi_k\big(w(|w| - j + 1), \theta_k(w, |w| - j + 1) + h - \theta_k(w(|w| - j + 1)) \big)^{-1} \\
          &\quad \equiv \prod_{j=1}^{|w|} {\Phi_k\big(w(|w| - j + 1), \theta_k(w, |w| - j) + h \big)}^{-1} \\
          &\quad \equiv \left(\prod_{l=1}^{|w|} \Phi_k\big(w(l), \theta_k(w, l-1) + h \big) \right)^{-1} \\
          &\quad \equiv {\Phi_k\big(w, h\big)}^{-1}
        \end{align*}

    \item \begin{align*}
      \Phi_k(ww', h) &\equiv \prod_{j=1}^{|ww'|} \Phi_k((ww')(j), \theta_k(ww', j-1) + h) \\
      &\equiv \left( \prod_{j=1}^{|w|} \Phi_k(w(j), \theta_k(w, j-1) + h) \right) \\
       &\qquad \left( \prod_{j=1}^{|w'|} \Phi_k(w'(j), \theta_k(w', j-1) + \theta_k(w) + h) \right) \\
      &\equiv \Phi_k(w, h) \Phi_k(w', \theta_k(w) + h)
    \end{align*}

    \item It suffices to consider the case where $w'$ is obtained from $w$ by a free expansion.  Say $w \equiv uv$ and $w' \equiv u x x^{-1} v$ where $u, v \in \fm{X}$ and $x \in \mc{X}^{\pm1}$.  Then \begin{align*}
          &\Phi_k(w', h) \\
          &\quad \equiv \Phi_k(u, h) \; \Phi_k(x, \theta_k(u) + h) \; \Phi_k(x^{-1}, \theta_k(ux) + h) \; \Phi_k(v, \theta_k(uxx^{-1}) + h) \\
          &\quad \equiv \Phi_k(u, h) \; \Phi_k(x, \theta_k(u) + h) \; {\Phi_k(x, \theta_k(ux) + h + \theta_k(x^{-1}))}^{-1} \; \Phi_k(v, \theta_k(u) + h) \\
          &\quad \equiv \Phi_k(u, h) \; \Phi_k(x, \theta_k(u) + h) \; {\Phi_k(x, \theta_k(u) + h)}^{-1} \; \Phi_k(v, \theta_k(u) + h) \\
          &\quad \FreeEq \Phi_k(u, h) \; \Phi_k(v, \theta_k(u) + h) \\
          &\quad \equiv \Phi_k(w, h)
        \end{align*}
  \end{enumerate}
\end{proof}

\begin{prop} \label{prop6}
  Suppose $n \geq 2$.  Then for all words $w \in \fm{X}$ with $\theta(w) = 1$, there exists a word $w' \in \fm{X}$ with the following properties: \begin{enumerate}
    \item $w' = w$ in $D$.

    \item $|w'| \leq 8^r |w|^{r+1}$.

    \item $\height_i(w') \leq 1$ for all $i$.
  \end{enumerate}
\end{prop}

\begin{proof}
  If $w \equiv \emptyset$ then the result is trivial.  We may thus assume that $|w| \geq 1$.  We claim that for all $j \in \{0, \ldots, r\}$ there exists a word $w_j \in \fm{X}$ with the following properties:
  \renewcommand{\labelenumi}{(\roman{enumi})}
  \begin{enumerate}
    \item $w_j = w$ in $D$.

    \item $|w_j| \leq 8^j |w|^{j+1}$.

    \item $\height_l(w_j) \leq \begin{cases}
      1 \quad &\text{if $1 \leq l \leq j$,}\\
      |w| \quad &\text{if $j+1 \leq l \leq r$.}
    \end{cases}$
  \end{enumerate}
  The proposition then follows by taking $j = r$.  We prove the claim by induction on $j$, with $w_0 \equiv w$.  Suppose that for some $j$ there exists a $w_j$ with the given properties.  Then define $w_{j+1} \equiv \Phi_{j+1}(w_j, 0)$.  It is immediate by Proposition~\ref{prop5}~(1) and (3) that $w_{j+1}$ satisfies (i) and (iii).  Furthermore, by Proposition~\ref{prop5}~(2), \begin{align*}
  |w_{j+1}| &\leq 4 |w_j| (\height_{j+1}(w_j) + 1) \\
  &\leq 4 \cdot 8^j |w|^{j+1} (|w| + 1) \\
  &\leq 8^{j+1} |w|^{j+2}.
  \end{align*}
\end{proof}

\begin{proof}[Proof of Theorem~\ref{thm8}]
   Since each $\Gamma_i$ is finitely generated we can take each $\mc{B}_i$ to be finite and so $D$ is finitely generated by $\mc{X}$.

  Let $g$ be an arbitrary element of $\ker \theta$ and choose a geodesic word $w \in \fm{X}$ representing $g$ in $D$.  Let $w' \in \fm{X}$ be a word equal to $w$ in $D$ and satisfying properties (2) and (3) of Proposition~\ref{prop6}.  Then $\Dep_\mc{X}(w', \ker \theta) \leq r$ by Lemma~\ref{lem12} and $|w'| \leq 8^r |w|^{r+1} = 8^r (d_\mc{X}(1, g))^{r+1}$ since $w$ is geodesic.  The result follows by applying Proposition~\ref{prop3}.
\end{proof}

\subsection{Isoperimetric functions 1} \label{sec6}

In this section we prove the following result, which, when combined with Proposition~\ref{prop15}, gives Parts~(2) and (3) of Theorem~\ref{thm6}.  Note that this provides an alternative proof, in addition to that given in Theorem~\ref{thm7}, of the finite presentability $\ker \theta$.

\begin{thm} \label{thm10}
   Let $\theta$ be a homomorphism from a direct product $D = \Gamma_1 \times \ldots \times \Gamma_n$ of $n \geq 3$ finitely presented groups to a finitely generated free abelian group $A$ such that the restriction of $\theta$ to each $\Gamma_i$ is surjective. Then $\ker \theta$ is finitely presented.

   Suppose additionally that, for each $i$, $(\alpha_i, \rho_i)$ is an area-radius pair for some finite presentation of $\Gamma_i$.  Then $\rho^{2r} \alpha$ is an isoperimetric function for $\ker \theta$, where $r = \dim A \otimes \Q$ and $\alpha$ and $\rho$ are given by $$\alpha(l) = \max ( \{l^2\} \cup \{ \alpha_i(l) \, : \, 1 \leq i \leq n\})$$ and $$\rho(l) = \max ( \{l\} \cup \{ \rho_i(l) \, : \, 1 \leq i \leq n\}).$$
\end{thm}

The proof of Theorem~\ref{thm10} is analogous to the proof of Theorem~\ref{thm8}, except that instead of pulling down words (representing edge paths in the Cayley graph of $D$) one pulls down $\mc{P}$-expressions (representing filling discs in the Cayley $2$-complex of $D$).  We first establish some notation.

Recall that, for each $i$, $\mc{X}_i = \mc{A}_i \cup \mc{B}_i$ is a generating set for $\Gamma_i$, with $\mc{A}_i = \{ a_1^{(i)}, \ldots, a_r^{(i)} \}$ satisfying $\theta(a_k^{(i)}) = t_k$ and with $\theta(\mc{B}_i) = \{1\}$.  Since each $\Gamma_i$ is finitely generated we may take each $\mc{B}_i$ to be finite.  Thus $D$ is finitely generated by $\mc{X} = \cup_{i=1}^n \mc{X}_i$.  For each $i$, let $\mc{P}_i = \langle \mc{X}_i \, | \, \mc{R}_i \rangle$ be a finite presentation for $\Gamma_i$.  Define $\mc{R} = \cup_{i=1}^n \mc{R}_i$ and define $\mc{C}$ to be the set of relators $\{ [x, y] \, : \, x \in \mc{X}_i, y \in \mc{X}_j, 1 \leq i < j \leq n \} \subseteq \fm{X}$.  Then $D$ is finitely presented by $\mc{P} = \langle \mc{X} \, | \, \mc{C}, \mc{R} \rangle$.

The structure of the proof is as follows.  Given a null-homotopic word $w \in \fm{X}$, we will apply Proposition~\ref{prop9} to give a $\mc{P}$-expression for $w$ whose area and heights (as defined in Section~\ref{sec5}) are bounded in terms of $\alpha$ and $\rho$.  We then pull this down to give an `almost flat' $\mc{P}$-expression for $w$, i.e. one which has all its heights small, in the sense of being bounded in terms of the heights of $w$.  The departure of this $\mc{P}$-expression is then bounded in terms of the departure of $w$ and so the result will follow by Proposition \ref{prop4}.

As in the $1$-dimensional case, we will use the functions $\Phi_i$ to successively pull down expressions in each of the $r$ different directions.  Proposition~\ref{prop7} asserts that it is possible to pull down a $\mc{P}$-expression in a particular direction without (essentially) increasing the heights in the other directions.  In Proposition~\ref{prop8} we apply this result repeatedly to show that an arbitrary $\mc{P}$-expression can be pulled down to one that is almost flat.

Lemmas \ref{lem16}--\ref{lem19} give various calculations required in the proof of Proposition~\ref{prop7}.  When an expression for a word $w$ is pulled down in the $i^\text{th}$ direction, one does not immediately obtain an expression for $w$, but in fact an expression for $\Phi_i(w, 0)$.  The point is that if $w$ is almost flat, then $\Phi_i(w, 0)$ will lie close to $w$ and so one can be converted to the other at low cost.  This calculation is performed in Lemmas \ref{lem16} and \ref{lem17}.

In order to pull down an expression $\mc{E}$ in the $i^\text{th}$ direction, one needs almost flat fillings for the words $\Phi_i(s, h)$, where $s \in \mc{C}^{\pm1} \cup \mc{R}^{\pm1}$.  These are provided by Lemmas \ref{lem18} and \ref{lem19}.

\begin{lem} \label{lem16}
  Suppose $n \geq 2$.  Then for all $x \in \mc{X}^{\pm1}$, $h \in \Z$ and $k \in \{1, \ldots, r\}$ there exists a $\mc{P}$-sequence $\Sigma$ converting $\Phi_k(x, h)$ to ${e_k}^h x {e_k}^{-h - \theta_k(x)}$ with \begin{align*}
    \Area(\Sigma) &\leq 2(|h| + 1)^2 \\
    \height_i(\Sigma) &\leq \begin{cases}
      |h| + 1 \quad &\text{if $i = k$,} \\
      1 \quad &\text{if $i \neq k$.}
    \end{cases} \end{align*}
\end{lem}

\begin{proof}
  We consider 4 separate cases.

  \medskip

  \noindent \textbf{Case 1.} $x \in \mc{X}_1$.

  The following table presents a $\mc{P}$-scheme converting the word $\Phi_k(x, h)$ to the word ${e_k}^h x {e_k}^{-h-\theta_k(x)}$.  In lines 3 and 5 we have applied Lemma~\ref{lem14}.

  \bigskip

  \noindent \begin{minipage}{15cm}
  \begin{tabular}{c  c  c  c  c}
    \hline
    \\[-10pt]
    $j$ & $\sigma_j$ & $\Area$ & $\begin{gathered} \height_i \\[-6pt] (i \neq k) \end{gathered}$ & $\height_k$ \\
    \\[-10pt]
    \hline
    \\[-10pt]
    1 & $(e_k {f_k})^h x {f_k}^{-\theta_k(x)} (e_k {f_k}^{-1})^{-h-\theta_k(x)}$ & $\frac{1}{2} |h| (|h| + 1)$ & 1 & $\max\{|h|, 1\}$ \\
    2 & ${e_k}^h {f_k}^{-h} x {f_k}^{-\theta_k(x)} (e_k {f_k}^{-1})^{-h-\theta_k(x)}$ & $\frac{1}{2} (|h| + 1)(|h| + 2)$ & $1$ & $|h| + 1$ \\
    3 & ${e_k}^h {f_k}^{-h} x {f_k}^{-\theta_k(x)} {f_k}^{h+\theta_k(x)} {e_k}^{-h-\theta_k(x)}$ & $0$ & $1$ & $|h| + 1$ \\
    4 & ${e_k}^h {f_k}^{-h} x {f_k}^{h} {e_k}^{-h-\theta_k(x)}$ & $|h|$ & $1$ & $|h| + 1$ \\
    5 & ${e_k}^h {f_k}^{-h} {f_k}^{h} x {e_k}^{-h-\theta_k(x)}$ & $0$ & $1$ & $|h| + 1$ \\
    6 & ${e_k}^h x {e_k}^{-h-\theta_k(x)}$ & & & \\
    \\[-10pt]
    \hline
  \end{tabular}
  \end{minipage}

  \bigskip

  \noindent \textbf{Case 2.} $x \in \mc{X}_2 \cup \ldots \cup \mc{X}_n$.

  The following table presents a $\mc{P}$-scheme converting the word $\Phi_k(x, h)$ to the word ${e_k}^h x {e_k}^{-h-\theta_k(x)}$.  In lines 1 and 3 we have applied Lemma~\ref{lem14}.

  \begin{center}
  \begin{tabular}{c  c  c  c  c  }
    \hline
    \\[-10pt]
    $j$ & $\sigma_j$ & $\Area$ & $\begin{gathered} \height_i \\[-6pt] (i \neq k) \end{gathered}$ & $\height_k$ \\
    \\[-10pt]
    \hline
    \\[-10pt]
    1 & $x {e_k}^{-\theta_k(x)}$ & $0$ & $1$ & $\max\{ |h|, 1 \}$ \\
    2 & ${e_k}^h {e_k}^{-h} x {e_k}^{-\theta_k(x)}$ & $|h|$ & $1$ & $|h| + 1$ \\
    3 & ${e_k}^h x {e_k}^{-h} {e_k}^{-\theta_k(x)}$ & $0$ & $1$ & $|h| + 1$ \\
    4 & ${e_k}^h x {e_k}^{-h-\theta_k(x)}$ & & & \\
    \\[-10pt]
    \hline
  \end{tabular}
  \end{center}

  \bigskip

  \noindent \textbf{Case 3.} $x \in \mc{X}_1^{-1}$.

  Similar to the case $x \in \mc{X}_1$.

  \medskip

  \noindent \textbf{Case 4.} $x \in \mc{X}_2^{-1} \cup \ldots \cup \mc{X}_n^{-1}$

  Similar to the case $x \in \mc{X}_2 \cup \ldots \cup \mc{X}_n$.
\end{proof}

\begin{lem} \label{lem17}
  Suppose $n \geq 2$.  Let $w \in \fm{X}$, $h \in \Z$ and $k \in \{1, \ldots, r\}$.  Then there exists a $\mc{P}$-sequence $\Sigma$ converting $\Phi_k(w, h)$ to ${e_k}^h w {e_k}^{-h-\theta_k(w)}$ with \begin{align*}
    \Area(\Sigma) &\leq 2|w| (\height_k(w) + |h| + 1)^2 \\
    \height_i(\Sigma) &\leq \begin{cases}
      \height_k(w) + |h| + 1 \quad &\text{if $i=k$,} \\
      \height_i(w) + 1 \quad &\text{if $i \neq k$.}
    \end{cases}
  \end{align*}
\end{lem}

\begin{proof}
  Define \begin{align*}
    \sigma_1 &\equiv \prod_{j=1}^{|w|} \Phi_k(w(j), \theta_k(w, j-1) + h), \\
    \sigma_2 &\equiv \prod_{j=1}^{|w|} {e_k}^{h+ \theta_k(w, j-1)} w(j) {e_k}^{-h-\theta_k(w, j)}, \\
    \sigma_3 &\equiv {e_k}^h w {e_k}^{-h-\theta_k(w)}.
  \end{align*}

  By Lemma~\ref{lem16}, there exists a $\mc{P}$-sequence $\Sigma_1$ converting $\sigma_1$ to $\sigma_2$ with $\Area(\Sigma_1) \leq 2|w| \max_{1 \leq j \leq |w|} (|\theta_k(w, j-1) + h| + 1)^2 \leq 2 |w| (\height_k(w) + |h| + 1)^2$ and $$\height_i(\Sigma_1) \leq \begin{cases}
    \height_k(w) + |h| + 1 \quad &\text{$i=k$,}\\
    1 \quad &\text{$i \neq k$.}
  \end{cases}$$

  By Lemma~\ref{lem14}, there exists a $\mc{P}$-sequence $\Sigma_2$ converting $\sigma_2$ to $\sigma_3$ with $\Area(\Sigma_2) = 0$ and $$\height_i(\Sigma_2) \leq \begin{cases}
    \height_k(w) + |h| \quad &\text{$i=k$,}\\
    \height_i(w) \quad &\text{$i \neq k$.}
  \end{cases}$$ Take $\Sigma = \Sigma_1 \Sigma_2$.
\end{proof}

\begin{lem} \label{lem18}
  Suppose $n \geq 2$.  Then there exist constants $C_A \in \N$ and $C_H \in \N$ so that for all $s \in \mc{R}^{\pm1}$, $h \in \Z$ and $k \in \{1, \ldots, r\}$ there exists a null $\mc{P}$-sequence $\Sigma$ for the word $\Phi_k(s, h)$ with $\Area(\Sigma) \leq C_A$ and $\height_i(\Sigma) \leq C_H$ for all $i$.
\end{lem}

\begin{proof}
  If $\mc{R}$ is empty then there is nothing to prove, so we assume that this is not the case.  Note that, by Proposition~\ref{prop5}~(4), $\Phi_k(s^{-1}, h) \equiv \Phi_k(s, h - \theta_k(s))^{-1} \equiv \Phi_k(s, h)^{-1}$.  Thus, by Lemma~\ref{lem15}, it suffices to consider only those $s \in \mc{R}$.

  For each $s \in \mc{R}$ and $k = 1, \ldots, r$, choose a null $\mc{P}$-sequence $\Sigma_{s, k}$ for $\Phi_k(s, 0)$.  Define $C_A = \max \{ \Area(\Sigma_{s, k})\, : \, s \in \mc{R}, 1 \leq k \leq r \}$ and $C_H = \max \{ \height_i(\Sigma_{s, k}) \, : \, s \in \mc{R}, 1 \leq i,k \leq r \}$.

  Note that, if $i \neq k$, then $\height_i(\Phi_k(s, h)) = \height_i(\Phi_k(s, 0)) \leq C_H$.  Furthermore $$\height_k(\Phi_k(s, h)) = \begin{cases}
    0 \quad &\text{if $h=0$ and $\height_k(s) = 0$,} \\
    1 \quad &\text{otherwise,}
  \end{cases}$$ and so $\height_k(\Phi_k(s, h)) \leq 1$.  Note that $C_H \geq 1$ and so we have that $\height_i(\Phi_k(s, h)) \leq C_H$ for all $s \in \mc{R}$, $h \in \Z$ and $i, k \in \{1, \ldots, r\}$.

  If $s \in \mc{R}_2 \cup \ldots \cup \mc{R}_n$, then for all $h \in \Z$ we have $\Phi_k(s, h) \equiv \Phi_k(s, 0)$ and so the result is immediate.  If $s \in \mc{R}_1$, then note that $\Phi_k(s, h)$ is freely equal to $(e_k {f_k}^{-1})^h \Phi_k(s, 0) (e_k {f_k}^{-1})^{-h}$.  The following table presents a null $\mc{P}$-scheme for $\Phi_k(s, h)$.  In lines 1 and 3 we have used Lemma~\ref{lem14}.

\medskip

\begin{center}
    \begin{tabular}{cccc}
      \hline
      \\[-10pt]
      $j$ & $\sigma_i$ & $\Area$ & $\height_i$ \\
      \\[-10pt]
      \hline
      \\[-10pt]
      $1$ & $\Phi_k(s, h)$ & $0$ & $C_H$ \\
      $2$ & $(e_k {f_k}^{-1})^h \Phi_k(s, 0) (e_k {f_k}^{-1})^{-h}$ & $C_A$ & $C_H$ \\
      $3$ & $(e_k {f_k}^{-1})^h (e_k {f_k}^{-1})^{-h}$ & $0$ & $C_H$ \\
      \\[-10pt]
      \hline
    \end{tabular}
  \end{center}
\end{proof}

\begin{lem} \label{lem19}
  Suppose $n \geq 3$.  Let $s \in \mc{C}^{\pm1}$, $h \in \Z$ and $k \in \{1, \ldots, r\}$.  Then there exists a null $\mc{P}$-sequence for the word $\Phi_k(s, h)$ with $\Area(\Sigma) \leq 7(|h|+1)^2$ and $\height_i(\Sigma) \leq 2$ for each $i$.
\end{lem}

\begin{proof}
  By Lemma~\ref{lem15} and Proposition~\ref{prop5}~(4), we may assume that $s \in \mc{C}$.  We consider 6 disjoint cases.  For each case we give a table presenting a null $\mc{P}$-sequence for the word $\Phi_k(s, h)$.  Say $s \equiv [x, y]$ where $x \in \mc{X}_i$ and $y \in \mc{X}_j$ and $1 \leq i < j \leq n$.

  \medskip

  \noindent \begin{minipage}{15cm}
  \textbf{Case 1.}  $i, j \geq 2$.

  \medskip

  \begin{tabular}{cccc}
    \hline
    \\[-10pt]
    $j$ & $\sigma_j$ & $\Area$ & $\height_i$ \\
    \\[-10pt]
    \hline
    \\[-10pt]
    1 & $x {e_k}^{-\theta_k(x)} y {e_k}^{-\theta_k(y)} {e_k}^{\theta_k(x)} x^{-1} {e_k}^{\theta_k(y)} y^{-1}$ & $1$ & $2$ \\
    2 & $x {e_k}^{-\theta_k(x)} {e_k}^{\theta_k(x)} y {e_k}^{-\theta_k(y)} x^{-1} {e_k}^{\theta_k(y)} y^{-1}$ & $0$ & $2$ \\
    3 & $x y {e_k}^{-\theta_k(y)} x^{-1} {e_k}^{\theta_k(y)} y^{-1}$ & $1$ & $2$ \\
    4 & $x y {e_k}^{-\theta_k(y)} {e_k}^{\theta_k(y)} x^{-1} y^{-1}$ & $0$ & $2$ \\
    5 & $x y x^{-1} y^{-1}$ & $1$ & $2$ \\
    \\[-10pt]
    \hline
    \\[-10pt]
    & Total & $3$ & $2$ \\
    \\[-10pt]
    \hline
  \end{tabular}
  \end{minipage}

  \bigskip
  \bigskip

  \noindent \begin{minipage}{15cm}
  \textbf{Case 2.} $i=1, 2 \leq j \leq n$.  $\theta_k(x) = 1$.

  \medskip

  \begin{tabular}{cccc}
    \hline
    \\[-10pt]
    $j$ & $\sigma_j$ & $\Area$ & $\height_i$ \\
    \\[-10pt]
    \hline
    \\[-10pt]
    1 & $(e_k {f_k}^{-1})^h e_k {f_k}^{-1} (e_k {f_k}^{-1})^{-h-1} y {e_k}^{-\theta_k(y)} \ldots$ & & \\
      & $\quad \ldots (e_k {f_k}^{-1})^{h + 1 + \theta_k(y)} f_k {e_k}^{-1} (e_k {f_k}^{-1})^{-h-\theta_k(y)} {e_k}^{\theta_k(y)} y^{-1}$ & $0$ & $1$ \\
    \\[-10pt]
    \hline
    \\[-10pt]
    & Total & $0$ & $1$ \\
    \\[-10pt]
    \hline
  \end{tabular}
  \end{minipage}

  \bigskip
  \bigskip

  \noindent \begin{minipage}{15cm}
  \textbf{Case 3.} $i=1, 3 \leq j \leq n$.  $\theta_k(x) = 0, \theta_k(y) = 0$.

  \medskip

  \begin{tabular}{ccccc}
    \hline
    \\[-10pt]
    $j$ & $\sigma_j$ & $\Area$ & $\begin{gathered} \height_i \\[-6pt] (i \neq k) \end{gathered}$ & $\height_k$ \\
    \\[-10pt]
    \hline
    \\[-10pt]
    1 & $(e_k {f_k}^{-1})^h x (e_k {f_k}^{-1})^{-h} y (e_k {f_k}^{-1})^h x^{-1} (e_k {f_k}^{-1})^{-h} y^{-1}$ & $2|h|$ & $2$ & $1$ \\
    2 & $(e_k {f_k}^{-1})^h x y (e_k {f_k}^{-1})^{-h} (e_k {f_k}^{-1})^h x^{-1} (e_k {f_k}^{-1})^{-h} y^{-1}$ & $0$ & $2$ & $1$ \\
    3 & $(e_k {f_k}^{-1})^h x y x^{-1} (e_k {f_k}^{-1})^{-h} y^{-1}$ & $2|h|$ & $2$ & $1$ \\
    4 & $(e_k {f_k}^{-1})^h x y x^{-1} y^{-1} (e_k {f_k}^{-1})^{-h}$ & $1$ & $2$ & $1$\\
    5 & $(e_k {f_k}^{-1})^h (e_k {f_k}^{-1})^{-h}$ & $0$ & $0$ & $1$ \\
    \\[-10pt]
    \hline
    \\[-10pt]
    & Total & $4|h| + 1$ & $2$ & $1$ \\
    \\[-10pt]
    \hline
  \end{tabular}
  \end{minipage}

  \bigskip
  \bigskip

  \noindent \begin{minipage}{15cm}
  \textbf{Case 4.} $i=1, 3 \leq j \leq n$.  $\theta_k(x) = 0, \theta_k(y) = 1$.

  \medskip

  \begin{tabular}{ccccc}
    \hline
    \\[-10pt]
    $j$ & $\sigma_j$ & $\Area$ & $\begin{gathered} \height_i \\[-6pt] (i \neq k) \end{gathered}$ & $\height_k$ \\
    \\[-10pt]
    \hline
    \\[-10pt]
    1 & $(e_k {f_k}^{-1})^h x (e_k {f_k}^{-1})^{-h} y {e_k}^{-1} (e_k {f_k}^{-1})^{h+1} x^{-1} (e_k {f_k}^{-1})^{-h-1} e_k y^{-1}$ & $3|h|$ & $1$ & $2$ \\
    2 & $(e_k {f_k}^{-1})^h x y {e_k}^{-1} (e_k {f_k}^{-1})^{-h} (e_k {f_k}^{-1})^{h+1} x^{-1} (e_k {f_k}^{-1})^{-h-1} e_k y^{-1}$ & $0$ & $1$ & $1$ \\
    3 & $(e_k {f_k}^{-1})^h x y {e_k}^{-1} e_k {f_k}^{-1} x^{-1} (e_k {f_k}^{-1})^{-h-1} e_k y^{-1}$ & $3|h|$ & $1$ & $2$ \\
    4 & $(e_k {f_k}^{-1})^h x y {e_k}^{-1} e_k {f_k}^{-1} x^{-1} (e_k {f_k}^{-1})^{-1} e_k y^{-1} (e_k {f_k}^{-1})^{-h}$ & $0$ & $1$ & $1$\\
    5 & $(e_k {f_k}^{-1})^h x y {f_k}^{-1} x^{-1} f_k y^{-1} (e_k {f_k}^{-1})^{-h}$ & $2$ & $1$ & $1$ \\
    6 & $(e_k {f_k}^{-1})^h (e_k {f_k}^{-1})^{-h}$ & $0$ & $0$ & $1$ \\
    \\[-10pt]
    \hline
    \\[-10pt]
    & Total & $6|h| + 2$ & $1$ & $2$ \\
    \\[-10pt]
    \hline
  \end{tabular}
  \end{minipage}

  \bigskip

  \noindent \begin{minipage}{15cm}
  \textbf{Case 5.} $i=1, j=2$.  $\theta_k(x) = 0, \theta_k(y) = 0$.

  \medskip

  As shorthand, write $g_k$ for the letter $a_k^{(3)} \in \mc{X}_3$.

  \medskip

  \begin{tabular}{ccccc}
    \hline
    \\[-10pt]
    $j$ & $\sigma_j$ & $\Area$ & $\begin{gathered} \height_i \\[-6pt] (i \neq k) \end{gathered}$ & $\height_k$ \\
    \\[-10pt]
    \hline
    \\[-10pt]
    1 & $(e_k {f_k}^{-1})^h x (e_k {f_k}^{-1})^{-h} y (e_k {f_k}^{-1})^h x^{-1} (e_k {f_k}^{-1})^{-h} y^{-1}$ & $0$ & $2$ & $1$ \\
    2 & $(e_k {f_k}^{-1})^h x (e_k {f_k}^{-1})^{-h} (g_k {e_k}^{-1})^{-h} (g_k {e_k}^{-1})^{h} y \ldots$ &&& \\
    & $ \ldots (e_k {f_k}^{-1})^h x^{-1} (e_k {f_k}^{-1})^{-h} y^{-1}$ & $2|h|$ & $2$ & $1$ \\
    3& $(e_k {f_k}^{-1})^h x (e_k {f_k}^{-1})^{-h} (g_k {e_k}^{-1})^{-h} y  \ldots$ &&& \\
    & $ \ldots (g_k {e_k}^{-1})^{h} (e_k {f_k}^{-1})^h x^{-1} (e_k {f_k}^{-1})^{-h} y^{-1}$ & $\frac{3}{2}|h|(|h| + 1)$ & $2$ & $2$ \\
    4 & $(e_k {f_k}^{-1})^h x (g_k {f_k}^{-1})^{-h} y  \ldots$ &&& \\
    & $ \ldots (g_k {e_k}^{-1})^{h} (e_k {f_k}^{-1})^h x^{-1} (e_k {f_k}^{-1})^{-h} y^{-1}$ & $\frac{3}{2}|h|(|h| + 1)$ & $2$ & $2$ \\
    5 & $(e_k {f_k}^{-1})^h x (g_k {f_k}^{-1})^{-h} y (g_k {f_k}^{-1})^{h} x^{-1} (e_k {f_k}^{-1})^{-h} y^{-1}$ & $2|h|$ & $2$ & $1$ \\
    6& $(e_k {f_k}^{-1})^h (g_k {f_k}^{-1})^{-h} x y (g_k {f_k}^{-1})^{h} x^{-1} (e_k {f_k}^{-1})^{-h} y^{-1}$ & $2|h|$ & $2$ & $1$ \\
    7& $(e_k {f_k}^{-1})^h (g_k {f_k}^{-1})^{-h} x y x^{-1} (g_k {f_k}^{-1})^{h} (e_k {f_k}^{-1})^{-h} y^{-1}$ & $\frac{3}{2}|h|(|h| + 1)|$ & $2$ & $2$ \\
    8& $(e_k {g_k}^{-1})^h x y x^{-1} (g_k {f_k}^{-1})^{h} (e_k {f_k}^{-1})^{-h} y^{-1}$ & $\frac{3}{2}|h|(|h| + 1)|$ & $2$ & $2$ \\
    9& $(e_k {g_k}^{-1})^h x y x^{-1} (g_k {e_k}^{-1})^{h} y^{-1}$ & $2|h|$ & $2$ & $1$ \\
    10& $(e_k {g_k}^{-1})^h x y x^{-1} y^{-1} (g_k {e_k}^{-1})^{h}$ & $1$ & $2$ & $1$ \\
    11& $(e_k {g_k}^{-1})^h (g_k {e_k}^{-1})^{h}$ & $0$ & $0$ & $1$ \\
    \\[-10pt]
    \hline
    \\[-10pt]
    & Total & $6|h|^2 + 14|h| + 1$ & $2$ & $2$ \\
    \\[-10pt]
    \hline
  \end{tabular}
  \end{minipage}

  \bigskip
  \bigskip

  \noindent \begin{minipage}{15cm}
  \textbf{Case 6.} $i=1, j=2$.  $\theta_k(x) = 0, \theta_k(y) = 1$.

  \medskip

  \begin{tabular}{ccccc}
    \hline
    \\[-10pt]
    $j$ & $\sigma_j$ & $\Area$ & $\begin{gathered} \height_i \\[-6pt] (i \neq k) \end{gathered}$ & $\height_k$ \\
    \\[-10pt]
    \hline
    \\[-10pt]
    1 & $(e_k {f_k}^{-1})^h x (e_k {f_k}^{-1})^{-h} f_k {e_k}^{-1} \ldots$ & & & \\
     & $\ldots (e_k {f_k}^{-1})^{h+1} x^{-1} (e_k {f_k}^{-1})^{-h-1} e_k {f_k}^{-1}$ & $0$ & $1$ & $1$ \\
    \\[-10pt]
    \hline
    \\[-10pt]
    & Total & $0$ & $1$ & $1$ \\
    \\[-10pt]
    \hline
  \end{tabular}
  \end{minipage}

  \medskip

\end{proof}

\begin{prop} \label{prop7}
  Suppose $n \geq 3$.  Then there exist constants $C_A' \in \N$ and $C_H' \in \N$ so that for any $k \in \{1, \ldots, r\}$ and any $\mc{P}$-expression $\mc{E}$ for a word $w \in \fm{X}$ there exists a $\mc{P}$-expression $\ol{\mc{E}}$ for $w$ with \begin{align*}
    \Area(\ol{\mc{E}}) &\leq C_A' \Area(\mc{E})(\height_k(\mc{E}) + 1)^2 + 2|w|(\height_k(w) + 1)^2 \\
    \height_i(\ol{\mc{E}}) &\leq \begin{cases}
      \max\{ \height_i(w) + 1, C_H' \} \quad &\text{if $i = k$,} \\
      \max\{ \height_i(w) + 1, C_H',  \height_i(\mc{E}) \} \quad &\text{if $i \neq k$.}
    \end{cases} \end{align*}
\end{prop}

\begin{proof}
  Say $\mc{E} = (x_j, r_j)_{j=1}^m$.  Let $C_A$ and $C_H$ be the constants given by Lemma~\ref{lem18} and define $C_A' = \max\{C_A, 7\}$ and $C_H' = \max \{ C_H, 2 \}$.  Then, by Lemmas~\ref{lem18}, \ref{lem19} and \ref{lem13}, for each $j$ there exists a $\mc{P}$-expression $\mc{E}_j$ for $\Phi_k(r_j, \theta_k(x_j))$ with $\Area(\mc{E}_j) \leq C_A' (|\theta_k(x_j)| + 1)^2 \leq C_A' (\height_k(\mc{E}) + 1)^2$ and $\height_i(\mc{E}_j) \leq C_H'$ for all $i$.  Say $\mc{E}_j = (x_{jl}, r_{jl})_{l=1}^{m_j}$.  For each $j$, define $\mc{E}'_j$ to be the $\mc{P}$-expression  $(\Phi_k(x_j, 0)x_{jl}, r_{jl})_{l=1}^{m_j}$ and define $\mc{E}'$ to be the $\mc{P}$-expression $\mc{E}'_1 \ldots \mc{E}'_m$.  Then $\Area(\mc{E}') = \sum_{j=1}^m m_j \leq C_A' \Area(\mc{E}) (\height_k(\mc{E}) + 1)^2$ and \begin{align*}
    \height_i(\mc{E}') &\leq \max_{\substack{1 \leq j \leq m \\ 1 \leq l \leq m_j}} \{ \height_i(\Phi_k(x_j, 0) x_{jl}) \} \\
    &\leq \max_{\substack{1 \leq j \leq m \\ 1 \leq l \leq m_j}} \{ \height_i(\Phi_k(x_j, 0)), \height_i(x_{jl}) \} \\
    &\leq \max_{1 \leq j \leq m} \{ \height_i(\Phi_k(x_j, 0)), C_H' \} \\
    &\leq \begin{cases}
      \max \{1, C_H'\} \quad &i=k, \\
      \max \{\height_i(x_j), C_H' \} &i\neq k
    \end{cases} \\
    &\leq \begin{cases}
      C_H' \quad &i = k, \\
      \max\{ \height_i(\mc{E}),  C_H' \} &i \neq k,
    \end{cases}
  \end{align*} where we have made use of Proposition~\ref{prop5} (3) and the fact that \mbox{$\theta_i(\Phi_k(x_j, 0)) = 0$.}  Furthermore \begin{align*}
    \partial \mc{E}' &\equiv \prod_{j=1}^m \partial \mc{E}_j' \\
    &\equiv \prod_{j=1}^m \prod_{l=1}^{m_j} \Phi_k(x_j, 0) x_{jl} r_{jl} {x_{jl}}^{-1} \Phi_k(x_j, 0)^{-1} \\
    &\FreeEq \prod_{j=1}^m \Phi_k(x_j, 0) \left( \prod_{l=1}^{m_j} x_{jl} r_{jl} {x_{jl}}^{-1} \right) \Phi_k(x_j, 0)^{-1} \\
    &\equiv \prod_{j=1}^m \Phi_k(x_j, 0) \partial \mc{E}_j \Phi_k(x_j, 0)^{-1} \\
    &\FreeEq \prod_{j=1}^m \Phi_k(x_j, 0) \Phi_k(r_j, \theta_k(x_j)) \Phi_k(x_j, 0)^{-1} \\
    &\equiv \prod_{j=1}^m \Phi_k(x_j r_j, 0) \Phi_k({x_j}^{-1}, \theta_k(x_j)) \\
    &\equiv \prod_{j=1}^m \Phi_k(x_j r_j, 0) \Phi_k({x_j}^{-1}, \theta_k(x_jr_j)) \\
    &\equiv \prod_{j=1}^m \Phi_k(x_j r_j {x_j}^{-1}, 0) \\
    &\equiv \Phi_k \left( \prod_{j=1}^m x_j r_j {x_j}^{-1}, 0 \right) \\
    &\equiv \Phi_k \left( \partial \mc{E}, 0 \right) \\
    &\FreeEq \Phi_k(w, 0),
  \end{align*} where we have made use of Proposition~\ref{prop5} (4), (5) and (6).

  Since $w$ is null-homotopic, $\theta_k(w) = 0$ and hence, by Proposition~\ref{prop5} (1), $\Phi_k(w, 0) = w$ in $D$.  By Lemma~\ref{lem17}, there exists a $\mc{P}$-sequence $\Sigma = (\sigma_j)_{j=0}^m$ converting $\Phi_k(w, 0)$ to $w$ with $\Area(\Sigma) \leq 2|w| (\height_k(w) + 1)^2$ and $\height_i(\Sigma) \leq \height_i(w) + 1$ for each $i$. Let $\Sigma^{-1}$ be the $\mc{P}$-sequence $\sigma_m, \sigma_{m-1}, \ldots, \sigma_0$ converting $w$ to $\Phi_k(w, 0)$.  Then $\Area(\Sigma^{-1}) = \Area(\Sigma)$ and $\height_i(\Sigma^{-1}) = \height_i(\Sigma)$ for each $i$.  By Lemma~\ref{lem13}, there exists a $\mc{P}$-expression $\mc{E}''$ for $w \left( \Phi_k(w, 0)\right) ^{-1}$ with $\Area(\mc{E}'') = \Area(\Sigma^{-1})$ and $\height_i(\mc{E}'') \leq \height_i(\Sigma^{-1})$ for each $i$.  Define $\ol{\mc{E}} = \mc{E}'' \mc{E}' $.  Then $\ol{\mc{E}}$ is a $\mc{P}$-expression for $w$ with the required bounds on its area and heights.
\end{proof}

\begin{prop} \label{prop8}
  Suppose $n \geq 3$.  Then there exist constants $C_A'' \in \N$ and $C_H'' \in \N$ so that given any $\mc{P}$-expression $\mc{E}$ for a word $w \in \fm{X}$ there exists a $\mc{P}$-expression $\ol{\mc{E}}$ for $w$ with \begin{align*}
    \Area(\ol{\mc{E}}) &\leq C_A'' (\Area(\mc{E}) + |w|) \prod_{j=1}^r \zeta_j^2, \\
    \height_i(\ol{\mc{E}}) &\leq \max \{ \height_i(w) + 1, C_H'' \}, \\
  \end{align*} where $\zeta_j = \max\{ \height_j(w) + 1, \height_j(\mc{E}) + 1, C_H'' \}$.
\end{prop}

\begin{proof}
  Let $C_A'$ and $C_H'$ be the constants given by Proposition~\ref{prop7}.  We claim that, for each $l \in \{0, 1, \ldots, r\}$, there exists a $\mc{P}$-expression $\mc{E}_l$ for $w$ with \begin{align*}
    \Area(\mc{E}_l) &\leq (C_A')^{l-1} (C_A' \Area(\mc{E}) + 2l|w|) \prod_{j=1}^l \zeta_j^2 \\
    \height_i(\mc{E}_l) &\leq \begin{cases}
      \max \{\height_i(w) + 1, C_H' \} \quad &\text{if $1 \leq i \leq l$,} \\
      \max \{\height_i(w) + 1, C_H', \height_i(\mc{E})+ 1 \} \quad &\text{if $l+1 \leq i \leq r$}.
    \end{cases}
  \end{align*}

  The claim is proved by induction on $l$.  Set $\mc{E}_0 = \mc{E}$ and, given that $\mc{E}_{l-1}$ has been defined, define $\mc{E}_l$ be the $\mc{P}$-expression given by applying Proposition~\ref{prop7} to $\mc{E}_{l-1}$ with $k=l$.  Then $\mc{E}_l$ certainly satisfies the required bounds on its heights and \begin{align*}
    \Area(\mc{E}_l) &\leq C_A' \Area(\mc{E}_{l-1}) (\height_l(\mc{E}_{l-1}) + 1)^2 + 2|w|(\height_l(w) + 1)^2 \\
    &\leq C_A' \Area(\mc{E}_{l-1}) \zeta_l^2 + 2|w| \zeta_l^2 \\
    &\leq (C_A')^{l-1} (C_A' \Area(\mc{E}) + 2 (l-1) |w| ) \prod_{j=1}^l \zeta_j^2 + 2|w| \zeta_l^2 \\
    &\leq (C_A')^{l-1} (C_A' \Area(\mc{E}) + 2 (l-1) |w| ) \prod_{j=1}^l \zeta_j^2 + 2 (C_A')^l |w| \prod_{j=1}^l \zeta_j^2 \\
    &\leq (C_A')^{l-1} (C_A' \Area(\mc{E}) + 2 l |w| ) \prod_{j=1}^l \zeta_j^2
  \end{align*} as required.

  The proposition now follows by setting $C_A'' = (C_A')^{l-1} \max\{ C_A', 2r \}$ and $C_H'' = C_H'$ and taking $\ol{\mc{E}}$ to be $\mc{E}_r$.
\end{proof}

\begin{prop} \label{prop9}
  For each $i = 1, \ldots, n$, let $(\alpha_i, \rho_i)$ be an area-radius pair for some finite presentation of $\Gamma_i$, and define $\alpha(l) = \max ( \{l^2\} \cup \{ \alpha_i(l) \, : \, 1 \leq i \leq n\})$ and $\rho(l) = \max ( \{l\} \cup \{ \rho_i(l) \, : \, 1 \leq i \leq n\})$.  Then there exist functions $\ol{\alpha}, \ol{\rho} : \N \rightarrow \N$ with $\ol{\alpha} \simeq \alpha$ and $\ol{\rho} \simeq \rho$ such that the following property holds:  For any null-homotopic word $w \in \fm{X}$, there exists a $\mc{P}$-expression $\mc{E}$ for $w$ with $\Area(\mc{E}) \leq \ol{\alpha}(|w|)$ and $\height_j(\mc{E}) \leq \ol{\rho}(|w|)$ for each $j$.
\end{prop}

\begin{proof}
   Proposition~\ref{prop11} shows that, for each $i = 1, \ldots, n$, there exist functions $\alpha_i', \rho_i' : \N \rightarrow \N$ with $\alpha_i' \simeq \alpha_i$ and $\rho_i' \simeq \rho_i$ so that $(\alpha_i', \rho_i')$ is an area-radius pair for $\mc{P}_i$.  Define functions $\alpha', \rho' : \N \rightarrow \N$ by $\alpha'(l) = \max ( \{l^2\} \cup \{ \alpha_i'(l) \, : \, 1 \leq i \leq n\})$ and $\rho'(l) = \max ( \{l\} \cup \{ \rho_i'(l) \, : \, 1 \leq i \leq n\})$.  By the same reasoning as in Proposition~\ref{prop15}, one sees that $\alpha' \simeq \alpha$ and $\rho' \simeq \rho$.

  Now let $w \in \fm{X}$ be a null-homotopic word.  For each $i =1, \ldots n$, define $w_i$ to be the word $p_i(w)$, where $p_i$ is the projection map $\fm{X} \rightarrow \mc{X}_i^{\pm\bast}$.  Then $w = w_1 \ldots w_n$ in $D$, each $w_i$ is null-homotopic and $|w| = |w_1 \ldots w_n|$.  Observe that there exists a $\mc{P}$-sequence $\Sigma = (\sigma_l)_{l=1}^m$ converting $w$ to $w_1 \ldots w_n$ with area at most $|w|^2$ and with each $\sigma_{i+1}$ being obtained from $\sigma_i$ by applying a relator from $\mc{C}$.  It follows that, for each $l$, $\height_j(\sigma_l) \leq |\sigma_l| = |w|$.  Thus, by Lemma~\ref{lem13}, there exists a $\mc{P}$-expression $\mc{E}'$ for $w (w_1 \ldots w_n)^{-1}$ with $\Area(\mc{E}') \leq |w|^2$ and $\height_j(\mc{E}') \leq |w|$.

  For each $i$, let $\mc{E}_i$ be a $\mc{P}_i$-expression for $w_i$ with $\Area(\mc{E}_i) \leq \alpha_i'(|w_i|) \leq \alpha_i'(|w|) \leq \alpha'(|w|)$ and $\Rad(\mc{E}_i) \leq \rho_i'(|w_i|) \leq \rho_i'(|w|) \leq \rho'(|w|)$.  Then $\height_j(\mc{E}_i) \leq \Rad(\mc{E}_i) \leq \rho'(|w|)$.  If we set $\mc{E}''$ to be the $\mc{P}$-expression $\mc{E}_1 \ldots \mc{E}_n$ then $\Area(\mc{E}'') \leq n \alpha'(|w|)$ and $\height_j(\mc{E}'') \leq \rho'(|w|)$.  Define $\mc{E} = \mc{E}' \mc{E}''$. Then $\Area(\mc{E}) \leq |w|^2 + n \alpha'(|w|) \leq (n+1) \alpha'(|w|)$ and $\height_j(\mc{E}) \leq \max \{ |w| , \rho'(|w|) \} \leq \rho'(|w|)$.  Define $\ol{\alpha}$ and $\ol{\rho}$ by $\ol{\alpha}(l) = (n+1) \alpha'(l)$ and $\ol{\rho}(l) = \rho'(l)$.
\end{proof}

\begin{proof}[Proof of Theorem~\ref{thm10}]
   Suppose that $w \in \fm{X}$ is a null-homotopic word with $w \not\equiv \emptyset$, and let $C_A''$ and $C_H''$ be the constants given by Proposition~\ref{prop8}.  Since $w$ is null-homotopic, there exists a $\mc{P}$-expression $\mc{E}$ for $w$, and so Proposition~\ref{prop8} implies that there exists a $\mc{P}$-expression $\ol{\mc{E}}$ for $w$ with $\height_i(\ol{\mc{E}}) \leq \max \{ \height_i(w) + 1, C_H'' \}$ for each $i$.  Lemma~\ref{lem12} therefore gives that $\Dep_\mc{X}(\ol{\mc{E}}, \ker \theta) \leq r \max \{\Dep_\mc{X}(w, \ker \theta) + 1, C_H'' \}$.  By Theorem~\ref{thm8}, $\ker \theta$ is finitely generated and so Proposition~\ref{prop4}~(1) implies that $\ker \theta$ is finitely presented.

  Now suppose that, for each $i$, $(\alpha_i, \rho_i)$ is an area-radius pair for some finite presentation of $\Gamma_i$.  Let $\ol{\alpha}$ and $\ol{\rho}$ be as given by Proposition~\ref{prop9}.  Then we can take $\mc{E}$ to have $\Area(\mc{E}) \leq \ol{\alpha}(|w|)$ and $\height_i(\mc{E}) \leq \ol{\rho}(|w|)$ for each $i$.  Thus, by Proposition~\ref{prop8}, \begin{align*}
    \Area(\ol{\mc{E}}) &\leq C_A''( \ol{\alpha}(|w|) + |w|)\prod_{i=1}^r \left( \max \{ \height_i(w) + 1, \ol{\rho}(|w|) + 1, C_H'' \} \right)^2 \\
    &\leq 2 C_A'' \ol{\alpha}(|w|) \left( \max\{ |w|+1, \ol{\rho}(|w|) + 1, C_H'' \} \right)^{2r} \\
    &\leq 2 C_A'' \ol{\alpha}(|w|) \left( \max\{ \ol{\rho}(|w|) + 1, C_H'' \} \right)^{2r} \\
    &\leq 2 C_A'' \ol{\alpha}(|w|) \left( 2 C_H'' \ol{\rho}(|w|) \right)^{2r} \\
    &\leq 2^{2r+1} C_A'' (C_H'')^{2r} \ol{\alpha}(|w|) \ol{\rho}(|w|)^{2r}.
  \end{align*}  Therefore, by Proposition~\ref{prop4}~(2), $\ol{\alpha} \, \ol{\rho}^{2r}$ is an isoperimetric function for $H$. Since $\alpha(l), \rho(l), \ol{\alpha}(l), \ol{\rho}(l)$ are all $\geq l$, it follows that $\ol{\alpha} \, \ol{\rho}^{2r} \simeq \alpha \rho^{2r}$ and so $\alpha \rho^{2r}$ is an isoperimetric function for $H$.
\end{proof}

\subsection{Isoperimetric functions 2}

In this section we prove Theorem~\ref{thm6}~(4), which will follow directly from Corollary~\ref{cor2} and Proposition~\ref{prop15}.  The following result generalises \cite[Theorem~2.1]{brid01} which treats the $n=3, r=1$ case.

\begin{thm} \label{thm12}
  Let $\theta$ be a homomorphism from a direct product $D = \Gamma_1 \times \ldots \times \Gamma_n$ of $n \geq 3$ finitely presented groups to a finitely generated free abelian group $A$ such that the restriction of $\theta$ to each factor $\Gamma_i$ is surjective.  Suppose that $n \geq 2 r$, where $r = \dim A \otimes \Q$.  Then $\ker \theta$ is finitely presented.

  Define $D_1 = \Gamma_1 \times \ldots \times \Gamma_{n-r}$ and $D_2 = \Gamma_{n-r+1} \times \ldots \times \Gamma_{n}$.  Let $\Delta$ be the distortion function of $\ker \theta \cap D_1$ in $D_1$ with respect to some choice of finite generating sets and let $\beta_1$ and $\beta_2$ be the Dehn functions of $D_1$ and $D_2$ respectively with respect to some choice of finite presentations.  Then there exist functions $\Delta' \simeq \Delta$, $\beta_1' \simeq \beta_1$ and $\beta_2' \simeq \beta_2$ so that the function $\beta'$ defined by $$\beta'(l) = l \beta_1'(\Delta'(l)) + \beta_2'(l)$$ is an isoperimetric function for $\ker \theta$.  Furthermore, $\Delta'$, $\beta_1'$ and $\beta_2'$ can be chosen to be increasing and superlinear.
\end{thm}

Here a function $f: \N \rightarrow \N$ is said to be \emph{superlinear} if $f(l) \geq l$.  Note that the conditions $n \geq 2r$ and $n \geq 3$ imply that $n-r \geq 2$.  Thus $\ker \theta \cap D_1$ is the fibre product of the homomorphisms $\theta |_{\Gamma_1}$ and $-\theta|_{\Gamma_2 \times \ldots \times \Gamma_{n-r}}$ and hence is finitely generated by Lemma~\ref{lem9}.  The function $\Delta$ is therefore well-defined up to $\approx$-equivalence and hence, in particular, up to the the weaker $\simeq$-equivalence.

\begin{proof}
  Let $x_1, \ldots, x_r$ be a free abelian basis for $A$.  Note that the condition $n \geq 2r$ implies that $n - r \geq r$.  For each $i=1, \ldots, r$, let $t_i \in \Gamma_i$ and $t_i' \in \Gamma_{n-r+i}$ be such that $\theta (t_i) = \theta(t_i') = x_i$.  Define $\mc{T} = \{ t_1, \ldots, t_r\}$ and $\mc{T}' = \{ t_1', \ldots, t_r'\}$.  Define $\mc{A} = \{a_1, \ldots, a_r\}$, where $a_i = t_i (t_i')^{-1}$.

  Let $\mc{B}_1$ be a finite generating set for  $K = \ker \theta \cap D_1$.  For each $t \in \mc{T}$, $b \in \mc{B}_1$ and $\epsilon \in \{\pm1\}$, let $w_{bt\epsilon} \in \mc{B}_1^{\pm\bast}$ be a word representing $t^{\epsilon} b t^{-\epsilon}$.  Let $\mc{P}_1 = \langle \mc{B}_1, \mc{T} \, | \, \mc{S} \rangle$ be a finite presentation for $D_1$ where $\mc{S}$ includes all relations $t^{\epsilon} b t^{-\epsilon} w_{bt\epsilon}^{-1}$.  Let $\bar{\beta}_1$ be the Dehn function of $\mc{P}_1$ and define $\beta_1'(l) = \bar{\beta_1}(l) + l$.  Then $\beta_1'$ is increasing and superlinear and $\beta_1' \simeq \beta_1$.

  Let $\mc{B}_2 \subseteq \ker \theta \cap D_2$ be a finite collection of elements such that $\mc{B}_2 \cup \mc{T}'$ generates $D_2$.  Note that $\mc{B}_2$ may not generate $ker \theta \cap D_2$, which may not even be finitely generated.  Define a homomorphism $s: D_2 \rightarrow \ker \theta$ by mapping each word $w(\mc{B}_2, \mc{T}') \mapsto w(\mc{B}_2, \mc{A}^{-1})$.  Note that $s$ does indeed define a homomorphism since if $w(\mc{B}_1, \mc{T}')$ is null-homotopic then its exponent sum in each letter of $\mc{T}'$ is $0$ and hence $w(\mc{B}_1, \mc{T}')$ and $w(\mc{B}_1, \mc{A}^{-1})$ define the same group element.  Observe that $s$ is a splitting of the short exact sequence $1 \rightarrow K \rightarrow \ker \theta \rightarrow D_2 \rightarrow 1$ where the homomorphism $\ker \theta \rightarrow D_2$ is the projection homomorphism.  Define $H \cong D_2$ to be the image of $s$.  Then $\ker \theta \cong K \rtimes H$.  Let $\mc{P}_2 = \langle \mc{A}, \mc{B}_2 \, | \, \mc{R} \rangle$ be a finite presentation for $H$ and let $\bar{\beta}_2$ be the Dehn function of $\mc{P}_2$.  Define $\beta_2'(l) = \bar{\beta_2}(l) + l$.  Then $\beta_2'$ is increasing and superlinear and $\beta_2' \simeq \beta_2$.

  Define $\mc{S}' = \{ s(\mc{A}, \mc{B}_1) : s(\mc{T}, \mc{B}_1) \in \mc{S} \}$ and note that the words in $\mc{S}'$ are null-homotopic.  Define $\mc{C} = \{ [b_1, b_2] : b_i \in \mc{B}_i \}$.

  \begin{clm}
    $\ker \theta$ is presented by $\mc{P} = \langle \mc{A}, \mc{B}_1, \mc{B}_2 \, | \, \mc{R}, \mc{S}', \mc{C} \rangle$.
  \end{clm}

  To prove the claim, suppose that $w = w(\mc{A}, \mc{B}_1, \mc{B}_2)$ is a null-homotopic word.  By applying relations from $\mc{C}$ and $\mc{S}'$ we can convert $w$ to a word $w_1 w_2$, where $w_1 = w_1(\mc{B}_1)$ and $w_2 = w_2(\mc{A}, \mc{B}_2)$.  Then $w_2$ is null-homotopic in $H$ and so $w_1w_2$ can be converted to $w_1$ by applying relators from $\mc{R}$.  Furthermore, $w_1$ is null-homotopic in $D_1$, so there exists a free equality $w_1(\mc{B}_1) \FreeEq \prod u_i(\mc{T}, \mc{B}_1) s_i(\mc{T}, \mc{B}_1) u_i^{-1}(\mc{T}, \mc{B}_1)$ for some words $u_i$ and some relators $s_i \in \mc{S}$.  Thus $w_1(\mc{B}_1) \FreeEq u_i(\mc{A}, \mc{B}_1) s_i(\mc{A}, \mc{B}_1) u_i^{-1}(\mc{A}, \mc{B}_1)$, completing the proof of the claim.

  \emph{A priori}, the above scheme gives an exponential isoperimetric function for $\ker \theta$.  We now show how this can be improved.  Let $\bar{\Delta}$ be the distortion function of $K$ in $D_1$ with respect to the generating sets $\mc{B}_1$ and $\mc{B}_1 \cup \mc{T}$.  Define $\Delta'(l) = \bar{\Delta}(l) + l$.  Then $\Delta'$ is increasing and superadditive.  Furthermore, $\Delta' \simeq \Delta$ since $\bar{\Delta} \approx \Delta$.

  \begin{clm}
    Let $\sigma = \sigma(\mc{A}, \mc{B}_1)$ be a word of length at most $l$ having exponent sum $0$ in each letter $a \in \mc{A}$. Let $b \in \mc{B}_2$.  Then $\Area_\mc{P}([b, \sigma]) \leq 3 \beta_1'(\Delta'(l))$.
  \end{clm}

  To prove the claim, note that, since $\sigma(\mc{A}, \mc{B}_1)$ has exponent sum $0$ in each $a \in \mc{A}$, it represents the same element of $K$ as $\sigma(\mc{T}, \mc{B}_1)$.  It is thus represented by some word $\tau = \tau(\mc{B}_1)$ with $|\tau| \leq \ol{\Delta}(l)$.  Then $\sigma(\mc{T}, \mc{B}_1) \tau^{-1}(\mc{B}_1)$ is null-homotopic and so there exists a null $\mc{P}_1$-expression $(\rho_i(\mc{T}, \mc{B}_1), s_i(\mc{T}, \mc{B}_1))$ for $\sigma(\mc{T}, \mc{B}_1) \tau^{-1}(\mc{B}_1)$ with area at most $\beta_1'(\ol{\Delta}(l) + l)$.  Thus $(\rho_i(\mc{A}, \mc{B}_1), s_i(\mc{A}, \mc{B}_1))$ is a null $\mc{P}$-expression for $\sigma(\mc{A}, \mc{B}_1) \tau^{-1}(\mc{B}_1)$ and so there exists a $\mc{P}$-sequence converting $b \sigma(\mc{A}, \mc{B}_1)$ to $b \tau(\mc{B}_1)$  with area at most $\beta_1'(\ol{\Delta}(l) + l) \leq \beta_1'(\Delta'(l))$.

  By applying relators from $\mc{C}$, we see that there exists a $\mc{P}$-sequence converting $b \tau(\mc{B}_1)$ to $\tau(\mc{B}_1) b$ with area at most $|\tau| \leq \ol{\Delta}(l)$.  Finally, we can convert $\tau(\mc{B}_1) b$ to $\sigma(\mc{A}, \mc{B}_1) b$ by a $\mc{P}$-sequence of area at most $\beta_1'(\Delta'(l))$.  Thus $\Area_\mc{P}([b, \sigma]) \leq 2\beta_1'(\Delta'(l)) + \ol{\Delta}(l) \leq 3\beta_1'(\Delta'(l))$, completing the proof of the claim.

  Now, to obtain the stated isoperimetric function for $\ker \theta$, let $w \in (\mc{A} \cup \mc{B}_1 \cup \mc{B}_2)^{\pm\bast}$ be a null-homotopic word in $\ker \theta$.  Then $w$ can be written as $u_0 x_1 u_1 x_2 \ldots x_n u_n$ where each $x_i \in (\mc{A} \cup \mc{B}_2)^{\pm1}$ and each $u_i \in \mc{B}_1^{\pm\bast}$ is some (possibly empty) word.

  We will define a sequence of words $U_0, \ldots, U_n \in (\mc{A} \cup \mc{B}_1)^{\pm\bast}$ with each $U_i$ having zero exponent sum in each letter $a \in \mc{A}$ and with the word $$w_i \equiv u_0 x_1 \ldots u_{n-i-1} x_{n-i} U_i x_{n-i+1} x_{n-i+2} \ldots x_n$$ representing the same element as $w$.  Take $U_0 \equiv u_n$ and define the subsequent $U_i$ recursively as follows.  If $x_{n-i} \in \mc{A}$ then define $U_{i+1} :\equiv u_{n-i-1} x_{n-i} U_i x_{n-i}^{-1}$.  If $x_{n-i} \in \mc{B}_2$ then define $U_{i+1}:\equiv u_{n-i-1}U_i$.  In the former case we see that $w_{i+1}$ is freely equal to $w_i$.  In the latter case the above claim shows that there exists a $\mc{P}$-sequence converting $w_i$ to $w_{i+1}$ with area at most $3 \beta_1'(\Delta'(|U_{i-1}|))$.  Now, $|U_i| \leq |u_{n-i}| + |U_{i-1}| + 2 \leq |u_{n-i}| + |u_{n-i+1}| + \ldots + |u_n| + 2i \leq 2|w|$.  Thus there exists a $\mc{P}$-sequence converting $w$ to $w_n \equiv U_n x_1 \ldots x_n$ with area at most $3n \beta_1'(\Delta'(2|w|)) \leq 3|w| \beta_1'(\Delta'(2|w|))$.

  Note that $x_1 \ldots x_n$ represents an element of $H$ and, since $U_n = U_n(\mc{A}, \mc{B}_1)$ has exponent sum $0$ in each letter $a \in \mc{A}$, that $U_n$ represents an element of $K$.  Thus, since $U_n x_1 \ldots x_n$ represents the identity in the semidirect product $K \rtimes H$, it follows that $U_n$ and $x_1 \ldots x_n$ are both null-homotopic.  Since $U_n = U_n(\mc{A}, \mc{B}_1)$ has exponent sum $0$ in each letter $a \in \mc{A}$ it represents the same element as $U_n(\mc{T}, \mc{B}_1)$.  Let $(\rho_i(\mc{T}, \mc{B}_1), s_i(\mc{T}, \mc{B}_1))$ be a $\mc{P}_1$-expression for $U_n(\mc{T}, \mc{B}_1)$ with area at most $\beta_1'(|U_n|) \leq \beta_1'(2|w|)$. Then $(\rho_i(\mc{A}, \mc{B}_1), s_i(\mc{A}, \mc{B}_1))$ is a $\mc{P}$-expression for $U_n(\mc{A}, \mc{B}_1)$, so $\Area_\mc{P}(U_n) \leq \beta_1'(2|w|)$.  Any $\mc{P}_2$-expression for $x_1 \ldots x_n$ is also a $\mc{P}$-expression for $x_1 \ldots x_n$, so $\Area_\mc{P}(x_1 \ldots x_n) \leq \beta_2'(n) \leq \beta_2'(|w|)$.  Putting these bounds together demonstrates that $\Area_\mc{P}(w) \leq 3|w| \beta_1'(\Delta'(2|w|) + \beta_1'(2|w|) + \beta_2'(|w|) \leq 4|w|\beta_1'(\Delta'(2|w|)) + \beta_2'(|w|)$.
\end{proof}

\begin{cor} \label{cor2}
  We continue with the notation and hypotheses of Theorem~\ref{thm12}.  Then the function $\beta$ defined by $$\beta(l) = l \beta_1(l^2) + \beta_2(l)$$ is an isoperimetric function for $\ker \theta$.
\end{cor}

\begin{proof}
  Since $A$ is abelian it admits a quadratic isoperimetric function.  Thus, by Lemma~\ref{lem9}, together with the definition of $\simeq$-equivalence, there exists $K \in \N$ so that the function $\Delta$ satisfies $\Delta(l) \leq K l^2$.  By the definition of $\simeq$-equivalence, there exists $C \in \N$ so that $\Delta'(l) \leq Cl^2$.  Since $\beta_1'$ is increasing, it follows that the function $\bar{\beta}$, defined by $\bar{\beta}(l) = l\beta_1'(Cl^2) + \beta_2'(l)$, is an isoperimetric function for $\ker \theta$.

  Note that the conditions $n \geq 2r$ and $n \geq 3$ imply that $n - r \geq 2$.  If $r=0$ then the result is trivial, so we may assume $r \geq 1$.  Since the restriction of $\theta$ to each of the $\Gamma_i$ is surjective, each $\Gamma_i$ contains an element of infinite order, and hence $D_1$ contains $\Z^2$ as a subgroup.  By \cite[Theorem~6.1.10~(1)]{brid02} $D_1$ is thus not hyperbolic and hence by Lemma~\ref{lem31} there exists $C' \in \N$ so that $l^2 \leq C' \beta_1(l) + C'$.

  Let $M_1 \in \N$ and $M_2 \in \N$ be the constants arising in the definition of $\beta_1'$ and $\beta_2'$ being $\preceq$ $\beta_1$ and $\beta_2$ respectively.  Define $M = \max \{M_1, M_2 \}$.  Then \begin{align*}
    \bar{\beta}(l) &= l\beta_1'(Cl^2) + \beta_2'(l) \\
    &\leq l[M \beta_1(M(Cl^2) + M) + MCl^2 + M] + M\beta_2(Ml + M) + Ml + M \\
    &= Ml \beta_1(MCl^2 + M) + M\beta_2(Ml+M) + MCl^3 + 2Ml + M \\
    &\leq Ml \beta_1(MCl^2 + M) + M\beta_2(Ml + M) + MCl(C' \beta_1(l) + C') + 2Ml + M \\
    &\leq M (CC' + 1) l \beta_1(MCl^2 + M) + M \beta_2(Ml+M) + (CC' + 2) Ml + M \\
    &\leq M (CC' + 1) \beta(MCl + M) + (CC' + 2) Ml + M.
  \end{align*}  Thus $\bar{\beta} \preceq \beta$ and so $\beta$ is an isoperimetric function for $\ker \theta$.
\end{proof}

\section{Depth of subdirect products} \label{sec15}

\subsection{Definition}

\begin{defn}
  Let $D = \Gamma_1 \times \ldots \times \Gamma_n$ be a direct product of groups.  Write $\mc{L}_n$ for the lattice of subsets of $\{1, \ldots, n\}$.  Given a subset $\mc{S} = \{i_1, \ldots, i_k \} \in \mc{L}_n$, define $D_\mc{S}$ to be the direct product $\Gamma_{i_1} \times \ldots \times \Gamma_{i_k}$ and define $p_\mc{S}$ to be the projection homomorphism $D \rightarrow D_\mc{S}$.

  The \emph{depth} of a subgroup $H \leq D$ is defined to be $$\Depth (H) = n - \max \{ k : \text{$[D_\mc{S}: p_\mc{S}(H)] < \infty$ for all $\mc{S} \in \mc{L}_n$ with $|\mc{S}| = k$} \}.$$
\end{defn}

We remark that the depth of a subgroup $H \leq D$ depends on the choice of a particular decomposition of $D$ as a direct product.  Also note that if $D$ has $n$ factors then $0 \leq \Depth(H) \leq n$.  The depth $0$ subgroups are precisely the finite index subgroups of $D$; the depth $1$ subgroups are precisely the virtually-full subgroups of $D$; and the depth $n-1$ subgroups are precisely the subdirect products of finite index subgroups of $D$.  The following lemma shows that the definition of depth given here agrees with the definition of depth given by Meinert \cite{mein94} for coabelian subgroups $H \leq D$.

\begin{lem}
  Let $H$ be a coabelian subgroup of the direct product $D = \Gamma_1 \times \ldots \times \Gamma_n$ with quotient homomorphism $\theta : D \rightarrow D/H$.  Then \begin{align*}
    \Depth(H) &= \min\{ k : \text{$[D : D_\mc{S}H] < \infty$ for all $\mc{S} \in \mc{L}_n$ with $|\mc{S}| = k$}\} \\
    &= \min\{ k : \text{$[D/H : \theta(D_\mc{S})] < \infty$ for all $\mc{S} \in \mc{L}_n$ with $|\mc{S}| = k$}\}.
  \end{align*}
\end{lem}

\begin{proof}
  That the two integers defined in the lemma are equal follows from the fact that $[D: D_\mc{S}H] = [D/H : D_\mc{S}H/H] = [D/H : \theta(D_\mc{S})]$.  To see that these are equal to the depth of $H$, note that, for any $\mc{S} \in \mc{L}_n$, $[D : D_\mc{S}H] = {[D / D_\mc{S} : D_\mc{S}H / D_\mc{S}]} = [D_{\mc{S}'} : p_{\mc{S}'}(H)]$, where $\mc{S}'$ is the complement of $\mc{S}$ in $\{1, \ldots, n\}$.  Thus $D_\mc{S}H$ has finite index in $D$ for all $\mc{S} \in \mc{L}_n$ with $|\mc{S}| = k$ if and only if $p_\mc{S}(H)$ has finite index in $D_\mc{S}$ for all $\mc{S} \in \mc{L}_n$ with $|\mc{S}| = n-k$.
\end{proof}

\subsection{Depth 1 subgroups}

The following result is essentially contained in \cite[Theorem~4.7]{Bridson2}.

\begin{prop}
  Let $H$ be a depth $1$ subgroup of a direct product $D = \Gamma_1 \times \ldots \times \Gamma_n$, where $n \geq 3$.  Then $H$ is virtually-coabelian.
\end{prop}

\begin{proof}
  Since $H$ has depth $1$, $[D: \Gamma_iH] < \infty$ for each $i$.  Define $D'$ to be the finite-index subgroup $\cap_{i=1}^n \Gamma_i H \leq D$ and, for each $i$, define $\Gamma_i' = \Gamma_i \cap D'$.  Then, for each $k$, $\Gamma_k' H = (\Gamma_k \cap D')H = \Gamma_k H \cap D' = D'$.  Thus $H$ is full in $D'$.

We will show that $H$ is coabelian in $D'$ by demonstrating that, for each $i$, $[\Gamma_i', \Gamma_i'] \leq H$.  Given $i \in \{1, \ldots, n\}$, choose $j, k \in \{1, \ldots, n\}$ so that $i, j, k$ are pairwise distinct.  Then, given $\gamma_1, \gamma_2 \in \Gamma_i'$, there exist $g_1 \in \Gamma_j'$, $g_2 \in \Gamma_k'$ and $h_1, h_2 \in H$ so that $\gamma_1 = g_1 h_1$ and $\gamma_2 = g_2 h_2$.  Thus $[\gamma_1, \gamma_2] = [\gamma_1 g_1^{-1}, \gamma_2 g_2^{-1}] = [h_1, h_2] \in H$.
\end{proof}

We thus have the following corollary to Theorem~\ref{thm6}.  Note that Part~(2) of this result was first proved by Bridson, Howie, Miller and Short \cite{Bridson1}, but our proof is independent of theirs.

\begin{cor} \label{cor3}
  Let $H$ be a depth $1$ subgroup of a direct product $D = \Gamma_1 \times \ldots \times \Gamma_n$, where $n \geq 3$.\begin{enumerate}
      \item If each $\Gamma_i$ is finitely generated then $H$ is finitely generated and the distortion function $\Delta$ of $H$ in $D$ satisfies $\Delta(l) \preccurlyeq l^2$.

      \item If each $\Gamma_i$ is finitely presented then $H$ is finitely presented.

      \item If, furthermore, for each $i$, there exist polynomials $\alpha_i$ and $\rho_i$ such that $(\alpha_i, \rho_i)$ is an area-radius pair for some finite presentation of $\Gamma_i$, then $H$ satisfies a polynomial isoperimetric inequality.
    \end{enumerate}
\end{cor}

\subsection{Subdirect products of limit groups}

The following conjecture, for which the author of this thesis makes no claims of ownership, has been suggested by various people.

\begin{conj} \label{conj1}
  Let $L_1, \ldots, L_n$ be $n \geq 2$ non-abelian limit groups and let $H$ be a subdirect product of $D = L_1 \times \ldots \times L_n$ that intersects each factor non-trivially.  Let $k$ be an integer $\geq 2$.  Then the following are equivalent:  \begin{enumerate}
  \item $H$ is of type $\mathrm{F}_k$;

  \item $H$ is of type $\mathrm{FP}_k(\Q)$;

  \item $H_i(H'; \Q)$ has finite $\Q$-dimension for all $i \leq k$ and all finite-index subgroups $H' \leq H$;

  \item $\Depth H \leq n-k$.
\end{enumerate}
\end{conj}

Note that it is easy to construct examples demonstrating that each of the 3 conditions ($H$ being subdirect; each $L_i$ being non-abelian; and each intersection $H \cap L_i$ being non-trivial) are necessary for depth to be related to finiteness in this way.

Various results provide corroborating evidence for Conjecture~\ref{conj1}.  It is standard that (1) implies (2) implies (3).  Meinert \cite{mein94} has proved that if the $L_i$ are free and $H$ is coabelian in $D$ then conditions (1), (2) and (4) are equivalent.  Bridson, Howie, Miller and Short \cite{Bridson1} have proved that, in the $k=2$ case, the conditions (1), (2) and (4) are equivalent.  It then follows from standard results that (1) and (2) are equivalent in all cases.  Building on this work, Kochloukova \cite{Kochloukova1} has proved that condition (3) implies condition (4); and that (3) and (4) are equivalent under certain stronger hypotheses.

We have the following corollary to Kochloukova's result.

\begin{cor} \label{cor4}
  Let $L_1, \ldots, L_n$ be non-abelian limit groups, with $n \geq 3$, and let $H$ be a subdirect product of $D = L_1 \times \ldots \times L_n$ that intersects each factor $L_i$ non-trivially.  Suppose that $H$ is of type $\mathrm{FP}_{n-1}(\Q)$.  Then $H$ is finitely presented and satisfies a polynomial isoperimetric inequality, and the distortion function $\Delta$ of $H$ in $D$ satisfies $\Delta(l) \preccurlyeq l^2$.
\end{cor}

\begin{proof}
  Since $H$ is of type $\mathrm{FP}_{n-1}(\Q)$, \cite[Theorem~7]{Kochloukova1} implies that $H$ has depth $1$ in $D$.  The result then follows from Corollary~\ref{cor3} on noting that, since limit groups are CAT(0) \cite{Alibegovic06}, they admit a quadratic-linear area-radius pair \cite[Proposition~III.$\Gamma$.1.6]{brid99}.
\end{proof}

\begin{prop} \label{prop16}
  Let $L_1, \ldots, L_n$ be limit groups and let $H$ be a finitely generated subgroup of the direct product $D = L_1 \times \ldots \times L_n$.  Then there exist non-abelian limit groups $L_1', \ldots, L_{n'}'$, with $n' \leq n$, and there exists a subdirect product $H' \leq D' = L_1' \times \ldots \times L_{n'}'$ with each intersection $L_i' \cap H'$ non-trivial, so that $H' \times A'$ is isomorphic to a finite index subgroup of $H$ for some finitely generated free abelian group $A'$.  Furthermore, if $\Delta$ and $\Delta'$ are the distortion functions of $H$ in $D$ and $H'$ in $D'$ respectively, then $\Delta(l) \preccurlyeq \Delta'(l) + l$.
\end{prop}

\begin{proof}
  If one of the intersections $L_i \cap H$ is trivial then the projection homomorphism $q_i : D \rightarrow L_1 \times \ldots \times L_{i-1} \times L_{i+1} \times \ldots \times L_n$ is injective on $H$.  Thus $H$ is isomorphic to a subgroup $q_i(H) \leq L_1 \times \ldots \times L_{i-1} \times L_{i+1} \times \ldots \times L_n$ and, by Lemma~\ref{lem25}, the distortion of $H$ in $D$ is at most the distortion of $q_i(H)$ in $q_i(D)$.  Thus, without loss of generality, we may assume that each of the intersections $L_i \cap H$ is non-trivial.

  For each $i$, let $p_i : D \rightarrow L_i$ be the projection homomorphism onto the factor $L_i$.  Since $H$ is finitely generated, each $p_i(H)$ is finitely generated and is thus a limit group.  By \cite[Corollary~3.12]{Wilton1}, $p_i(H)$ is undistorted in $L_i$.  Thus, by Lemma~\ref{lem24}~(1), we may assume that $H$ projects onto each $L_i$.

  If all of the $L_i$ are non-abelian then the proposition is proved.  Otherwise, define $A$ to be the direct product of those $L_i$ which are abelian, and let $L_1', \ldots, L_{n'}'$ be those $L_i$ (in some order) which are non-abelian.  Define $D' = L_1' \times \ldots \times L_{n'}'$.  Then $A$ is finitely generated free abelian and $H$ is a subdirect product of $D' \times A$ with each intersection $L_i' \cap H$ non-trivial and the intersection $A \cap H$ non-trivial.  Since $A$ is finitely generated free abelian, $A \cap H$ is a direct factor of some finite-index subgroup $\bar{A} \leq A$.  Define $K$ to be the finite-index subgroup $(D' \times \bar{A}) \cap H$ of $H$ and note that $K \leq D' \times \bar{A}$ and that $\bar{A} \cap K = A \cap H$ is a direct factor of $\bar{A}$.  Let $C$ be a choice of complement of $\bar{A} \cap K$ in $\bar{A}$ and define $\lambda$ to be the projection homomorphism $D' \times \bar{A} = D' \times (\bar{A} \cap K) \times C \rightarrow D' \times (\bar{A} \cap K)$.  Note that $\lambda$ is injective on $K$ and that $\lambda(K) = H' \times A'$ where $H'$ is the image of $K$ under the projection $D' \times \bar{A} \rightarrow D'$ and $A' = \bar{A} \cap K$.

  Given a pair of finitely generated groups $G_1 \leq G_2$, we write $\Delta_{G_1}^{G_2}$ for the distortion function of $G_1$ in $G_2$ (defined up to $\approx$-equivalence).  Then, applying Corollary~\ref{cor1} and Lemma~\ref{lem25}, we have that $\Delta = \Delta_H^D \approx \Delta_K^{D' \times \bar{A}} \preccurlyeq \Delta_{\lambda(K)}^{D' \times (\bar{A} \cap K)} = \Delta_{H' \times A'}^{D' \times A'} \approx \Delta_{H'}^{D'} + \Delta_{A'}^{A'} = \Delta' + \Delta_{A'}^{A'}$.  Thus the proof is complete on noting that for any group $G$, $\Delta_G^G(l) = l$.
\end{proof}

\begin{lem} \label{lem30}
  Let $H$ be a subgroup of a direct product $D$ of at most $2$ limit groups and suppose that $H$ is of type $\mathrm{FP}_{2}(\Q)$.  Then $H$ is finitely presented, satisfies a quadratic isoperimetric inequality, and is undistorted in $D$.
\end{lem}

\begin{proof}
  Since limit groups are CAT(0) \cite{Alibegovic06} they admit quadratic isoperimetric functions \cite[Proposition~III.$\Gamma$.1.6]{brid99}.  Thus $D$ admits a quadratic isoperimetric function.  By \cite[Lemma~7]{Bridson07}, $H$ is a virtual retract of $D$ and so the result follows immediately.
\end{proof}

\begin{thm} \label{thm19}
  Let $L_1, \ldots, L_n$ be limit groups and let $H$ be a subgroup of the direct product $D = L_1 \times \ldots \times L_n$.  Suppose that $H$ is of type $\mathrm{FP}_{m}(\Q)$, where $m = \max \{2, n-1 \}$.  Then $H$ is finitely presented and satisfies a polynomial isoperimetric inequality, and the distortion function $\Delta$ of $H$ in $D$ satisfies $\Delta(l) \preccurlyeq l^2$.
\end{thm}

\begin{proof}
  If $A$ is a finitely generated free abelian group and $G$ is an arbitrary group, then each of the following three group-theoretic properties is possessed by $G$ if and only if it is possessed by $G \times A$: being finitely presented;  being of type $\mathrm{FP}_m(\Q)$; and satisfying a polynomial isoperimetric inequality.  Furthermore, each of these three properties is preserved under passage to finite index subgroups and finite index extensions.  The theorem thus follows directly from Corollary~\ref{cor4}, Proposition~\ref{prop16} and Lemma~\ref{lem30}.
\end{proof}

Note that the assertion that a subgroup of a direct product of $3$ limit groups that is of type $\mathrm{FP}_2(\Q)$ is finitely presented was first obtained by Bridson, Howie, Miller and Short \cite{Bridson1}.

\begin{cor}
  Let $H$ be a finitely presented subgroup of a direct product $D$ of at most $3$ limit groups.  Then $H$ satisfies a polynomial isoperimetric inequality and the distortion function $\Delta$ of $H$ in $D$ satisfies $\Delta(l) \preccurlyeq l^2$.
\end{cor}

\section{A class of full coabelian subdirect products of free groups} \label{sec10}

In this section we study a class of full, coabelian subdirect products of free groups that have particularly regular structure.  We focus in detail on the member $K^3_2(2)$ of this class; this group is singled out as it is the simplest subdirect product of free groups which is not already well-understood. We derive a finite presentation for $K^3_2(2)$ and prove that its Dehn function $\delta$ satisfies $\delta(l) \succeq l^3$.  This is the first known example of a subdirect product of free groups that has Dehn function growing faster than that of the ambient direct product.

\subsection{Defining the class}

We first fix some notation which will be used throughout the section. Given integers $i, m \in \N$ let $F^{(i)}_m$ be the rank $m$ free group with basis $e^{(i)}_1, \ldots, e^{(i)}_m$.  Given an integer $r \in \N$ let $\Z^r$ be the rank $r$ free abelian group with basis $t_1, \ldots, t_r$.

Given positive integers $n, m \geq 1$ and $r \leq m$ we wish to define a group $K^n_m(r)$ to be the kernel of a homomorphism $\theta: F^{(1)}_m \times \ldots \times F^{(n)}_m \rightarrow \Z^r$ whose restriction to each factor $F^{(i)}_m$ is surjective.  For fixed $n$, $m$ and $r$, the isomorphism class of the group $K^n_m(r)$ is, up to an automorphism of the factors of the ambient group $F^{(1)}_m \times \ldots \times F^{(n)}_m$, independent of the homomorphism $\theta$. This is proved by the following lemma.

\begin{lem} \label{lem26}
Let $F$ be a rank $m$ free group.  Given a surjective homomorphism
$\phi : F \rightarrow \Z^r$ there exists a basis $e_1, \ldots, e_m$
of $F$ so that
\begin{equation*} \phi(e_i) = \begin{cases} t_i & \text{if $1 \leq
i \leq r$} \\ 0 & \text{if $r+1 \leq i \leq m$.} \end{cases}
\end{equation*}
\end{lem}

\begin{proof}
$\phi$ factors through the abelianisation homomorphism $\Ab : F
\rightarrow A$, where $A$ is the rank $m$ free abelian group $F /
[F, F]$, as $\phi = \bar{\phi} \circ \Ab$ for some homomorphism
$\bar{\phi} : A \rightarrow \Z^r$. Since $\bar{\phi}$ is surjective
$A$ splits as $A_1 \oplus A_2$ where $\bar{\phi}$ is an isomorphism
on the first factor and $0$ on the second factor. There thus exists
a basis $s_1, \ldots, s_m$ for $A$ so as
\begin{equation*} \bar{\phi}(s_i) = \begin{cases} t_i & \text{if
$1 \leq i \leq r$} \\ 0 & \text{if $r+1 \leq i \leq m$.}
\end{cases} \end{equation*}

We claim that the $s_i$ lift under $\Ab$ to a basis for $F$.  To see
this let $f_1, \ldots , f_m$ be any basis for $F$ and let
$\bar{f}_1, \ldots, \bar{f}_m$ be its image under $\Ab$, a basis for
$A$.  Let $\rho \in \Aut(A)$ be the change of basis isomorphism from
$\bar{f}_1, \ldots, \bar{f}_m$ to $s_1, \ldots, s_m$.  It suffices
to show that this lifts under $\Ab$ to an automorphism of $F$.  But
this is certainly the case since $\Aut(A) \cong GL_m(\Z)$ is
generated by the elementary transformations and each of these
obviously lifts to an automorphism.
\end{proof}

\begin{defn} \label{def2}
For integers $n, m \geq 1$ and $r \leq m$ define $K^n_m(r)$ to be
the kernel of the homomorphism $\theta: F^{(1)}_m \times \ldots
\times F^{(n)}_m \rightarrow \Z^r$ given by
\begin{equation*} \theta(e^{(i)}_j) =
\begin{cases} t_j & \text{if $1 \leq j \leq r$} \\ 0 & \text{if $r+1
\leq j \leq m$.}
\end{cases} \end{equation*}
\end{defn}

Note that $K^n_2(1)$ is the $n^\text{th}$ Stallings-Bieri group ${\rm SB}_n$.  By a result in Section~1.6 of \cite{mein94}, if $r \geq 1$ and $m \geq 2$ then $K^n_m(r)$ is of type $\mathrm{F}_{n-1}$ but not of type $\mathrm{FP}_{n}$.

\begin{prop} \mbox{}
  \begin{enumerate}
    \item If $n \geq 2$, then $K^n_m(r)$ is finitely generated and has distortion function $\Delta$ in $F^{(1)}_m \times \ldots \times F^{(n)}_m$ satisfying $\Delta(l) \preccurlyeq l^2$.

    \item If $n \geq 3$, then $K^n_m(r)$ is finitely presented and has Dehn function $\delta$ satisfying $\delta(l) \preccurlyeq l^{2 + 2r}$.

    \item If $n \geq \{3, 2r\}$, then $K^n_m(r)$ is finitely presented and has Dehn function $\delta$ satisfying $\delta(l) \preceq l^5$.
   \end{enumerate}
\end{prop}

\begin{proof}
  This follows immediately from Theorem~\ref{thm6}.  For (2), note that a finitely generated free group admits an area-radius pair $(\alpha, \rho)$ with $\alpha$ and $\rho$ linear.  For (3), note that a direct products of finitely generated free groups has Dehn function $d$ satisfying $d(l) \leq Cl^2$, for some $C \in \N$.
\end{proof}

\subsection{A splitting theorem}

Given a collection of groups $M, L_1, \ldots, L_r$ with $M \leq L_i$
for each $i$, we denote by $\hast_{i=1}^r(L_i \, ; \, M)$ the
amalgamated product $L_1 \ast_M \ldots \ast_M L_r$.

\begin{thm} \label{thm14}
If $n \geq 2$ and $r \geq 1$ then $$K^n_m(r) \cong \Big[
\hast_{k=1}^r(L_k \, ; \, M) \Big] \enspace \underset{M}{\hast}
\enspace \Big[ M \times F_{m-r} \Big]$$ where $F_{m-r}$ is a rank
$m-r$ free group, $M = K^{n-1}_m(r)$, and for each $k=1, \ldots, r$
the group $L_k \cong K^{n-1}_m(r-1)$ is the kernel of the
homomorphism $$\theta_k : F^{(1)}_m \times \ldots \times F^{(n-1)}_m
\rightarrow \Z^{r-1}$$ given by
\begin{equation*} \theta_k(e_j^{(i)}) =
\begin{cases} t_j & \text{if $1 \leq j \leq k-1$,} \\
0  & \text{if $j = k$,} \\ t_{j-1} & \text{if $k+1 \leq j \leq r$,}
\\ 0 & \text{if $r+1 \leq j \leq m$.}
\end{cases}
\end{equation*}
\end{thm}

\begin{proof}
Projecting $K^n_m(r)$ onto the factor $F^{(n)}_m$ gives the short
exact sequence $1 \rightarrow K^{n-1}_m(r) \rightarrow K^n_m(r)
\rightarrow F^{(n)}_m \rightarrow 1$.  This splits to show that
$K^n_m(r)$ has the structure of an internal semidirect product $M
\rtimes \hat{F}^{(n)}_m$ where $\hat{F}^{(n)}_m \cong F^{(n)}_m$ is
the subgroup of $F^{(n-1)}_m \times F^{(n)}_m$ generated by
$$e_1^{(n-1)}(e_1^{(n)})^{-1}, \, \ldots, \,
e_r^{(n-1)}(e_r^{(n)})^{-1}, \, e_{r+1}^{(n)}, \, \ldots, \,
e_m^{(n)}.$$ Since the action by conjugation of
$e_k^{(n-1)}(e_k^{(n)})^{-1}$ on $M$ is the same as the action of
$e_k^{(n-1)}$ and since $e_k^{(n)}$ centralises $M$ we have that
\begin{align*}K^n_m(r) &= M \rtimes
\hat{F}^{(n)}_m \\
&= \bigg[ \hast_{k=1}^r \Big(M \rtimes \Big\langle
e_k^{(n-1)}(e_k^{(n)})^{-1} \Big\rangle \, ; \, M \Big) \bigg]
\enspace \underset{M}{\hast} \enspace \bigg[ \hast_{k=r+1}^m \Big(M
\rtimes \Big\langle e_k^{(n)} \Big\rangle \, ; \, M \Big) \bigg] \\
&\cong \bigg[ \hast_{k=1}^r \Big(M \rtimes \Big\langle e_k^{(n-1)}
\Big\rangle \, ; \, M \Big) \bigg] \enspace \underset{M}{\hast}
\enspace \bigg[ \hast_{k=1}^{m-r} \Big(M
\times \Z \, ; \, M \Big) \bigg] \\
&\cong \bigg[ \hast_{k=1}^r \Big(M \rtimes \Big\langle e_k^{(n-1)}
\Big\rangle \, ; \, M \Big) \bigg] \enspace \underset{M}{\hast}
\enspace \bigg[ M \times F_{m-r} \bigg]. \\
\end{align*}

Define a homomorphism $p_k : F^{(1)}_m \times \ldots \times
F^{(n-1)}_m \rightarrow \Z$ by
\begin{equation*} p_k\left(e_j^{(i)}\right) = \begin{cases} 1 & \text{if
$j=k$,} \\ 0 & \text{otherwise,} \end{cases} \end{equation*} and
note that $L_k \cap \ker p_k$ is the kernel of the standard
homomorphism $\theta : F^{(1)}_m \times \ldots \times F^{(n-1)}_m
\rightarrow \Z^r$ given in definition \ref{def2}. Considering the
restriction of $p_k$ to $L_k$ gives the short exact sequence $1
\rightarrow K^{n-1}_m(r) \rightarrow L_k \rightarrow \Z \rightarrow
1$ which demonstrates that $L_k = K^{(n-1)}_m(r) \rtimes \langle
e_k^{(n-1)} \rangle$.
\end{proof}

Note that as a special case of this proposition we obtain $${\rm
SB}_3 = K^3_2(1) \cong K^2_2(0) \ast_{K^2_2(1)} (K^2_2(1) \times \Z)
\cong (F_2 \times F_2) \dot{\ast}_{K^2_2(1)}$$ where $\dot{\ast}$
denotes the trivial HNN extension with amalgamating homomorphism the
identity. This yields the presentation of Stallings' group used in
\cite{gers95}.

\subsection{Generating sets}

We give finite generating sets for those groups $K^n_m(r)$ which are
finitely generated.

\begin{prop} \label{prop13}
If $n \geq 2$ then $K^n_m(r)$ is generated by $S_1
\cup S_2 \cup S_3$ where \begin{align*} S_1 &=
\{e_i^{(1)}(e_i^{(k)})^{-1} \, : \, 1 \leq i \leq r , 2 \leq k \leq
n\}, \\ S_2 &= \{e_i^{(k)} \, : \, r+1 \leq i \leq m, 1 \leq k \leq
n\}, \\ S_3 &= \{[e_i^{(1)}, e_j^{(1)}] \, : \, 1 \leq i < j \leq r
\}. \end{align*}  If $n \geq 3$ then $K^n_m(r)$ is generated by $S_1 \cup S_2$.
\end{prop}

\begin{proof}
  Fix $n \geq 2$, $m \geq 1$ and $r \leq m$.  Let $\theta$ be the homomorphism given in Definition~\ref{def2}.  Since $n \geq 2$, $K^n_m(r)$ is the fibre product of the homomorphisms $\theta|_{F^{(1)}_m}$ and $-\theta|_{F^{(2)}_m \times \ldots \times F^{(1)}_m}$.  Define the following collections of elements of $K^n_m(r)$: \begin{align*}
    \mc{T}_1 &= \{ e^{(1)}_i (e^{(2)}_i)^{-1} : 1 \leq i \leq r \} \cup \{ e^{(1)}_i : r+1 \leq i \leq m \}; \\
    \mc{T}_2 &= \{ e^{(2)}_i (e^{(k)}_i)^{-1} : 1 \leq i \leq r, 3 \leq k \leq n \} \cup \{ e^{(k)}_i : r+1 \leq i \leq m , 2 \leq k \leq n \}; \\
    \mc{T}_3 &= \{ [e^{(1)}_i, e^{(1)}_j] : 1 \leq i < j \leq r \}.
  \end{align*}  By Lemma~\ref{lem9}, $K^n_m(r)$ is generated by $\mc{T}_1 \cup \mc{T}_2 \cup \mc{T}_3$.  Now note that each element of $\mc{T}_1 \cup \mc{T}_2 \cup \mc{T}_3$ can be expressed in terms of the $\mc{S}_1 \cup \mc{S}_2 \cup \mc{S}_3$.

If $n \geq 3$ then $S_1 \cup S_2$ suffices since as group elements $[e_i^{(1)}, e_j^{(2)}] = [e_i^{(1)} (e_i^{(2)})^{-1}, e_j^{(1)} (e_j^{(3)})^{-1}]$.
\end{proof}

\subsection{A presentation for $K^3_2(1)$} \label{sec11}

In Sections~\ref{sec11} and \ref{sec4} we derive finite presentations for the groups $K^n_m(r)$ in the case $m=2, n=3$.  To simplify notation we write $x_i$ for $e_1^{(i)}$ and $y_i$ for $e_2^{(i)}$.  Note that we have a short exact sequence $1 \rightarrow K^3_2(2) \rightarrow K^3_2(1) \rightarrow \Z \rightarrow 1$, where the homomorphism $K^3_2(1) \rightarrow \Z$ is given by mapping each $x_i \mapsto 0$ and each $y_i \mapsto 1$.  Finite presentations for $K^3_2(1)$ have been derived elsewhere; we derive a presentation in positive normal form with respect to the above short exact sequence, so as we can apply Theorem~\ref{thm15} to derive a presentation for $K^3_2(2)$.

Let $\alpha_1 = x_1 x_2^{-1}$, $\alpha_2 = x_1 x_3^{-1}$, $\beta_1 = y_1 y_2^{-1}$, $\beta_2 = y_1 y_3^{-1}$ and $t=y_1$.  Define $\mc{R}$ to be the collection of relations:
\begin{equation*}
\begin{gathered}
  {[\alpha_1, \alpha_2]}\\
  [\beta_1, \beta_2]\\
\end{gathered}
\qquad
\begin{gathered}
~[\alpha_1, \beta_2] [\alpha_2, \beta_1]^{-1}\\
[\alpha_1^{-1}, \beta_2] [\alpha_2^{-1}, \beta_1]^{-1}
\end{gathered}
\qquad
\begin{gathered}
~[\alpha_1, \beta_2^{-1}] [\alpha_2, \beta_1^{-1}]^{-1}\\
[\alpha_1^{-1}, \beta_2^{-1}] [\alpha_2^{-1}, \beta_1^{-1}]^{-1}
\end{gathered}
\end{equation*}

\begin{prop} \label{prop14}
  Each of the following presents $K_2^3(1)$: \begin{align*}
    \mc{P}_1 &= \langle \alpha_1, \alpha_2, y_1, y_2, y_3 \, | \, [\alpha_1,
    \alpha_2], \, [y_1, y_2], \, [y_1, y_3], \, [y_2, y_3], \, [\alpha_1, y_3],
    \, [\alpha_2, y_2], \, [\alpha_1^{-1}\alpha_2, y_1] \rangle \\
    \smallskip
    \mc{P}_2 &= \langle \alpha_1, \alpha_2, \beta_1, \beta_2, t \, | \, [\alpha_1,
    \alpha_2], \, [\beta_1, \beta_2], \, [t, \beta_1], [t, \beta_2], [\alpha_1,
    t\beta_2^{-1}], \, [\alpha_2, t \beta_1^{-1}], [\alpha_1^{-1} \alpha_2, t]
    \rangle \\
    \smallskip
    \mc{P}_3 &= \langle \alpha_1, \alpha_2, \beta_1, \beta_2, t \, | \, \mc{R}, \,
    [t, \beta_1] , \, [t, \beta_2], \, \alpha_1^t = \alpha_1^{\beta_2}, \,
    \alpha_2^t = \alpha_2^{\beta_1} \rangle
  \end{align*}
\end{prop}

\begin{proof}
  That the stated elements generate follows from Proposition~\ref{prop13}.  The proof that the relations in presentation $\mc{P}_1$ suffice is almost identical to a proof given by Gersten~\cite{gers95}, who derives a presentation of the group $\ker(F_2^{(1)} \times F_2^{(2)} \times F_2^{(3)} \rightarrow \Z)$ where the homomorphism maps each of the chosen basis elements of $F_2^{(i)}$ to the chosen generator of $\Z$.  We briefly recount the argument.

  Let $w \equiv w(\alpha_1, \alpha_2, y_1, y_2, y_3)$ be a null-homotopic word in $K^3_2(1)$.   Note that $w$ is freely equal to a word $w'(\alpha_1, \alpha_2, y_2, y_3) \prod_{i=1}^k y_1^{\epsilon_i w_i(\alpha_1, \alpha_2, y_2, y_3)}$ for some words $w'$ and $w_i$ and some $\epsilon_i \in \{\pm1\}$, and that the relations $[\alpha_1, \alpha_2]$, $[\alpha_1, y_3]$, $[\alpha_2, y_2]$ and $[y_2, y_3]$ are sufficient to convert this to a word of the form $$u(\alpha_1, y_2) v(\alpha_2, y_3) \prod_{i=1}^k y_1^{\epsilon_i u_i(\alpha_1, y_2) v_i(\alpha_2, y_3)}$$ for some words $u$, $u_i$ and $v_i$.  The relation $[\alpha_1^{-1} \alpha_2, y_1]$ is equivalent to $y_1^{\alpha_1} = y_1^{\alpha_2}$ and this, together with the relations $[y_1, y_2]$ and $[\alpha_2, y_2]$, are sufficient to convert the above word to a word $u(\alpha_1, y_2) v(\alpha_2, y_3) \prod_{i=1}^k y_1^{\epsilon_i v'_i(\alpha_2, y_3)}$ for some words $v_i'$.  Finally this can be converted to a word $u(\alpha_1, y_2) v(\alpha_2, y_3) \prod_{i=1}^k y_1^{\epsilon_i \alpha_1^{n_i}}$, where the $n_i \in \Z$, by applying the relations $[\alpha_1^{-1} \alpha_2, y_1]$, $[\alpha_1, y_3]$ and $[y_1, y_3]$.

  As a group element this word is equal to $$u(x_1, \emptyset) v(x_1, \emptyset) u(x_2^{-1}, y_2) v(x_3^{-1}, y_3) \prod_{i=1}^k x_1^{n_i} y_1^{\epsilon_i} x_1^{-n_i}.$$  Since $\{x_2^{-1}, y_2\}$ and $\{x_3^{-1}, y_3\}$ form free bases for $F_2^{(2)}$ and $F_3^{(3)}$ respectively it must be that $u$ and $v$ are freely equal to the empty word.  Similarly the elements $\{x_1^n y_1 x_1^{-n} \, : \, n \in \Z \}$ are freely independent so the product term also freely reduces to the empty word.  This completes the proof that $\mc{P}_1$ presents $K_2^3(1)$.

  To show that presentations $\mc{P}_1$ and $\mc{P}_2$ are equivalent, substitute $t=y_1$, $\beta_1 = ty_2^{-1}$ and $\beta_2 = ty_3^{-1}$ into $\mc{P}_1$ to give the presentation $$\langle \alpha_1, \alpha_2, \beta_1, \beta_2, t \, | \, [\alpha_1, \alpha_2], \, [t, \beta_1^{-1}t], \, [t, \beta_2^{-1}t], \, [\beta_1^{-1}t, \beta_2^{-1}t], \, [\alpha_1, \beta_2^{-1}t], \, [\alpha_2, \beta_1^{-1}t], \, [\alpha_1^{-1}\alpha_2, t] \rangle$$ which can easily be converted to $\mc{P}_2$.

  Finally, we show that the presentations $\mc{P}_2$ and $\mc{P}_3$ are Tietze equivalent.  The van Kampen diagram in Figure~\ref{fig4} (together with three similar ones) demonstrates that the relations in $\mc{R}$ are null-homotopic over $\mc{P}_2$.  Conversely, the van Kampen diagram in Figure~\ref{fig5} demonstrates that the relation $[\alpha_1^{-1} \alpha_2, t]$ is null-homotopic over presentation $\mc{P}_3$.

  \begin{figure}[h]
    \psfrag{a1}{$\alpha_1$}
    \psfrag{b1}{$\beta_1$}
    \psfrag{a2}{$\alpha_2$}
    \psfrag{b2}{$\beta_2$}
    \psfrag{t}{$t$}
    \centering \includegraphics{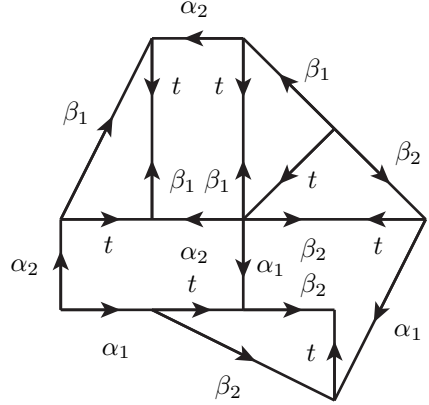}
    \caption{$\mc{P}_2$-van Kampen diagram for $[\alpha_1, \beta_2][\alpha_2, \beta_1]^{-1}$} \label{fig4}
  \end{figure}

  \begin{figure}[h]
    \psfrag{a1}{$\alpha_1$}
    \psfrag{b1}{$\beta_1$}
    \psfrag{a2}{$\alpha_2$}
    \psfrag{b2}{$\beta_2$}
    \psfrag{t}{$t$}
    \centering \includegraphics{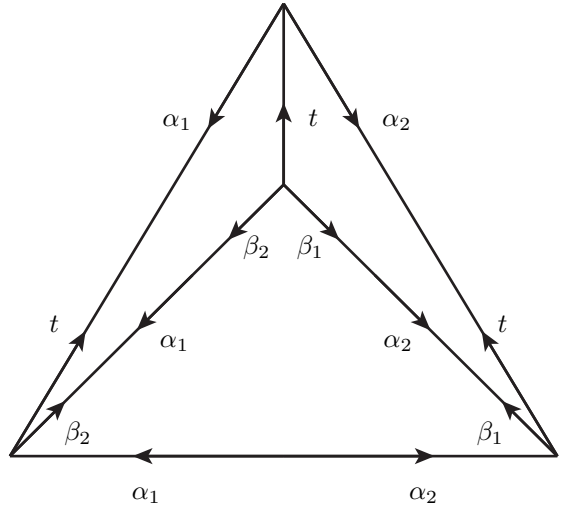}
    \caption{$\mc{P}_3$-van Kampen diagram for $[\alpha_1^{-1} \alpha_2, t]$} \label{fig5}
  \end{figure}

\end{proof}

\subsection{A presentation for $K^3_2(2)$} \label{sec4}

By Proposition~\ref{prop13}, the group $K^3_2(2)$ is generated by $\mc{X} = \{ \alpha_1, \alpha_2, \beta_1, \beta_2 \}$.  Define $\mc{R}_1$ to be the collection of relations $\mc{R}$, which we recall here for ease of use:

\begin{equation*}
\begin{gathered}
  {[\alpha_1, \alpha_2]}\\
  [\beta_1, \beta_2]\\
\end{gathered}
\qquad
\begin{gathered}
~[\alpha_1, \beta_2] [\alpha_2, \beta_1]^{-1}\\
[\alpha_1^{-1}, \beta_2] [\alpha_2^{-1}, \beta_1]^{-1}
\end{gathered}
\qquad
\begin{gathered}
~[\alpha_1, \beta_2^{-1}] [\alpha_2, \beta_1^{-1}]^{-1}\\
[\alpha_1^{-1}, \beta_2^{-1}] [\alpha_2^{-1}, \beta_1^{-1}]^{-1}
\end{gathered}
\end{equation*}
Define $\mc{R}_2$ to be the collection of relations:
\begin{equation*}
\begin{gathered}
  {[\alpha_1, \alpha_2]}\\
  [\beta_1, \beta_2]
\end{gathered}
\quad
\begin{gathered}
  {[\beta_2^{\alpha_1}, \beta_2^{-1} \beta_1]}\\
  [\beta_2^{\alpha_1^{-1}}, \beta_2^{-1} \beta_1]
\end{gathered}
\quad
\begin{gathered}
  {[\alpha_2^{\beta_1}, \alpha_2^{-1} \alpha_1]}\\
  [\alpha_2^{\beta_1^{-1}}, \alpha_2^{-1} \alpha_1]
\end{gathered}
\quad
  [\alpha_1, \beta_2] [\alpha_2, \beta_1]^{-1}
\end{equation*}

\begin{prop}
  The group $K^3_2(2)$ is presented by both $\mc{Q}_1 = \langle \mc{X} \, | \, \mc{R}_1 \rangle$ and $\mc{Q}_2 = \langle \mc{A} \, | \, \mc{R}_2 \rangle$.
\end{prop}

\begin{proof}
  We first prove that $\mc{Q}_1$ presents $K^3_2(2)$.  For each $x \in \mc{X}$, define words $w_x^+, w_x^- \in \fm{X}$ as in the following table.

  \begin{center}
    \begin{tabular}{c c c}
      \hline
      \\[-10pt]
      $x \in \mc{X}$  & $w_x^+$ & $w_x^-$ \\
      \\[-10pt]
      \hline
      \\[-10pt]
      $\alpha_1$ & $\beta_2 \alpha_1 \beta_2^{-1}$ & $\beta_2^{-1} \alpha_1 \beta_2$ \\
      $\alpha_2$ & $\beta_1 \alpha_2 \beta_1^{-1}$ & $\beta_1^{-1} \alpha_2 \beta_1$ \\
      $\beta_1$ & $\beta_1$ & $\beta_1$ \\
      $\beta_2$ & $\beta_2$ & $\beta_2$ \\
      \\[-10pt]
      \hline
    \end{tabular}
  \end{center}

  Define $\Phi^+$, $\Phi^-$ and $\mc{S}^+$, $\mc{S}^-$ as in the preamble to Theorem~\ref{thm15}.  By Proposition~\ref{prop14}, $K^3_2(1)$ is presented by $\langle \mc{X}, t \, | \, \mc{R}_1, \mc{S}^+ \rangle$.  The relations $\mc{S}^-$ are (easy) consequences of the relations $\mc{R}_1 \cup \mc{S}^+$ and so $K^3_2(1)$ is also presented by $\langle \mc{X}, t \, | \, \mc{R}_1, \mc{S}^+, \mc{S}^- \rangle$.  We are thus in a position to apply Theorem~\ref{thm15}.

  For each $x \in \mc{X}$, the relation $x \Phi^-(\Phi^+(x))$ is freely trivial.  It thus suffices to show that all the words $\Phi^\epsilon(r)$, where $\epsilon \in \{ \pm1 \}$ and $r \in \mc{R}_1$ are null-homotopic over $\mc{P}_1$.  These relations are given in the following table.

  \begin{center}
    \begin{tabular}{c c c}
      \hline
      \\[-10pt]
      $r \in \mc{R}_1$ & $\Phi^+(r)$ & $\Phi^-(r)$ \\
      \\[-10pt]
      \hline
      \\[-10pt]
      $[\alpha_1, \alpha_2]$ & $[\alpha_1^{\beta_2}, \alpha_2^{\beta_1}]$ & $[\alpha_1^{\beta_2^{-1}}, \alpha_2^{\beta_1^{-1}}]$ \\
      $[\beta_1, \beta_2]$ & $[\beta_1, \beta_2]$ & $[\beta_1, \beta_2]$ \\
      $[\alpha_1, \beta_2] [\alpha_2, \beta_1]^{-1}$ & $[\alpha_1^{\beta_2}, \beta_2] [\alpha_2^{\beta_1}, \beta_1]^{-1}$ & $[\alpha_1^{\beta_2^{-1}}, \beta_2] [\alpha_2^{\beta_1^{-1}}, \beta_1]^{-1}$ \\
      $[\alpha_1^{-1}, \beta_2] [\alpha_2^{-1}, \beta_1]^{-1}$ & $[\alpha_1^{-\beta_2}, \beta_2] [\alpha_2^{-\beta_1}, \beta_1]^{-1}$ & $[\alpha_1^{-\beta_2^{-1}}, \beta_2] [\alpha_2^{-\beta_1^{-1}}, \beta_1]^{-1}$ \\
      $[\alpha_1, \beta_2^{-1}] [\alpha_2, \beta_1^{-1}]^{-1}$ & $[\alpha_1^{\beta_2}, \beta_2^{-1}] [\alpha_2^{\beta_1}, \beta_1^{-1}]^{-1}$ & $[\alpha_1^{\beta_2^{-1}}, \beta_2^{-1}] [\alpha_2^{\beta_1^{-1}}, \beta_1^{-1}]^{-1}$ \\
      $[\alpha_1^{-1}, \beta_2^{-1}] [\alpha_2^{-1}, \beta_1^{-1}]^{-1}$ & $[\alpha_1^{-\beta_2}, \beta_2^{-1}] [\alpha_2^{-\beta_1}, \beta_1^{-1}]^{-1}$ & $[\alpha_1^{-\beta_2^{-1}}, \beta_2^{-1}] [\alpha_2^{-\beta_1^{-1}}, \beta_1^{-1}]^{-1}$ \\
      \\[-10pt]
      \hline
    \end{tabular}
  \end{center}
\end{proof}

Define a monoid endomorphism $\Lambda_\alpha: \fm{X} \rightarrow \fm{X}$, which commutes with the inversion automorphism, by mapping $\alpha_i \mapsto \alpha_i^{-1}$ and $\beta_i \mapsto \beta_i$.  Similarly, define an endomorphism $\Lambda_\beta: \fm{X} \rightarrow \fm{X}$ which commutes with the inversion automorphism by mapping $\alpha_i \mapsto \alpha_i$ and $\beta_i \mapsto \beta_i^{-1}$.  Note that if $r \in \mc{R}_1$, then both $\Lambda_\alpha(r)$ and $\Lambda_\beta(r)$ are cyclic conjugates of relations also in $\mc{R}_1$. It follows that if $w \in \fm{X}$ is null-homotopic over $\mc{Q}_1$, then so are are $\Lambda_\alpha(r)$ and $\Lambda_\beta(r)$.  Taking this symmetry into account, it thus suffices to show that the words $\Phi^+([\alpha_1, \alpha_2])$, $\Phi^+([\alpha_1, \beta_2] [\alpha_2, \beta_1]^{-1})$ and $\Phi^-([\alpha_1, \beta_2] [\alpha_2, \beta_1]^{-1})$ are null-homotopic over $\mc{Q}_1$.  $\mc{Q}_1$-van Kampen diagrams for these words are displayed in Figures~\ref{fig6}, \ref{fig7} and \ref{fig8}.

  \begin{figure}[h]
    \psfrag{a1}{$\alpha_1$}
    \psfrag{b1}{$\beta_1$}
    \psfrag{a2}{$\alpha_2$}
    \psfrag{b2}{$\beta_2$}
    \psfrag{t}{$t$}
    \centering \includegraphics{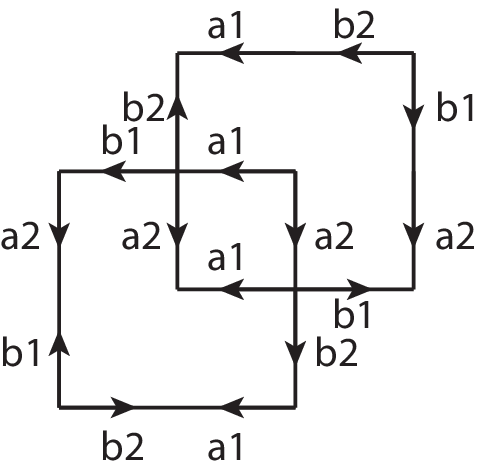}
    \caption{$\mc{Q}_1$-van Kampen diagram for $\Phi^+([\alpha_1, \alpha_2])$} \label{fig6}
  \end{figure}

  \begin{figure}[h]
    \psfrag{a1}{$\alpha_1$}
    \psfrag{b1}{$\beta_1$}
    \psfrag{a2}{$\alpha_2$}
    \psfrag{b2}{$\beta_2$}
    \psfrag{t}{$t$}
    \centering \includegraphics{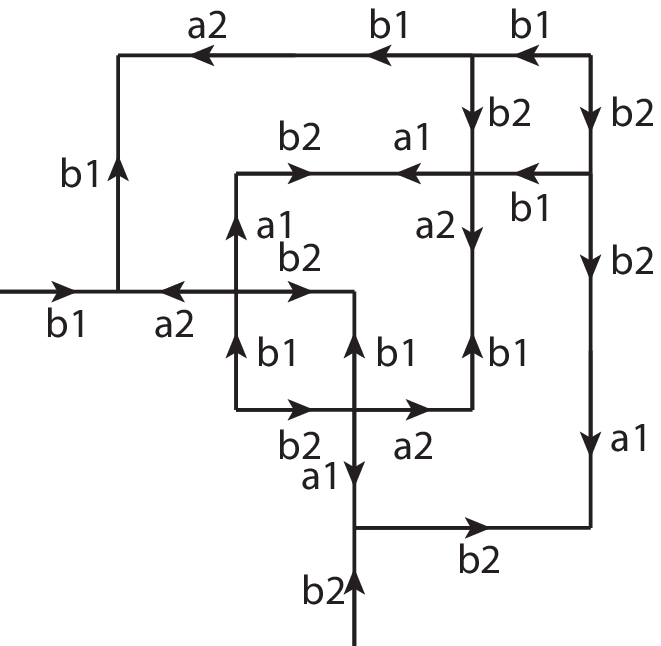}
    \caption{$\mc{Q}_1$-van Kampen diagram for $\Phi^+([\alpha_1, \beta_2] [\alpha_2, \beta_1]^{-1})$} \label{fig7}
  \end{figure}

  \begin{figure}[h]
    \psfrag{a1}{$\alpha_1$}
    \psfrag{b1}{$\beta_1$}
    \psfrag{a2}{$\alpha_2$}
    \psfrag{b2}{$\beta_2$}
    \psfrag{t}{$t$}
    \centering \includegraphics{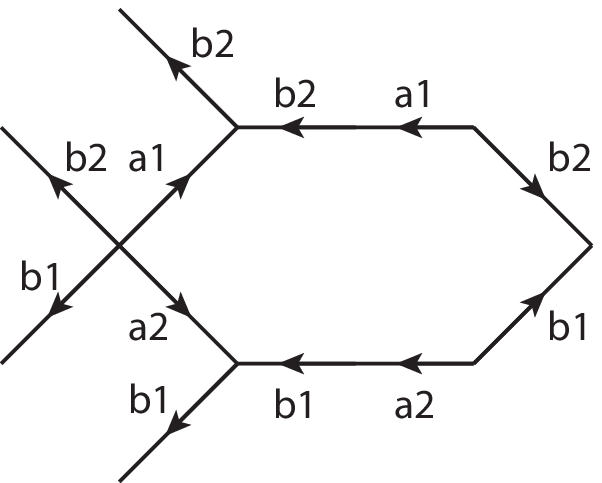}
    \caption{$\mc{P}_3$-van Kampen diagram for $\Phi^-([\alpha_1, \beta_2] [\alpha_2, \beta_1]^{-1})$} \label{fig8}
  \end{figure}

Finally, we show that $\mc{Q}_1$ and $\mc{Q}_2$ define the same group.  Define a monoid endomorphism $\ol{\Lambda} : \fm{X} \rightarrow \fm{X}$, commuting with the inversion automorphism, by mapping $\alpha_i \mapsto \beta_i$ and $\beta_i \mapsto \alpha_i$.  Note that, for $i=1$ or $2$, if $r$ is a relation in $\mc{R}_i$, then $\ol{\Lambda}(r)$ is a cyclic conjugate of some relation also in $\mc{R}_i$.  We show that each of $\mc{Q}_1$ and $\mc{Q}_2$ is Tietze equivalent to the presentation $\langle \mc{X} \, | \, \mc{R}_1, \mc{R}_2 \rangle$.  For the first equivalence, note that $\mc{R}_2$ contains 4 relations distinct from those in $\mc{R}_1$.  Taking into account the symmetries $\Lambda_\alpha$, $\Lambda_\beta$ and $\ol{\Lambda}$, it suffices to show that the word $[\beta_2^{\alpha_1}, \beta_2^{-1} \beta_1]$ is null-homotopic over $\mc{Q}_1$.  A $\mc{Q}_1$-van Kampen diagram for this word is displayed in Figure~\ref{fig9}.  For the other equivalence, note that $\mc{R}_1$ contains 3 relations distinct from those in $\mc{R}_2$.  Taking into account the symmetry $\ol{\Lambda}$, it suffices to show that the words $[\alpha_1^{-1}, \beta_2] [\alpha_2^{-1}, \beta_1]^{-1}$ and $[\alpha_1^{-1}, \beta_2^{-1}] [\alpha_2^{-1}, \beta_1^{-1}]^{-1}$ are null-homotopic over $\mc{Q}_2$.  $\mc{Q}_2$-van Kampen diagrams for these words are displayed in Figures~\ref{fig10} and \ref{fig11}.

  \begin{figure}[h]
    \psfrag{a1}{$\alpha_1$}
    \psfrag{b1}{$\beta_1$}
    \psfrag{a2}{$\alpha_2$}
    \psfrag{b2}{$\beta_2$}
    \psfrag{t}{$t$}
    \centering \includegraphics{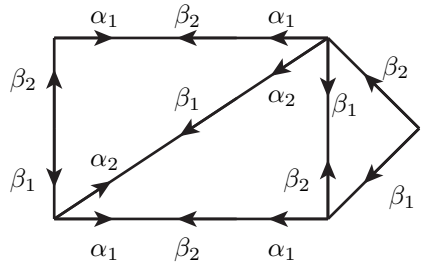}
    \caption{$\mc{Q}_1$-van Kampen diagram for $[\beta_2^{\alpha_1}, \beta_2^{-1} \beta_1]$} \label{fig9}
  \end{figure}

  \begin{figure}[h]
    \psfrag{a1}{$\alpha_1$}
    \psfrag{b1}{$\beta_1$}
    \psfrag{a2}{$\alpha_2$}
    \psfrag{b2}{$\beta_2$}
    \psfrag{t}{$t$}
    \centering \includegraphics{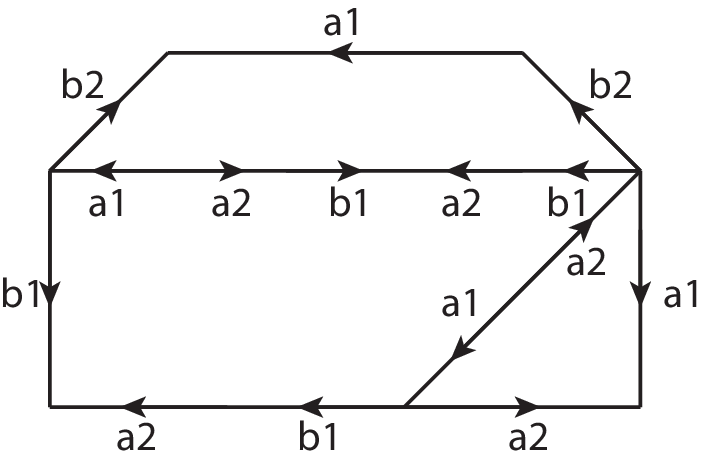}
    \caption{$\mc{Q}_2$-van Kampen diagram for $[\alpha_1^{-1}, \beta_2] [\alpha_2^{-1}, \beta_1]^{-1}$} \label{fig10}
  \end{figure}

  \begin{figure}[h]
    \psfrag{a1}{$\alpha_1$}
    \psfrag{b1}{$\beta_1$}
    \psfrag{a2}{$\alpha_2$}
    \psfrag{b2}{$\beta_2$}
    \psfrag{t}{$t$}
    \centering \includegraphics{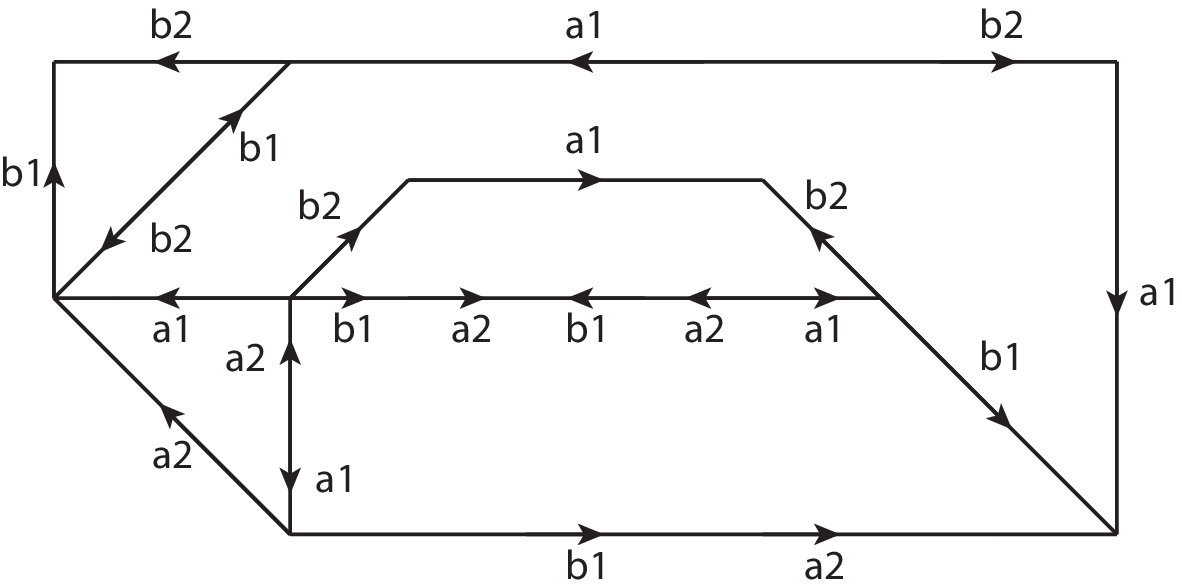}
    \caption{$\mc{Q}_2$-van Kampen diagram for $[\alpha_1^{-1}, \beta_2^{-1}] [\alpha_2^{-1}, \beta_1^{-1}]^{-1}$} \label{fig11}
  \end{figure}

\subsection{A lower bound on the Dehn function of $K^3_2(2)$}

\begin{thm}
  The Dehn function $\delta$ of $K^3_2(2)$ satisfies $\delta(l) \succeq l^3$.
\end{thm}

\begin{proof}
By Proposition~\ref{prop13} and Theorem~\ref{thm14}, we have that $K^3_2(2) \cong L_1 \ast_M L_2$ where, as subgroups of $F_2^{(1)} \times F_2^{(2)}$, $L_1 = K^2_2(1)$ is generated by $\mathcal{A}_1 = \{x_1x_2^{-1}, y_1, y_2\}$, $L_2 \cong K^2_2(1)$ is generated by $\mathcal{A}_2 = \{x_1, x_2, y_1y_2^{-1}\}$ and $M = K^2_2(2)$ is generated by $\mathcal{B} = \{x_1x_2^{-1}, y_1y_2^{-1}, [x_1, y_1] \}$.  To obtain the generating set for $L_2$ we have here implicitly used the automorphism of $F_2^{(1)} \times F_2^{(2)}$ which interchanges $x_i$ with $y_i$ and realises the isomorphism between $L_2$ and $K_2^2(1)$.

For each $l \in \N$, define $h_l$ to be the element $[x_1^l, y_1^l] \in K^2_2(2)$ and define $w_l$ to be the word $[(x_1x_2^{-1})^l, y_1^l] \in \mc{A}_1^{\pm\bast}$ representing $h_l$.  Note that $h_l$ commutes with both $y_2 \in \mathcal{A}_1$ and $x_2 \in \mathcal{A}_2$ so, by Theorem~\ref{thm13}, the word $[w_l, (y_2x_2)^l]$, which has length $16l$, has area at least $2l \, d_\mathcal{B}(1, h_l)$.  We claim that $d_\mathcal{B}(1, h_l) \geq l^2$.

Suppose that in $F_2^{(1)} \times F_2^{(2)}$ the element $h_l$ is represented by a word $w \equiv w(x_1x_2^{-1}, y_1y_2^{-1}, [x_1, y_1])$ in the generators $\mathcal{B}$.  Let $k$ be the number of occurrences of the third variable in the word $w$.  We will show that $k \geq l^2$.

Observe that as group elements the word $w(x_1x_2^{-1}, y_1y_2^{-1}, [x_1, y_1])$ is equal to the word $w(x_1, y_1, [x_1, y_1]) \, w(x_2^{-1}, y_2^{-1}, 1).$ Thus we have that $[x_1^l, y_1^l]$ is freely equal to $w(x_1, y_1, [x_1, y_1])$ and that $w(x_2^{-1}, y_2^{-1}, 1)$, and thus $w(x_1, y_1, 1)$, is freely equal to the empty word. It follows that there exists a null $\mc{P}$-sequence for $[x_1^l, y_1^l]$ with area $k$, where $\mc{P}$ is the presentation $\langle x_1, y_1 \, | \, [x_1, y_1] \rangle$.  But $\mc{P}$ presents the rank $2$ free abelian group, and basic results on Dehn functions give that $[x_1^l, y_1^l]$ has area $l^2$ over this presentation. Thus $k \geq l^2$.
\end{proof}

\section{Bestvina-Brady groups} \label{sec3}

\begin{defn}
  A simplicial complex is said to be \emph{flag} if every collection of pairwise adjacent vertices spans a simplex.  A finite flag simplicial complex $\Delta$ with vertices $v_1, \ldots, v_k$ defines an associated \emph{right-angled Artin group} $A_\Delta$ given by the presentation $$\mc{P}_A = \langle v_1, \ldots, v_k \, | \, [v_i, v_j] \text{ whenever $v_i$ and $v_j$ are joined by an edge in $\Delta$} \rangle.$$ The \emph{Bestvina-Brady group} $H_\Delta$ associated to $\Delta$ is defined to be the kernel of the homomorphism $A \rightarrow \Z = \langle t \rangle$ which maps each $a_i \mapsto t$.
\end{defn}

\begin{defn}
  A simplicial complex $\Delta$ is said to be \emph{$n$-connected} (respectively \emph{$n$-acyclic}), where $n$ is a positive integer, if $\pi_i(\Delta)$ (resp. $H_i(\Delta, \Z)$) is trivial for all $i \leq n$.
\end{defn}

\begin{thm}[Bestvina-Brady \cite{Bestvina1}]
  \mbox{}
  \begin{enumerate}
    \item $H_\Delta$ is of type $\rm{F}_m$ if and only if $\Delta$ is $(m-1)$-connected.

    \item $H_\Delta$ is  of type $\rm{FP}_m$ if and only if $\Delta$ is $(m-1)$-acyclic.
  \end{enumerate}
\end{thm}

This section is devoted to proving the following result.

\begin{thm} \label{thm3}
  Every finitely presented Bestvina-Brady group has $l^4$ as an isoperimetric function.
\end{thm}

Theorem~\ref{thm3} provides an obstruction to the method suggested in \cite{Brady1} for producing finitely presented Bestvina-Brady groups whose Dehn functions are $\simeq$-equivalent to $l^m$ for arbitrary integers $m$.

If a Bestvina-Brady group is finitely presented, then Dicks and Leary \cite{Dicks1} have shown how to read off from the defining complex a particularly pleasant finite presentation.  Let $\Edge(\Delta)$ be the set of directed edges of $\Delta$ (so the cardinality of $\Edge(\Delta)$ is twice the number of $1$-simplices in $\Delta$).  We write $\iota e$ and $\tau e$ respectively for the initial and terminal vertices of $e$ and $\overline{e}$ for the edge $e$ with the opposite orientation.  We say that the directed edges $e_1, \ldots, e_n$ form a \emph{combinatorial path} in $\Delta$, written $e_1 \cdot \ldots \cdot e_n$, if $\tau e_i = \iota e_{i+1}$ for all $i$. If furthermore $\tau e_n = \iota e_1$ then we say that $e_1 \cdot \ldots \cdot e_n$ is a \emph{combinatorial $1$-cycle}.

Define $\mc{R}_\Delta \subseteq \Edge(\Delta)^{\pm\bast}$ to consist of all words $e \overline{e}$ for $e \in \Edge(\Delta)$ and all words $efg$ and $e^{-1} f^{-1} g^{-1}$ where $e \cdot f \cdot g$ is a combinatorial $1$-cycle in $\Delta$.

\begin{thm}[Dicks-Leary \cite{Dicks1}]
  If $\Delta$ is simply connected then $H_\Delta$ is presented by $\langle \Edge(\Delta) \, | \, \mc{R}_\Delta \rangle$ with the embedding $H_\Delta \hookrightarrow A_\Delta$ given by $e \mapsto \iota e (\tau e)^{-1}$.
\end{thm}

The structure of the proof of Theorem~\ref{thm3} is as follows.  Let $H_\Delta$ and $A_\Delta$ be the Bestvina-Brady and right-angled Artin groups respectively associated to a simply-connected finite flag simplicial complex $\Delta$.  The cyclic extension $1 \rightarrow H_\Delta \rightarrow A_\Delta \rightarrow \Z \rightarrow 1$ splits and we take a positive normal form presentation $\langle \Edge(\Delta), t \, | \, \mc{R}_\Delta, \mc{S}_\Delta \rangle$ for $A_\Delta$, where $\mc{P}_H = \langle \Edge(\Delta) \, | \, \mc{R}_\Delta \rangle$ is the Dicks-Leary presentation for $H_\Delta$ and $\mc{S}_\Delta$ consists of a relator of the form $tet^{-1} w_e^{-1}$ with $w_e \in \Edge(\Delta)^{\pm\bast}$ for each $e \in \Edge(\Delta)$.  Since $A_\Delta$ is CAT(0) it admits a quadratic-linear area-radius pair \cite[Proposition~III.$\Gamma$.1.6]{brid02}, and so we can apply Theorem~\ref{thm1} to produce an infinite indexed presentation $(\mc{P}_H^\infty, \| \cdot \|)$ for $H_\Delta$ that admits a quadratic-linear area-penetration pair.  Lemma~\ref{lem8} shows that the relational area function $\RArea_H$ of $(\mc{P}_H^\infty, \| \cdot \|)$ over $\mc{P}_H$ is $\preceq$ quadratic and hence Theorem~\ref{thm3} follows by Proposition~\ref{prop1}.  The individual calculations required to prove Lemma~\ref{lem8} are set out in Lemmas~\ref{lem1}--\ref{lem7}.

Choose a base vertex $q$ and a spanning tree $T$ in the $1$-skeleton of $\Delta$.  Given $n \in \Z$ and vertices $u$ and $v$ of $\Delta$ write $p_n(u, v)$ for the element $e_1^n \ldots e_l^n$ of $\Edge(\Delta)^{\pm\bast}$ where $e_1 \cdot \ldots \cdot e_l$ is the unique geodesic combinatorial path in $T$ from $u$ to $v$.  We write $p(u, v)$ as shorthand for $p_1(u, v)$.  Note that as group elements \begin{equation} \label{Eqn1} \begin{aligned} p_n(u, v)^{-1} &= (e_1^n \ldots e_l^n)^{-1} \\ &= e_l^{-n} \ldots e_1^{-n} \\ &= \ol{e_l}^n \ldots \ol{e_1}^n \\ &= p_n(v, u) \end{aligned} \end{equation} in $H_\Delta$.  For each $e \in \Edge(\Delta)$, define $w_e$ to be the word $p(q, \iota e) e p(\iota e, q) \in \Edge(\Delta)^{\pm\bast}$.  In \cite{Dicks1} it is proved that mapping $e \mapsto w_e$ defines an automorphism $\theta$ of $H_\Delta$  and that $H_\Delta \rtimes_\theta \Z$ is isomorphic to $A_\Delta$ with $e \in \Edge(\Delta)$ corresponding to $\iota e (\tau e)^{-1}$ and the generator $t$ of $\Z$ corresponding to $q \in A_\Delta$. It is also shown that if $e_1 \cdot \ldots \cdot e_l$ is a combinatorial $1$-cycle then $e_1^n \ldots e_l^n$ is null-homotopic in $H_\Delta$. Define $\mc{S}_\Delta$ to be the set of words $\{t e t^{-1} w_e \, : \, e \in \Edge(\Delta)\} \subseteq (\Edge(\Delta) \cup \{t\})^{\pm\bast}$.  Then $A_\Delta$ is finitely presented by $\mc{P}_A' = \langle \Edge(\Delta) , t \, | \, \mc{R}_\Delta, \mc{S}_\Delta \rangle$.

The following lemma details some properties of the automorphism $\theta$ of $H_\Delta$.  Of these we will only need (vii), but this property is most easily proved via the preceding sequence of assertions.

\renewcommand{\labelenumi}{(\roman{enumi})}
\begin{lem}
  For all $e \in \Edge(\Delta)$ and $n \in \Z$ the following equalities hold in $H_\Delta$:
  \begin{enumerate}
  \item $\theta(e) = p(q, \iota e) e p(q, \iota e)^{-1} = p(q, \iota e) e^2 p(\tau e, q) = p(q, \iota e) e^2 p(q, \tau e)^{-1}$.

  \item $\theta(e^n) = p(q, \iota e) e^n p(\iota e, q) = p(q, \iota e) e^{n+1} p(\tau e, q) = p(q, \iota e) e^{n+1} p(q, \tau e)^{-1}$.

  \item If $e_1 \cdot \ldots \cdot e_l$ is a combinatorial path then $$\theta(e_1^n \ldots e_l^n) = p(q, \iota e_1) e_1^{n+1} \ldots e_l^{n+1} p(\tau e_l, q).$$

  \item $\theta^{-1}(e) = p_{-1}(q, \iota e) p_{-1}(\tau e, q) = p_{-1}(q, \iota e) e p_{-1}(\iota e, q) = p_{-1}(q, \iota e) e p_{-1}(q, \iota e)^{-1}$.

  \item $\theta^{-1}(e^n) = p_{-1}(q, \iota e) e^n p_{-1}(\iota e, q) = p_{-1}(q, \iota e) e^{n-1} p_{-1}(\tau e, q) = p_{-1}(q, \iota e) e^{n-1} p_{-1}(q, \tau e)^{-1}$.

  \item If $e_1 \cdot \ldots \cdot e_l$ is a combinatorial path then $$\theta^{-1}(e_1^n \ldots e_l^n) = p_{-1}(q, \iota e_1) e_1^{n-1} \ldots e_l^{n-1} p_{-1}(\tau e_l, q).$$

  \item $\theta^k(e) = p_k(q, \iota e) e^{k+1} p_k(\tau e, q)$.
\end{enumerate}
\end{lem}

\begin{proof}
\mbox{}
   \begin{enumerate}
     \item The first and third equalities follow from equation \eqref{Eqn1}.  The second equality follows from the fact that $p(q, \iota e) e p(\tau e ,q)$ is null-homotopic.

     \item The first equality holds since $\theta(e^n) = \theta(e)^n = [p(q, \iota e) e p(q, \iota e)^{-1}]^n = p(q, \iota e) e^n p(q, \iota e)^{-1} = p(q, \iota e) e^n p(\iota e, q)$ in $H_\Delta$.  The second and third equalities then hold since $p(\iota e, q) = e p(\tau e ,q)$ in $H_\Delta$ and by equation \eqref{Eqn1} respectively.

     \item Follows from the fact that $\theta(e_i^n) = p(q, \iota e) e_i^{n+1} p(q, \tau e)^{-1}$ in $H_\Delta$.

     \item The first equality holds since $\theta(p_{-1}(q, \iota e) p_{-1}(\tau e, q)) =\\ p(q, q) p_0(q, \iota e) p(\iota e, q) p(q, \tau e) p_0(\tau e, q) p(q,q) = p(\iota e, q) p(q, \tau e) = e$ in $H_\Delta$.  The second and third equalities follows from the fact that $p_{-1}(q, \tau e) \bar{e}^{-1} p_{-1}(\iota e, q) = p_{-1}(q, \tau e) e p_{-1}(\iota e, q)$ is null-homotopic.

     \item Follows from (iv) as in the proof of (ii).

     \item Follows from (v) as in the proof of (iii).

     \item Follows from (iii) and (vi) by induction on $|k|$.
   \end{enumerate}
\end{proof}

For each $n \in \Z$, define a homomorphism $\Phi_n : \Edge(\Delta)^{\pm\bast} \rightarrow \Edge(\Delta)^{\pm\bast}$ which commutes with the inversion involution and is a lift of $\theta^n$ by mapping $e \mapsto p_n(q, \iota e) e^{n+1} p_n(\tau e, q)$. Define the collections of words \begin{align*} \overline{\mc{R}}_\Delta &= \{\Phi_n(r) \, : \, r \in \mc{R}_\Delta, n \in \Z\},\\
\overline{\mc{S}}_\Delta &= \{ \Phi_{n+1}(e) \Phi_n(w_e)^{-1} \, : \, e \in \Edge(\Delta), n \in \Z\}\end{align*} in $\Edge(\Delta)^{\pm\bast}$, and consider the presentation $\mc{P}_H^\infty = \langle \Edge(\Delta) \, | \, \overline{\mc{R}}_\Delta, \overline{\mc{S}}_\Delta \rangle$ of $H_\Delta$.  Define an index $\| \cdot \|$ on $\overline{\mc{R}}_\Delta \cup \overline{\mc{S}}_\Delta$ by setting $\|\omega \|$ to be the minimal value of $|k|$ such that either $\omega \equiv \Phi_k(r)$ for some $r \in \mc{R}_\Delta$ or $\omega \equiv \Phi_{k+1}(e) \Phi_k(w_e)^{-1}$ for some $e \in \Edge(\Delta)$.

Let $\dist$ be the length metric on the $1$-skeleton of $\Delta$ given by setting the length of each edge to $1$.  Define $$L = \max\{\dist(u, v) \, : \, u, v \in \Vert(\Delta) \}.$$

\begin{lem} \label{lem1}
  $\Area_{\mc{P}_H}\big(\Phi_n(e \ol{e})\big) \leq (2L+1)|n| + 1$ for all $e \in \Edge(\Delta)$.
\end{lem}

\begin{proof}
  The calculation (\ref{Eqn1}) shows that $p_n(q, v)^{-1}$ can be converted to $p_n(v, q)$ at a $\mc{P}_H$-cost of at most $L|n|$ for all $v \in \Vert(\Delta)$.  The following is a null $\mc{P}_H$-scheme for the word $\Phi_n(e \ol{e})$:

  \medskip

  \begin{center}
    \begin{tabular}{ccc}
      \hline
      \\[-10pt]
      $j$ & $\sigma_j$ & Area \\
      \\[-10pt]
      \hline
      \\[-10pt]
      $1$ & $p_n(q, \iota e) e^{n+1} p_n(\tau e, q) p_n(q, \tau e) \ol{e}^{n+1} p_n(\iota e, q)$ & $L|n|$ \\
      $2$ & $p_n(q, \iota e) e^{n+1} \ol{e}^{n+1} p_n(\iota e, q)$ & $|n|+1$ \\
      $3$ & $p_n(q, \iota e) p_n(\iota e, q)$ & $L|n|$ \\
      \\[-10pt]
      \hline
      \\[-10pt]
      & Total & $(2L + 1) |n| + 1$ \\
      \\[-10pt]
      \hline
    \end{tabular}
  \end{center}
\end{proof}

\begin{lem} \label{lem2}
  Let $e \cdot f \cdot g$ be a combinatorial $1$-cycle in $\Delta$. Then $\Area_{\mc{P}_H}(e^n f^n g^n) \leq 3|n|^2$.
\end{lem}

\begin{proof}
  Note that the relators $efg$ and $e^{-1} f^{-1} g^{-1}$ imply that $ef = g^{-1} = fe$, so $[e, f]$ is null-homotopic with $\mc{P}_H$-$\Area$ $2$.  The following is a null $\mc{P}_H$-scheme for
  the word $e^n f^n g^n$:

  \medskip

  \begin{center}
    \begin{tabular}{ccc}
      \hline
      \\[-10pt]
      $j$ & $\sigma_j$ & Area \\
      \\[-10pt]
      \hline
      \\[-10pt]
      $1$ & $e^n f^n g^n$ & $|n|$ \\
      $2$ & $e^n f^n (f^{-1} e^{-1})^n$ & $2|n|^2$ \\
      $3$ & $e^n f^n f^{-n} e^{-n}$ & $0$ \\
      \\[-10pt]
      \hline
      \\[-10pt]
      & Total & $2|n|^2 + |n|$ \\
      \\[-10pt]
      \hline
    \end{tabular}
  \end{center}
\end{proof}

\begin{lem} \label{lem3}
  Let $e \cdot f \cdot g$ be a combinatorial $1$-cycle in $\Delta$.  Then $\Area_{\mc{P}_H} \big( \Phi_n(efg) \big) \leq 3|n|^2 + (3L + 6)|n| + 3$.
\end{lem}

\begin{proof}
  The following is a null $\mc{P}_H$-scheme for the word $\Phi_n(efg)$:

  \medskip

  \begin{center}
    \begin{tabular}{ccc}
      \hline
      \\[-10pt]
      $j$ & $\sigma_j$ & Area \\
      \\[-10pt]
      \hline
      \\[-10pt]
      $1$ & $p_n(q, \iota e) e^{n+1} p_n(\tau e, q) p_n(q, \iota f) f^{n+1} p_n(\tau f, q) \ldots$ & \\
      & $\ldots p_n(q, \iota g) g^{n+1} p_n(\tau g, q)$ & $2L|n|$ \\
      $2$ & $p_n(q, \iota e) e^{n+1} f^{n+1} g^{n+1} p_n(\tau g, q)$ & $3|n+1|^2$ \\
      $3$ & $p_n(q, \iota e) p_n(\tau g, q)$ & $L|n|$ \\
      \\[-10pt]
      \hline
      \\[-10pt]
      & Total & $3|n|^2 + (3L + 6)|n| + 3$ \\
      \\[-10pt]
      \hline
    \end{tabular}
  \end{center}
\end{proof}

\begin{defn}
  Given a combinatorial $1$-cycle $C$ in $\Delta$, a sequence $(C_i)_{i=0}^m$ of combinatorial $1$-cycles is said to be \emph{combinatorial null-homotopy} for $C$ if $C_0 = C$, $C_m = \emptyset$ and each $C_{i+1}$ is obtained from $C_i$ by one of the following moves:
  \begin{itemize}
    \item \emph{$1$-cell expansion}: $C_i = e_1 \cdot \ldots \cdot e_l \rightsquigarrow C_{i+1} = e_1 \cdot \ldots \cdot e_k \cdot e \cdot \ol{e} \cdot e_{k+1} \cdot \ldots \cdot e_l$ for some $k$, where $e \in \Edge(\Delta)$;

    \item \emph{$1$-cell collapse}: Reverse of a $1$-cell expansion;

    \item \emph{$2$-cell expansion}: $C_i = e_1 \cdot \ldots \cdot e_l \rightsquigarrow C_{i+1} = e_1 \cdot \ldots \cdot e_k \cdot e \cdot f \cdot g \cdot e_{k+1} \cdot \ldots \cdot e_l$ for some $k$, where $e \cdot f \cdot g$ is a combinatorial $1$-cycle;

    \item \emph{$2$-cell collapse}: Reverse of a $2$-cell expansion.
  \end{itemize}
\end{defn}

\begin{lem} \label{lem4}
  If $(C_i)_{i=0}^m$ is a combinatorial null-homotopy for the $1$-cycle $e_1 \cdot \ldots \cdot e_l$ then the word $e_1^n \ldots e_l^n$ has $\mc{P}_H$-Area $\leq 3m|n|^2$.
\end{lem}

\begin{proof}
  Given a combinatorial $1$-cycle $C = e_1 \cdot \ldots \cdot e_l$, write $W_n(C)$ for the word $e_1^n \ldots e_l^n \in \Edge(\Delta)^{\pm\bast}$.  If the $1$-cycle $C_i$ is obtained from $C_{i-1}$ by a $1$-cell expansion or collapse then, by repeated application of a relator $e \ol{e}$, the word $W_n(C_{i-1})$ can be converted to the word $W_n(C_i)$ at a $\mc{P}_H$-cost of at most $|n|$.  If the $1$-cycle $C_i$ is obtained from $C_{i-1}$ by a $2$-cell expansion or collapse then, by lemma \ref{lem2}, the word $W_n(C_{i-1})$ can be converted to the word $W_n(C_i)$ at a $\mc{P}_H$-cost of at most $3|n|^2$.

  Define $m_1$ to be the number of $i$ for which $C_i$ is obtained from $C_{i-1}$ by a $1$-cell expansion or collapse.  Define $m_2$ to be the number of $i$ for which $C_i$ is obtained from $C_{i-1}$ by a $2$-cell expansion or collapse.  Then the $\mc{P}_H$-Area of $e_1^n \ldots e_l^n = W_n(C)$ is at most $m_1|n| + 3m_2|n|^2 \leq 3(m_1 + m_2) |n|^2 = 3m|n|^2$.
\end{proof}

\begin{lem} \label{lem5}
  There exists a constant $K$ such that $\Area_{\mc{P}_H}\big( p_n(q, \iota e) e^n p_n(\tau e, q) \big) \leq K|n|^2$ for all $e \in \Edge(\Delta)$.
\end{lem}

\begin{proof}
  Given $e \in \Edge(\Delta)$ write $\gamma_\iota(e)$ and $\gamma_\tau(e)$ respectively for the unique combinatorial geodesic paths in $T$ from $q$ to $\iota e$ and from $\tau e$ to $q$.  Then $\gamma_\iota(e) \cdot e \cdot \gamma_\tau(e)$ is a combinatorial $1$-cycle for which there exists a combinatorial null-homotopy $\big(C_i(e)\big)_{i=0}^{m(e)}$ since $\Delta$ is simply-connected. By Lemma~\ref{lem4}, $\Area_{\mc{P}_H} \big( p_n(q, \iota e) e^n p_n(\tau e, q) \big) \leq 3m(e)|n|^2$, so we can take $K = 3\max\{ m(e) \, : \, e \in \Edge(\Delta)\}$.
\end{proof}

\begin{lem} \label{lem6}
  Let $e \cdot f \cdot g$ be a combinatorial $1$-cycle in $\Delta$.  Then $\Area_{\mc{P}_H} \big( \Phi_n(e^{-1}f^{-1}g^{-1}) \big) \leq (3K + 4) |n|^2 + (6L + 6) |n| + 5$, where $K$ is the constant from Lemma~\ref{lem5}.
\end{lem}

\begin{proof}
  The following is a null $\mc{P}_H$-scheme for the word $\Phi_n(e^{-1} f^{-1} g^{-1})$:

  \medskip

  \begin{center}
    \begin{tabular}{ccc}
      \hline
      \\[-10pt]
      $j$ & $\sigma_j$ & Area \\
      \\[-10pt]
      \hline
      \\[-10pt]
      $1$ & $p_n(\tau e, q)^{-1} e^{-n-1} p_n(q, \iota e)^{-1} p_n(\tau f, q)^{-1} f^{-n-1} \ldots$ & \\
        & $\ldots p_n(q, \iota f)^{-1} p_n(\tau g, q)^{-1} g^{-n-1} p_n(q, \iota g)^{-1}$ & $6L|n|$ \\
      $2$ & $p_n(q, \tau e) e^{-n-1} p_n(\iota e, q) p_n(q, \tau f) f^{-n-1} p_n(\iota f, q) \ldots$ & \\
        & $\ldots p_n(q, \tau g) g^{-n-1} p_n(\iota g, q)$ & $0$ \\
      $3$ & $p_n(q, \iota f) e^{-n-1} p_n(\tau g, q) p_n(q, \iota g) f^{-n-1} p_n(\tau e, q) \ldots$ & \\
        & $\ldots p_n(q, \iota e) g^{-n-1} p_n(\tau f, q) p_n(q , \iota f) p_n(q, \iota f)^{-1} $ & $3K|n|^2$ \\
      $4$ & $p_n(q, \iota f) e^{-n-1} g^{-n} f^{-n-1} e^{-n} g^{-n-1} f^{-n} p_n(q, \iota f)^{-1}$ & $2|n| + 1$ \\
      $5$ & $p_n(q, \iota f) e^{-n-1} (ef)^n f^{-n-1} e^{-n} (ef)^{n+1} f^{-n} p_n(q, \iota f)^{-1}$ & $2|n|^2 + 2|n+1|^2$ \\
      $6$ & $p_n(q, \iota f) e^{-n-1} e^n f^n f^{-n-1} e^{-n} e^{n+1} f^{n+1} f^{-n} p_n(q, \iota f)^{-1}$ & $0$ \\
      $7$ & $p_n(q, \iota f) e^{-1} f^{-1} e f p_n(q, \iota f)^{-1}$ & $2$ \\
      $8$ & $p_n(q, \iota f) g g^{-1} p_n(q, \iota f)^{-1}$ & $0$ \\
      \\[-10pt]
      \hline
      \\[-10pt]
      & Total & $\begin{gathered} (3K + 4) |n|^2  \\[-2pt] + (6L + 6) |n| + 5 \end{gathered}$ \\
      \\[-10pt]
      \hline
    \end{tabular}
  \end{center}
\end{proof}

\begin{lem} \label{lem7}
  $\Area_{\mc{P}_H} \big( \Phi_{n+1}(e) \Phi_n(w_e)^{-1} \big) \leq 2K|n|^2 + (3L^2 + 2L + 2K)|n| + L + K$ for all $e \in \Edge(\Delta)$, where $K$ is the constant from Lemma~\ref{lem5}.
\end{lem}

\begin{proof}
  Note that if $e_1 \cdot \ldots \cdot e_l$ is a combinatorial edge-path in $\Delta$ then $\Phi_n(e_1 \ldots e_l) = \prod_{i=1}^l p_n(q, \iota e_i) e_i^{n+1} p_n(\tau e_i, q)$ can be converted to $\prod_{i=1}^l p_n(q, \iota e_i) e_i^{n+1} p_n(q, \tau e_i)^{-1} \stackrel{\rm free}{=} p_n(q, \iota e_1) e_1^{n+1} \ldots e_l^{n+1} p_n(q, \tau e_l)^{-1}$ at a $\mc{P}_H$-cost of at most $lL|n|$.  It follows that for all $u, v \in \Vert(\Delta)$ the word $\Phi_n \big( p(u, v) \big)$ can be converted to the word $p_n(q, u) p_{n+1}(u, v) p_n(q, v)^{-1}$ at a $\mc{P}_H$-cost of at most $L^2 |n|$.

  The following is a null $\mc{P}_H$-scheme for the word $\Phi_{n+1}(e) \Phi_n(w_e)^{-1}$:

  \medskip

  \begin{center}
    \begin{tabular}{ccc}
      \hline
      \\[-10pt]
      $j$ & $\sigma_j$ & Area \\
      \\[-10pt]
      \hline
      \\[-10pt]
      $1$ & $p_{n+1}(q, \iota e) e^{n+2} p_{n+1}(\tau e, q) \left[\Phi_n \big( p(q, \iota e) e p(\iota e, q) \big) \right]^{-1}$ & $2L^2|n|$ \\
      $2$ & $p_{n+1}(q, \iota e) e^{n+2} p_{n+1}(\tau e, q) \big[ p_{n+1}(q, \iota e) p_n(q, \iota e)^{-1} \ldots$ & \\
      & $\ldots p_n(q, \iota e) e^{n+1} p_n(\tau e, q) p_n(q, \iota e) p_{n+1}(\iota e, q) \big]^{-1}$ & $0$ \\
      $3$ & $p_{n+1}(q, \iota e) e^{n+2} p_{n+1}(\tau e, q) p_{n+1}(\iota e, q)^{-1} \ldots$ & \\
      & $\ldots p_n(q, \iota e)^{-1} p_n(\tau e, q)^{-1} e^{-n-1} p_{n+1}(q, \iota e)^{-1}$ & $L|n+1|$ \\
      $4$ & $p_{n+1}(q, \iota e) e^{n+2} p_{n+1}(\tau e, q) p_{n+1}(q, \iota e) \ldots$ & \\
      & $\ldots p_n(q, \iota e)^{-1} p_n(\tau e, q)^{-1} e^{-n-1} p_{n+1}(q, \iota e)^{-1}$ & $K|n+1|^2 + K |n|^2$ \\
      $5$ & $p_{n+1}(q, \iota e) e^{n+2} e^{-n-1} e^n e^{-n-1} p_{n+1}(q, \iota e)^{-1}$ & $0$ \\
      \\[-10pt]
      \hline
      \\[-10pt]
      & Total & $\begin{gathered} 2K|n|^2 \\[-2pt] + (2L^2 + L + 2K)|n| \\[-2pt] + L + K \end{gathered}$ \\
      \\[-10pt]
      \hline
    \end{tabular}
  \end{center}
\end{proof}

Combining Lemmas~\ref{lem1}, \ref{lem3}, \ref{lem6} and \ref{lem7} gives the following result.

\begin{lem} \label{lem8}
  The relational area function $\RArea_H$ of $(\mc{P}_H^\infty, \| \cdot \|)$ over $\mc{P}_H$ satisfies $\RArea_H(l) \preceq l^2$.
\end{lem}

\begin{proof}[Proof of Theorem~\ref{thm3}]
  Since right-angled Artin groups are CAT(0) \cite{Charney1}, $A_\Delta$ has some finite presentation which admits an area-radius pair $(\alpha, \rho)$ with $\alpha(l) \simeq l^2$ and $\rho(l) \simeq l$ \cite[Proposition~III.$\Gamma$.1.6.]{brid99}.  Thus, by Proposition~\ref{prop11}, $\mc{P}_A'$ admits an area-radius
  pair $(\alpha', \rho')$ with $\alpha'(l) \simeq l^2$ and $\rho'(l) \simeq l$.  By Theorem~\ref{thm1}, $(\alpha', \rho')$ is an are-penetration pair for $(\mc{P}_H^\infty, \| \cdot \|)$ and hence, by Proposition~\ref{prop1} and Lemma~\ref{lem8}, the Dehn function $\delta$ of $\mc{P}_H$ satisfies $\delta(l) \preceq l^4$.
\end{proof}

\bibliographystyle{abbrv}
\bibliography{V:/Will}

\end{document}